\documentclass{amsart}
\usepackage{amsmath,amssymb,amsfonts,mathrsfs, amsthm, mathtools}
\usepackage[]{units}
\usepackage{hyperref}
\usepackage{parskip}
\usepackage[utf8x]{inputenc}
\usepackage[english]{babel}
\usepackage{lipsum}  
\usepackage{color}
\usepackage{bbm}
\usepackage{stmaryrd}
\usepackage[foot]{amsaddr}
\usepackage[shortlabels]{enumitem}
\usepackage{tikz}
\usepackage{subcaption}
\usepackage{adjustbox}

\newtheorem{theorem}{Theorem}[section]
\newtheorem{lemma}[theorem]{Lemma}
\newtheorem{corollary}[theorem]{Corollary}
\theoremstyle{definition}
\newtheorem{definition}[theorem]{Definition}

\newtheorem{remark}{Remark}
\newtheorem{example}{Example}

\DeclareMathOperator{\Int}{int}
\DeclareMathOperator{\Ext}{ext}

\begin{document}

	\title{On the Scalability of the Schwarz Method}
	
	\author{Gabriele Ciaramella}
	\address{Fachbereich Mathematik und Statistik, Universit\"at Konstanz, Germany}
	\email{gabriele.ciaramella@uni-konstanz.de}
	
	\author{Muhammad Hassan}
	\address{Center for Computational Engineering Science, Department of Mathematics, RWTH Aachen University, Germany}
	\email{hassan@mathcces.rwth-aachen.de}
	
	\author{Benjamin Stamm}
	\address{Center for Computational Engineering Science, Department of Mathematics, RWTH Aachen University, Germany}
	\email{stamm@mathcces.rwth-aachen.de}
	
	
	%% Mathematical classification (2010)

	%	
	%	% REQUIRED
	\begin{abstract}
		In this article, we analyse the convergence behaviour and scalability properties of the one-level Parallel Schwarz method (PSM) for domain decomposition problems in which the boundaries of many subdomains lie in the interior of the global domain. Such problems arise, for instance, in solvation models in computational chemistry. Existing results on the scalability of the one-level PSM are limited to situations where each subdomain has access to the external boundary, and at most only two subdomains have a common overlap. We develop a systematic framework that allows us to bound the norm of the Schwarz iteration operator for domain decomposition problems in which subdomains may be completely embedded in the interior of the global domain and an arbitrary number of subdomains may have a common overlap. 
	\end{abstract}
	\maketitle
	
	%	% REQUIRED
	%	\begin{keywords}
	%		Domain decomposition methods; Schwarz methods; chain of atoms; 
	%		Laplace equation; ddCOSMO; Scalability analysis.
	%	\end{keywords}
	%	
	%	% REQUIRED
	%	\begin{AMS}
	%		65N55, 65F10, 65N22, 35J05, 35J57.
	%	\end{AMS}
	
	% Example of section
	\section{Introduction}
	
	Recent results in computational chemistry have raised intriguing questions about the scalability
	of classical one-level domain decomposition methods. In \cite{Stamm3} the authors
	combined the COSMO solvation model (see, for instance \cite{Barone,Klamt,Truong}) with the classical one-level
	parallel Schwarz method (PSM) \cite{Lions:1988} in a sub-structured integral form. 
	This allowed the authors to define a very efficient numerical framework capable of solving solvation problems and obtaining a fast and accurate computation of the electrostatic contribution to
	the solvation energy. This new framework, called ddCOSMO, exploits the physical setting of the COSMO model to define a domain decomposition. The computational domain consists of the union of van der Waals cavities associated with the atoms in the molecule, and therefore these spherical cavities can be chosen as the subdomains. The ddCOSMO implementation received a good deal of attention in the computational chemistry community (see, e.g., \cite{Stamm1,Stamm2}), and a major reason for its success was due to the scalability property exhibited by the classical Schwarz method. 
	
	An algorithm is said to be weakly scalable if it can solve a larger and larger problem with
	more and more processors in a fixed amount of time. According to classical Schwarz theory,
	the PSM is not scalable in general; see, e.g., \cite{ToselliWidlund,CiaramellaGander4}.
	However, in contrast to this theory, the authors in \cite{Stamm3} provide numerical evidence indicating that in some cases the one-level PSM converges to a given tolerance in a fixed number
	of iterations of a linear solver independently of the number $N$ of subdomains. This behaviour is observed
	if fixed-sized subdomains form a `chain-like' domain and their number increases. This numerical result was subsequently rigorously proved in \cite{CiaramellaGander,CiaramellaGander2,CiaramellaGander3}
	for the PSM and in \cite{CiaramellaGander4} for other one-level methods;
	see also \cite{LINDGREN2018712,HassanBen,CiaramellaRefl} for similar scalability results.
	
	On the other hand, it was observed by the authors that the weak scalability of the PSM is lost if the fixed-sized subdomains form a `globular-type'
	domain $\Omega$ such that the boundaries of many subdomains lie in the interior of $\Omega$. 
	The following question therefore arises: is it possible to quantify the lack of scalability of the PSM?
	To do so, for increasing $N$ one would need to estimate the number of iterations necessary to achieve
	a given tolerance. A typical example is to consider, for instance, the PSM for the solution of a one-dimensional Laplace problem.
	In \cite{CiaramellaGander4} a heuristic argument is used to explain why for an unfortunate initialization,
	a contraction in the infinity norm is observed only after a number of iterations proportional to the number of subdomains $N$. A first rigorous proof of this non-scalability result was given in \cite{CiaramellaHassanStamm}. Interestingly, the analysis turns out to be quite complicated even in a simple one-dimensional setting.
	
	Naturally, the problem becomes extremely challenging if one considers a three-dimensional geometry involving subdomains consisting of spherical van der Waals cavities where the simultaneous intersection
	of more than two subdomains is possible. To the best of our knowledge, a rigorous extension of the
	results given in \cite{CiaramellaHassanStamm} to higher dimensional problems is missing in the literature.
	The novelty of our work is to fill this gap and develop a general framework that allows one to study the convergence
	and scalability of the PSM for `globular-type' domains. Our framework is based on an analysis of the infinite-dimensional Schwarz iteration operator, and as such can be viewed as a generalization
	of the analysis technique used in \cite{CiaramellaGander4, CiaramellaGander} where finite-dimensional Schwarz
	iteration matrices are studied instead. Furthermore, whereas the scalability analysis in \cite{CiaramellaGander,CiaramellaGander2,CiaramellaGander3} focuses on chains of fixed sized subdomains where only simple intersections are allowed, i.e., the intersection of any three distinct subdomains is assumed to be empty, our analysis is capable of covering the case of $M$-tuple intersections, i.e., we allow an arbitrary number of subdomains to have a non-empty intersection.
	
	It is important to remark that similar and very elegant results can be found in \cite{GanderZhao} and \cite{Mathew}.
	In \cite{GanderZhao}, the authors prove that the Schwarz waveform-relaxation method applied to the heat equation contracts at
	most every $m+2$ iterations, $m$ being an integer representing the maximum distance of the subdomains from the boundary.
	However, the overlapping subdomains considered in this work are constructed by `artificially' enlarging the given non-overlapping subdomains. This construction allows the authors to avoid the case that P. L. Lions refers to as ``weakly overlapping'' in \cite{Lions2} and present a convergence analysis that is based on the construction of a sequence of elliptic upper bounds, and relies on the structure of the artificially constructed overlapping domain decomposition. On the other hand, the domain decomposition considered in our paper is defined by the physics of the underlying problem and it is clearly weakly overlapping. Consequently, the proof given in \cite{GanderZhao} fails for the 
	geometric settings considered in our work. We wish also to remark that our convergence analysis is based on a direct study of the PSM iteration operator which allows us to identify the `worst initialization' and carefully track the propagation of the contraction across different subdomains comprising $\Omega$ in the course of the iterations. This is in contrast to a study of the Schwarz sequence as done in \cite{GanderZhao}. Furthermore, in order to undertake our analysis we introduce the notion of the so-called skeleton of a subdomain which allows us to carefully trace the effects of the Schwarz operator on functions defined inside the subdomains rather than on the boundaries, as done in \cite{GanderZhao}.
	
	Another interesting convergence analysis of classical Schwarz methods for several intersecting subdomains is presented in \cite{Mathew}.
	This work deals with stationary advection-reaction-diffusion problems and allows one to obtain powerful scalability results. In fact, the author's analysis which is based on the maximum principle and so-called
	barrier functions even yields estimates of the contraction factor.
	Unfortunately, this beautiful analysis fails in the case of the Laplace equation due to the lack of advection or reaction terms, and is also not valid for domain decomposition problems in which the boundaries of many subdomains lie in the interior of the global domain. In both these cases the bounds obtained via the barrier functions become ineffective and invalidate the analysis.
	
	The remainder of this paper is organized as follows: The main ideas and results of this work are stated in Section \ref{sec:2}. We first introduce the notation and all the mathematical objects needed for a detailed description of the domain geometry in Subsection \ref{sec:Geom}. Next, the PSM is formulated in Subsection \ref{sec:2.2}, and the convergence analysis is presented in Subsection \ref{sec:conv}. Our main results are stated in Subsection \ref{sec:ConvergenceResults}, and subsequently discussed with some examples in Subsection \ref{sec:Examples}. We remark that in Section \ref{sec:2} we restrict ourselves to a two-dimensional framework where only triple intersections are allowed. This choice is made in order to ease the notation, which is otherwise very technical and complicated, and to put more focus on the techniques used to prove our results. We show in Section \ref{sec:Extensions} how to extend our analysis to the case of arbitrary types and numbers of intersections and discuss the extension to three dimensions. Section \ref{sec:num} contains numerical experiments that support our main results. In Section \ref{sec:Protein}, we consider concrete biological molecules that have been studied in solvation models in computational chemistry (see, e.g., \cite{BenPaper}) and understand the consequences of our analysis as it pertains to these complex bio-molecules. Finally, we present our conclusions in Section \ref{sec:conclusions}.
	
	\section{The Schwarz method in two dimensions}\label{sec:2}
	
	\subsection{Geometric setting}\label{sec:Geom}
	
	We consider the Laplace equation in two dimensions. Let $\Omega \subset \mathbb{R}^2$ be an open, connected and bounded set, let $g \in C^0(\partial \Omega)$ be a given function. We must find $u \in L^{\infty}(\Omega)$ with the property that
	\begin{equation}\label{eq:1}
	\begin{split}
	\Delta u &=0 \qquad \text{in } \Omega,\\
	u&=g \qquad  \text{on } \partial\Omega.
	\end{split}
	\end{equation}
	Throughout this article, we assume that the domain $\Omega$ has the following structure: Let $\Omega_i, ~i=1,\ldots, N$ be a collection of $N$ intersecting disks of radii $r_i > 0, ~i=1,\ldots, N$. Then $\Omega = \cup_{i=1}^N \Omega_i$. An example of such a geometry is shown in Figure \ref{fig:1}. In order to avoid discussing an empty theory, we assume that $N>1$ and there exist at least two indices $i, j \in \{1, \ldots, N\}$ such that $\Omega_i \cap \Omega_j \neq \emptyset$.
	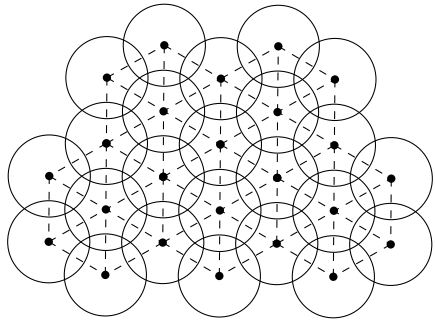
\begin{figure}
		\centering
		\begin{tikzpicture}[scale=0.35]
		\draw(26.02,12.5) circle (1.56cm);
		\draw(26,10) circle (1.56cm);
		\draw(23.82,8.77) circle (1.56cm);
		\draw(21.67,10.03) circle (1.56cm);
		\draw(21.69,12.53) circle (1.56cm);
		\draw(23.86,13.77) circle (1.56cm);
		\draw[dashed] (21.69,12.53)-- (26,10);
		\draw[dashed] (21.67,10.03)-- (26.02,12.5);
		\draw(23.84,11.26) circle (1.56cm);
		\draw(19.49,8.8) circle (1.56cm);
		\draw(17.34,10.07) circle (1.56cm);
		\draw(17.36,12.57) circle (1.56cm);
		\draw(19.53,13.8) circle (1.56cm);
		\draw[dashed] (17.36,12.57)-- (21.67,10.03);
		\draw[dashed] (17.34,10.07)-- (21.69,12.53);
		\draw(19.51,11.29) circle (1.56cm);
		\draw(23.88,16.27) circle (1.56cm);
		\draw(19.55,16.3) circle (1.56cm);
		\draw(21.73,17.53) circle (1.56cm);
		\draw[dashed] (19.55,16.3)-- (23.86,13.77);
		\draw[dashed] (19.53,13.8)-- (23.88,16.27);
		\draw(21.7,15.03) circle (1.56cm);
		\draw(15.2,13.84) circle (1.56cm);
		\draw(15.22,16.34) circle (1.56cm);
		\draw(17.4,17.57) circle (1.56cm);
		\draw[dashed] (15.22,16.34)-- (19.53,13.8);
		\draw[dashed] (15.2,13.84)-- (19.55,16.3);
		\draw(17.37,15.06) circle (1.56cm);
		\draw(15.16,8.84) circle (1.56cm);
		\draw(13.01,10.1) circle (1.56cm);
		\draw(13.03,12.6) circle (1.56cm);
		\draw[dashed] (13.03,12.6)-- (17.34,10.07);
		\draw[dashed] (13.01,10.1)-- (17.36,12.57);
		\draw(15.18,11.33) circle (1.56cm);
		\draw[dashed] (15.22,16.34)-- (17.4,17.57);
		\draw[dashed] (17.4,17.57)-- (19.55,16.3);
		\draw[dashed] (17.4,17.57)-- (17.37,15.06);
		\draw[dashed] (19.55,16.3)-- (21.73,17.53);
		\draw[dashed] (21.73,17.53)-- (23.88,16.27);
		\draw[dashed] (21.7,15.03)-- (21.73,17.53);
		\draw[dashed] (15.22,16.34)-- (15.2,13.84);
		\draw[dashed] (15.2,13.84)-- (17.36,12.57);
		\draw[dashed] (17.36,12.57)-- (17.37,15.06);
		\draw[dashed] (17.36,12.57)-- (19.53,13.8);
		\draw[dashed] (19.53,13.8)-- (19.51,11.29);
		\draw[dashed] (19.53,13.8)-- (21.69,12.53);
		\draw[dashed] (21.69,12.53)-- (21.7,15.03);
		\draw[dashed] (23.88,16.27)-- (23.86,13.77);
		\draw[dashed] (23.86,13.77)-- (21.69,12.53);
		\draw[dashed] (23.86,13.77)-- (23.84,11.26);
		\draw[dashed] (23.86,13.77)-- (26.02,12.5);
		\draw[dashed] (26.02,12.5)-- (26,10);
		\draw[dashed] (26,10)-- (23.82,8.77);
		\draw[dashed] (23.82,8.77)-- (23.84,11.26);
		\draw[dashed] (21.67,10.03)-- (23.82,8.77);
		\draw[dashed] (21.67,10.03)-- (19.49,8.8);
		\draw[dashed] (15.2,13.84)-- (15.18,11.33);
		\draw[dashed] (15.2,13.84)-- (13.03,12.6);
		\draw[dashed] (13.03,12.6)-- (13.01,10.1);
		\draw[dashed] (13.01,10.1)-- (15.16,8.84);
		\draw[dashed] (15.16,8.84)-- (15.18,11.33);
		\draw[dashed] (17.34,10.07)-- (15.16,8.84);
		\draw[dashed] (17.34,10.07)-- (19.49,8.8);
		\draw[dashed] (19.49,8.8)-- (19.51,11.29);
		\draw[dashed] (21.67,10.03)-- (21.69,12.53);
		\draw[dashed] (19.55,16.3)-- (19.53,13.8);
		\draw[dashed] (17.36,12.57)-- (17.34,10.07);
		\begin{scriptsize}
		\fill [color=black] (26,10) circle (4.5pt);
		\fill [color=black] (26.02,12.5) circle (4.5pt);
		\fill [color=black] (23.86,13.77) circle (4.5pt);
		\fill [color=black] (21.69,12.53) circle (4.5pt);
		\fill [color=black] (21.67,10.03) circle (4.5pt);
		\fill [color=black] (23.82,8.77) circle (4.5pt);
		\fill [color=black] (19.53,13.8) circle (4.5pt);
		\fill [color=black] (17.36,12.57) circle (4.5pt);
		\fill [color=black] (17.34,10.07) circle (4.5pt);
		\fill [color=black] (19.49,8.8) circle (4.5pt);
		\fill [color=black] (19.51,11.29) circle (4.5pt);
		\fill [color=black] (23.88,16.27) circle (4.5pt);
		\fill [color=black] (21.73,17.53) circle (4.5pt);
		\fill [color=black] (19.55,16.3) circle (4.5pt);
		\fill [color=black] (19.53,13.8) circle (4.5pt);
		\fill [color=black] (21.69,12.53) circle (4.5pt);
		\fill [color=black] (21.7,15.03) circle (4.5pt);
		\fill [color=black] (17.4,17.57) circle (4.5pt);
		\fill [color=black] (15.22,16.34) circle (4.5pt);
		\fill [color=black] (15.2,13.84) circle (4.5pt);
		\fill [color=black] (17.36,12.57) circle (4.5pt);
		\fill [color=black] (17.37,15.06) circle (4.5pt);
		\fill [color=black] (15.2,13.84) circle (4.5pt);
		\fill [color=black] (13.03,12.6) circle (4.5pt);
		\fill [color=black] (13.01,10.1) circle (4.5pt);
		\fill [color=black] (15.16,8.84) circle (4.5pt);
		\fill [color=black] (15.18,11.33) circle (4.5pt);
		\fill [color=black] (23.84,11.28) circle (4.5pt);
		\end{scriptsize}
		\end{tikzpicture}
		\caption{An example of a collection of disks $\{\Omega_i\}_{i=1}^N \subset \mathbb{R}^2$ on a triangular lattice composed of dots and dashed lines.}
		\label{fig:1}
	\end{figure}
	
	We impose the following constraints on the subdomains $\Omega_i, ~ i=1, \ldots, N$: 
	
	\begin{enumerate}
		%\item[A1)] No two disks intersect at a single point.
		\item[A1)] If two subdomains $\Omega_i$ and $\Omega_j$ have non-empty intersection, then there exist subsets of the boundaries $\Gamma_i \subset \partial \Omega_i \cap \Omega_j$ and $\Gamma_j \subset \partial \Omega_j \cap \Omega_i$ both of positive measure such that
		\begin{align*}
		\Gamma_i \cap \Omega_k &= \emptyset ~ \forall k = 1, \ldots, N \text{ with } k\neq i, j,\\
		\Gamma_j \cap \Omega_k &= \emptyset ~ \forall k = 1, \ldots, N \text{ with } k\neq i, j.
		\end{align*}
		In other words, we assume that if two subdomains $\Omega_i$ and $\Omega_j$ intersect, then they must have at least a simple intersection.
		\item[A2)] For any choice of distinct indices $i_1, i_2, i_3, i_4 \in \{1, \ldots, N\}$ it holds that $\Omega_{i_1} \cap \Omega_{i_2} \cap \Omega_{i_3} \cap \Omega_{i_4} = \emptyset$, i.e., at most only three {distinct subdomains} in the collection $\{\Omega_i\}_{i=1}^N$ have non-empty intersection. In other words, we assume that our geometry consists of at most triple intersections.
	\end{enumerate} 
	
	Constraint A1) is imposed purely for notational convenience, and it is readily seen that the subsequent analysis does not require such an assumption. Constraint A2) is much stronger but needs to be imposed-at least initially- in order to keep the focus on the analysis rather than the notational complexities that would otherwise be introduced. In Section \ref{sec:Extensions}, we discuss how to weaken this constraint, and show that all of our results can be extended to the case of arbitrary types of intersections.

	We now develop the necessary notation.
	
	\subsubsection{Partition of the boundary}\label{sec:2.1.1}
	
	Let $j \in \{1, \ldots, N\}$. Given the disk $\Omega_j$ we define the sets
	\begin{equation*}
	\Gamma_j^\text{ext}:= \partial \Omega_j \cap \partial \Omega,
	\quad
	\Gamma_j^{\Int} := \overline{\partial \Omega_j \setminus \Gamma_j^\text{ext}}.
	\end{equation*}
	{These sets} represent a decomposition of the boundary $\partial \Omega_j$ into an `external' part 
	{$\Gamma_j^\text{ext}$} that is common with the boundary $\partial \Omega$ of the global domain and an `internal' part {$\Gamma_j^{\Int}$} that is {contained in} the global domain $\Omega$. Notice that the sets $\Gamma_j^{\text{ext}}$ and $\Gamma_j^{\text{int}}$ are both closed. An example of a domain $\Omega=\cup_{j=1}^N \Omega_j$ and the decomposition of the subdomain boundaries $\partial \Omega_j$ is shown in Figure \ref{fig:2}. 
	
	\begin{figure}[h]
		\centering
		\begin{tikzpicture}[scale=0.85]
		
		\begin{scope}[scale=0.7,xshift=11.4cm,yshift=4.8cm]
		\draw(26.02,12.5) circle (1.56cm);
		\draw(26,10) circle (1.56cm);
		\draw(23.84,11.26) circle (1.56cm);
		\fill [black] (24.47,12.69) circle (2.4pt);
		\fill [black] (25.39,11.07) circle (2.4pt);
		\fill [black] (24.45,9.82) circle (2.4pt);
		\fill [black] (25.39,11.44) circle (2.4pt);
		\fill [black] (25.08,11.26) circle (2.4pt);
		\draw[black] (23.50,11.2) node {$\Omega_1$};
		\draw[black] (26.20,13.2) node {$\Omega_2$};
		\draw[black] (26.20,9.2) node {$\Omega_3$};
		\end{scope}
		
		\draw(31.84,11.26) circle (1.56cm);
		\fill [black] (32.47,12.69) circle (2.1pt);
		\fill [black] (33.39,11.07) circle (2.1pt);
		\fill [black] (32.45,9.82) circle (2.1pt);
		\fill [black] (33.39,11.44) circle (2.1pt);
		\fill [black] (33.08,11.26) circle (2.1pt);
		
		\draw [->] (34.2,12.25) -- (33.47,11.52);
		\draw [->] (34.2,10.47) -- (33.47,10.99);

		\draw [shift={(34,10)}] plot[domain=1.97:3.25,variable=\t]({1*1.56*cos(\t r)+0*1.56*sin(\t r)},{0*1.56*cos(\t r)+1*1.56*sin(\t r)});
		\draw [shift={(34.02,12.5)}] plot[domain=3.02:4.3,variable=\t]({1*1.56*cos(\t r)+0*1.56*sin(\t r)},{0*1.56*cos(\t r)+1*1.56*sin(\t r)});
		\draw [->](27.44,11.22)-- (29.08,11.22);
		\draw [<->,shift={(31.84,11.26)}] plot[domain=1.2:5.02,variable=\t]({1*1.78*cos(\t r)+0*1.78*sin(\t r)},{0*1.78*cos(\t r)+1*1.78*sin(\t r)});
		\draw [<->,shift={(31.84,11.26)}] plot[domain=-1.12:1.1,variable=\t]({1*1.78*cos(\t r)+0*1.78*sin(\t r)},{0*1.78*cos(\t r)+1*1.78*sin(\t r)});
		
		\draw[black] (34.05,11.2) node {$\Gamma_1^{\Int}$};
		\draw[black] (31.00,11.2) node {$\Omega_1$};
		\draw[black] (29.70,11.2) node {$\Gamma_1^\text{ext}$};
		\draw[black] (28.20,11.45) node {pick $\Omega_1$};
		\begin{scriptsize}
		\draw[black] (33.05,11.65) node {$\Gamma_2^{1,3}$};
		\draw[black] (32.20,11.8) node {$\Gamma_2^{1,0}$};
		\draw[black] (33.05,10.85) node {$\Gamma_3^{1,2}$};
		\draw[black] (32.20,10.35) node {$\Gamma_3^{1,0}$};
		
		\draw[black] (34.35,12.35) node {$A$};
		\draw[black] (34.35,10.47) node {$B$};
		\draw[black] (32.20,12.6) node {$P_1$};
		\draw[black] (32.20,9.95) node {$P_2$};
		
		\end{scriptsize}
		\end{tikzpicture}
		\caption{An example of three intersecting subdomains $\Omega_1$, $\Omega_2$ and $\Omega_3$ (left), external and internal boundaries
			of $\Omega_1$, and boundaries of $\Omega_2$ and $\Omega_3$ intersecting $\Omega_1$ (right).}
		\label{fig:2}
	\end{figure}
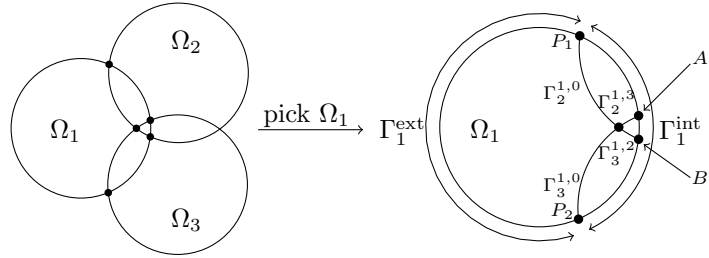

	Next, {a further decomposition of $\Gamma_j^{\Int}$ is considered}. 
	To this end, we first define the set ${N}_j$ as
	\begin{equation*}
	{N}_j := \{i \in \mathbb{N}, i\neq j \colon \Omega_i \cap \Omega_j \neq \emptyset\}.
	\end{equation*}
	Thus, ${N}_j$ is the set of indices of all subdomains $\{\Omega_i\}_{i=1, i\neq j}^N$ that intersect the subdomain $\Omega_j$. Informally, we refer to ${N}_j$ as the index set of neighbours of $\Omega_j$.
	Furthermore, given the set ${N}_j$ we define for each $k \in {N}_j$ the set $ N_{jk} \subset \mathbb{N}_0$ as
	\begin{equation*}
	N_{jk} := (N_j \cap N_k) \cup \{0\}.
	\end{equation*}
	
	Next, for each $j \in \{1, \ldots, N\}$, each $k \in {N}_j$ and each $i \in N_{jk}$ we define the set $\Gamma_{j}^{k, i} \subset \Gamma_j^{{\Int}}$ as
	\begin{equation*}
	\Gamma_{j}^{k, i}:= \begin{cases}
	%\text{int}\big(\partial \Omega_j \cap \Omega_k \cap \Omega_i\big) \quad & \text{ if } i \neq 0,\\
	%\text{int}\big(\partial \Omega_j \cap \Omega_k  \setminus \cup_{\substack{l \in N_j \\ l \neq k}}\Omega_l\big) \quad &\text{ if } i= 0,
	\text{int} \left\{x \in \partial \Omega_j \colon x \in \Omega_k \cap \Omega_{i}\right\} \quad & \text{ if } i \neq 0,\\
	\text{int} \left\{x \in \partial \Omega_j \colon (x\in \Omega_k) \land \big(x \notin \Omega_{\ell} ~\forall \ell \in N_{j} \text{ such that } \ell\neq k\big)\right\} \quad & \text{ if } i = 0.
	\end{cases}
	\end{equation*}
	where $\text{int}(\cdot)$ denotes the interior of a set. Intuitively, for a fixed $k \in {N}_j$ and $0\neq i \in N_{jk}$, the set $\Gamma_{j}^{k, i}$ denotes the portion {of $\partial \Omega_j$} that intersects both subdomain $k$ and subdomain $i$, while $\Gamma_{j}^{k, 0}$ denotes the portion {of $\partial \Omega_j$} that intersects only subdomain $k$. We remark that the set $\Gamma_j^{k, i}$ is open for every $j \in \{1, \ldots, N\}, k \in {N}_j$ and $i \in N_{jk}$.
	
	These definitions allow us to partition the `interior' part of the boundary $\Gamma_j^{\Int}$ in a natural manner. Indeed we obtain that 
	\begin{equation*}
	\Gamma_j^{{\Int}}= \overline{\cup_{k \in {N}_j} \cup_{i \in N_{jk}} \Gamma_{j}^{k, i}}.
	\end{equation*}
	
	We remark that by definition, for a fixed $j \in \{1, \ldots, N\}$ and $i, k \neq 0$ it holds that $\Gamma_j^{k, i}= \Gamma_j^{i, k}$.

	As the last step, we wish to introduce the notion of so-called \emph{skeletons} associated with each subdomain. To this end, let $j \in \{1, \ldots, N\}$ and $k \in {N}_j$ be fixed. Then we define the sets $\mathcal{S}_{j, k}^{{\Int}}$ and $\mathcal{S}_{j, k}$ as
	\begin{equation*}
	\mathcal{S}_{j, k} := \overline{\bigcup_{i \in N_{jk}}\Gamma_k^{j, i}}, \qquad
	\mathcal{S}_{j, k}^{{\Int}}  := \mathcal{S}_{j, k} \setminus \partial \Omega_j,
	\end{equation*}
	
	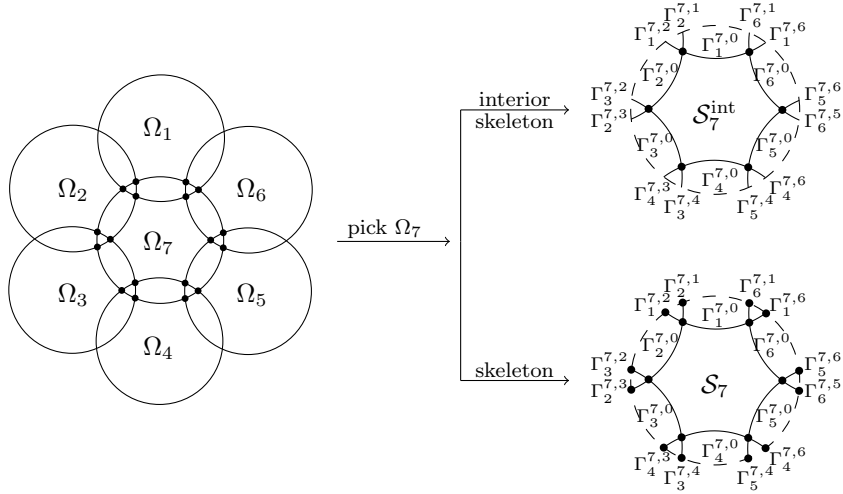
\begin{figure}[h]
		\centering
		\begin{tikzpicture}[scale=0.72]
		
		\begin{scope}[scale=0.75,xshift=10.5cm,yshift=4.2cm]
		\draw[black] (23.8,11.2) node {$\Omega_7$};
		\draw[black] (23.8,14.0) node {$\Omega_1$};
		\draw[black] (21.7,12.5) node {$\Omega_2$};
		\draw[black] (26.1,12.5) node {$\Omega_6$};
		\draw[black] (21.7,9.9) node {$\Omega_3$};
		\draw[black] (26.1,9.9) node {$\Omega_5$};
		\draw[black] (23.8,8.6) node {$\Omega_4$};
		\draw(26.02,12.5) circle (1.56cm);
		\draw(26,10) circle (1.56cm);
		\draw(23.82,8.77) circle (1.56cm);
		\draw(21.67,10.03) circle (1.56cm);
		\draw(21.69,12.53) circle (1.56cm);
		\draw(23.86,13.77) circle (1.56cm);
		\draw(23.84,11.26) circle (1.56cm);
		\fill [color=black] (22.29,11.46) circle (2.4pt);
		\fill [color=black] (22.29,11.09) circle (2.4pt);
		\fill [color=black] (22.61,11.28) circle (2.4pt);
		\fill [color=black] (22.89,10.02) circle (2.4pt);
		\fill [color=black] (23.22,10.21) circle (2.4pt);
		\fill [color=black] (23.22,9.83) circle (2.4pt);
		\fill [color=black] (24.45,9.82) circle (2.4pt);
		\fill [color=black] (24.45,10.2) circle (2.4pt);
		\fill [color=black] (24.77,10.01) circle (2.4pt);
		\fill [color=black] (25.39,11.07) circle (2.4pt);
		\fill [color=black] (25.08,11.26) circle (2.4pt);
		\fill [color=black] (25.39,11.44) circle (2.4pt);
		\fill [color=black] (24.78,12.5) circle (2.4pt);
		\fill [color=black] (24.47,12.69) circle (2.4pt);
		\fill [color=black] (24.47,12.33) circle (2.4pt);
		\fill [color=black] (23.24,12.7) circle (2.4pt);
		\fill [color=black] (23.24,12.34) circle (2.4pt);
		\fill [color=black] (22.92,12.52) circle (2.4pt);
		\end{scope}
		
		\begin{scope}[xshift=0.3cm]
		\draw [->](31,14)-- (33,14);
		\draw [->](31,9)-- (33,9);
		\draw (31,14)-- (31,11.57);
		\draw (31,11.57)-- (31,9);
		\draw [<-](30.8,11.57)-- (28.72,11.57);
		\begin{footnotesize}
		\draw[black] (29.6,11.80) node {pick $\Omega_7$};
		\draw[black] (32,14.2) node {interior};
		\draw[black] (32,13.8) node {skeleton};
		\draw[black] (32,9.2) node {skeleton};
		\end{footnotesize}
		\end{scope}
		
		\draw [dash pattern=on 5pt off 5pt] (36,9) circle (1.56cm);
		\draw [shift={(33.85,10.27)}] plot[domain=-1.18:0.11,variable=\t]({1*1.56*cos(\t r)+0*1.56*sin(\t r)},{0*1.56*cos(\t r)+1*1.56*sin(\t r)});
		\draw [shift={(33.83,7.77)}] plot[domain=-0.13:1.16,variable=\t]({1*1.56*cos(\t r)+0*1.56*sin(\t r)},{0*1.56*cos(\t r)+1*1.56*sin(\t r)});
		\draw [shift={(35.98,6.51)}] plot[domain=0.92:2.21,variable=\t]({1*1.56*cos(\t r)+0*1.56*sin(\t r)},{0*1.56*cos(\t r)+1*1.56*sin(\t r)});
		\draw [shift={(38.16,7.74)}] plot[domain=1.97:3.25,variable=\t]({1*1.56*cos(\t r)+0*1.56*sin(\t r)},{0*1.56*cos(\t r)+1*1.56*sin(\t r)});
		\draw [shift={(38.18,10.24)}] plot[domain=3.02:4.3,variable=\t]({1*1.56*cos(\t r)+0*1.56*sin(\t r)},{0*1.56*cos(\t r)+1*1.56*sin(\t r)});
		\draw [shift={(36.02,11.51)}] plot[domain=4.06:5.34,variable=\t]({1*1.56*cos(\t r)+0*1.56*sin(\t r)},{0*1.56*cos(\t r)+1*1.56*sin(\t r)});
		
		\draw[black] (36.0,8.90) node {$\mathcal{S}_{7}$};
		\begin{scriptsize}
		\draw[black] (36.1,10.20) node {$\Gamma_1^{7,0}$};
		\draw[black] (35.0,9.65) node {$\Gamma_2^{7,0}$};
		\draw[black] (37.05,9.65) node {$\Gamma_6^{7,0}$};
		\draw[black] (36.1,7.70) node {$\Gamma_4^{7,0}$};
		\draw[black] (34.9,8.35) node {$\Gamma_3^{7,0}$};
		\draw[black] (37.1,8.35) node {$\Gamma_5^{7,0}$};
		\draw[black] (35.4,10.75) node {$\Gamma_2^{7,1}$};
		\draw[black] (34.85,10.4) node {$\Gamma_1^{7,2}$};
		\draw[black] (36.8,10.75) node {$\Gamma_6^{7,1}$};
		\draw[black] (37.35,10.4) node {$\Gamma_1^{7,6}$};
		
		\draw[black] (38.0,8.80) node {$\Gamma_6^{7,5}$};
		\draw[black] (38.0,9.30) node {$\Gamma_5^{7,6}$};
		
		\draw[black] (34.05,9.30) node {$\Gamma_3^{7,2}$};
		\draw[black] (34.05,8.80) node {$\Gamma_2^{7,3}$};
		
		\draw[black] (35.4,7.2) node {$\Gamma_3^{7,4}$};
		\draw[black] (34.85,7.45) node {$\Gamma_4^{7,3}$};
		\draw[black] (36.75,7.2) node {$\Gamma_5^{7,4}$};
		\draw[black] (37.35,7.55) node {$\Gamma_4^{7,6}$};
		\end{scriptsize}
		
		\fill [color=black] (34.45,9.2) circle (2.1pt);
		\fill [color=black] (34.45,8.83) circle (2.1pt);
		\fill [color=black] (34.77,9.02) circle (2.1pt);
		\fill [color=black] (35.05,7.76) circle (2.1pt);
		\fill [color=black] (35.38,7.95) circle (2.1pt);
		\fill [color=black] (35.38,7.57) circle (2.1pt);
		\fill [color=black] (36.61,7.56) circle (2.1pt);
		\fill [color=black] (36.61,7.94) circle (2.1pt);
		\fill [color=black] (36.93,7.75) circle (2.1pt);
		\fill [color=black] (37.55,8.81) circle (2.1pt);
		\fill [color=black] (37.24,9) circle (2.1pt);
		\fill [color=black] (37.55,9.18) circle (2.1pt);
		\fill [color=black] (36.94,10.24) circle (2.1pt);
		\fill [color=black] (36.63,10.43) circle (2.1pt);
		\fill [color=black] (36.63,10.07) circle (2.1pt);
		\fill [color=black] (35.4,10.44) circle (2.1pt);
		\fill [color=black] (35.4,10.08) circle (2.1pt);
		\fill [color=black] (35.08,10.26) circle (2.1pt);
		
		\begin{scope}[yshift=5.0cm]
		\draw [dash pattern=on 5pt off 5pt] (36,9) circle (1.56cm);
		\draw [shift={(33.85,10.27)}] plot[domain=-1.18:0.11,variable=\t]({1*1.56*cos(\t r)+0*1.56*sin(\t r)},{0*1.56*cos(\t r)+1*1.56*sin(\t r)});
		\draw [shift={(33.83,7.77)}] plot[domain=-0.13:1.16,variable=\t]({1*1.56*cos(\t r)+0*1.56*sin(\t r)},{0*1.56*cos(\t r)+1*1.56*sin(\t r)});
		\draw [shift={(35.98,6.51)}] plot[domain=0.92:2.21,variable=\t]({1*1.56*cos(\t r)+0*1.56*sin(\t r)},{0*1.56*cos(\t r)+1*1.56*sin(\t r)});
		\draw [shift={(38.16,7.74)}] plot[domain=1.97:3.25,variable=\t]({1*1.56*cos(\t r)+0*1.56*sin(\t r)},{0*1.56*cos(\t r)+1*1.56*sin(\t r)});
		\draw [shift={(38.18,10.24)}] plot[domain=3.02:4.3,variable=\t]({1*1.56*cos(\t r)+0*1.56*sin(\t r)},{0*1.56*cos(\t r)+1*1.56*sin(\t r)});
		\draw [shift={(36.02,11.51)}] plot[domain=4.06:5.34,variable=\t]({1*1.56*cos(\t r)+0*1.56*sin(\t r)},{0*1.56*cos(\t r)+1*1.56*sin(\t r)});
		
		\draw[black] (36.0,8.90) node {$\mathcal{S}_{7}^{\Int}$};
		\begin{scriptsize}
		\draw[black] (36.1,10.20) node {$\Gamma_1^{7,0}$};
		\draw[black] (35.0,9.65) node {$\Gamma_2^{7,0}$};
		\draw[black] (37.05,9.65) node {$\Gamma_6^{7,0}$};
		\draw[black] (36.1,7.70) node {$\Gamma_4^{7,0}$};
		\draw[black] (34.9,8.35) node {$\Gamma_3^{7,0}$};
		\draw[black] (37.1,8.35) node {$\Gamma_5^{7,0}$};
		\draw[black] (35.4,10.75) node {$\Gamma_2^{7,1}$};
		\draw[black] (34.85,10.4) node {$\Gamma_1^{7,2}$};
		\draw[black] (36.8,10.75) node {$\Gamma_6^{7,1}$};
		\draw[black] (37.35,10.4) node {$\Gamma_1^{7,6}$};
		
		\draw[black] (38.0,8.80) node {$\Gamma_6^{7,5}$};
		\draw[black] (38.0,9.30) node {$\Gamma_5^{7,6}$};
		
		\draw[black] (34.05,9.30) node {$\Gamma_3^{7,2}$};
		\draw[black] (34.05,8.80) node {$\Gamma_2^{7,3}$};
		
		\draw[black] (35.4,7.2) node {$\Gamma_3^{7,4}$};
		\draw[black] (34.85,7.45) node {$\Gamma_4^{7,3}$};
		\draw[black] (36.75,7.2) node {$\Gamma_5^{7,4}$};
		\draw[black] (37.35,7.55) node {$\Gamma_4^{7,6}$};
		\end{scriptsize}
		
		%\fill [color=black] (34.45,9.2) circle (2.1pt);
		%\fill [color=black] (34.45,8.83) circle (2.1pt);
		\fill [color=black] (34.77,9.02) circle (2.1pt);
		%\fill [color=black] (35.05,7.76) circle (2.1pt);
		\fill [color=black] (35.38,7.95) circle (2.1pt);
		%\fill [color=black] (35.38,7.57) circle (2.1pt);
		%\fill [color=black] (36.61,7.56) circle (2.1pt);
		\fill [color=black] (36.61,7.94) circle (2.1pt);
		%\fill [color=black] (36.93,7.75) circle (2.1pt);
		%\fill [color=black] (37.55,8.81) circle (2.1pt);
		\fill [color=black] (37.24,9) circle (2.1pt);
		%\fill [color=black] (37.55,9.18) circle (2.1pt);
		%\fill [color=black] (36.94,10.24) circle (2.1pt);
		%\fill [color=black] (36.63,10.43) circle (2.1pt);
		\fill [color=black] (36.63,10.07) circle (2.1pt);
		%\fill [color=black] (35.4,10.44) circle (2.1pt);
		\fill [color=black] (35.4,10.08) circle (2.1pt);
		%\fill [color=black] (35.08,10.26) circle (2.1pt);
		\end{scope}
		\end{tikzpicture}
		\caption{The skeleton and interior skeleton corresponding to the subdomain $\Omega_7$ are shown. Notice that the difference between the skeleton and the interior skeleton of $\Omega_7$ given by the points of $\mathcal{S}_7$ that intersect $\partial \Omega_7$:
			$\mathcal{S}_7 \setminus \mathcal{S}_7^{\Int} =  \mathcal{S}_7 \cap \partial \Omega_7$.}
		\label{fig:3}
	\end{figure}
	
	Intuitively, for a fixed subdomain $\Omega_j$ and its neighbouring subdomain $\Omega_k$, the set $\mathcal{S}_{j, k}^{\Int}$ consists of the closure of all the interior boundaries of the subdomain $\Omega_k$ that intersect $\Omega_j$ \emph{excluding points on the boundary $\partial \Omega_j$}. Similarly, the set $\mathcal{S}_{j, k}$ consists of the closure of all the interior boundaries of the subdomain $\Omega_k$ that intersect {$\overline{\Omega}_j$}. It is now easy to see that for a fixed subdomain $\Omega_j$ it holds that
	\begin{equation*}
	\bigcup_{j \in {N}_k} \mathcal{S}_{j, k}= \Gamma_k^{{\Int}}.
	\end{equation*}
	Finally, for a fixed $j \in \{1, \ldots, N\}$ we define the sets $\mathcal{S}_{j}$ and $\mathcal{S}_{j}^{{\Int}}$ as
	\begin{equation*}
	\mathcal{S}_{j} := \bigcup_{k \in {N}_j} \mathcal{S}_{j, k},
	\quad
	\mathcal{S}_{j}^{{\Int}} := \bigcup_{k \in {N}_j} \mathcal{S}_{j, k}^{{\Int}},
	\end{equation*}
	and we say that $\mathcal{S}_j$ and $\mathcal{S}_j^{{\Int}}$ are the skeleton and the interior skeleton of the subdomain $\Omega_j$. An example of the skeleton and interior skeleton of a subdomain $\Omega_j$ is given in Figure \ref{fig:3}.
	
	\noindent\textbf{Notation:} Let $u_j \in L^{\infty}\big(\mathcal{S}_j\big)$ be a function. Then we define the infinity norm of $u_j$, denoted $\Vert u_j\Vert_{\infty}$, as \[\Vert u_j\Vert_{\infty}:= \text{ess}\sup_{\mathcal{S}_j} \vert u_j \vert.\] Additionally, let ${u}_j \in L^{\infty}\big(\mathcal{S}_j\big)$ for each  $j=1, \ldots, N$ be a family of functions and define $\bold{u}:= (u_1, u_2, \ldots, u_N)$. Then we define the infinity norm of $\bold{u}$, denoted $\Vert \bold{u}\Vert_{\infty}$ as \[\Vert \bold{u} \Vert_{\infty}:= \max_{j=1, \ldots, N} \text{ess}\sup_{\mathcal{S}_j} \vert u_j\vert.\]
	
	Finally, we remark that if $\bold{u}=(u_1, u_2, \ldots, u_N)$ is a continuous function, i.e., each $u_j, ~j=1, \ldots, N$ is continuous on $\mathcal{S}_j$ then we write $\bold{u} \in \Pi_{j=1}^N C^0(\mathcal{S}_j)$. Similarly, if each $u_j, ~j=1, \ldots, N$ is continuous on $\mathcal{S}^{\text{int}}_j$ then we write $\bold{u} \in \Pi_{j=1}^N C^0(\mathcal{S}^{\text{int}}_j)$.
	
	\subsubsection{Graph and layers of a domain}
	\begin{definition}[Graph of a domain $\Omega= \cup_{j=1}^N \Omega_j$]
		Consider a domain $\Omega= \cup_{j=1}^N \Omega_j$. We define the undirected graph $\mathcal{G}$ associated with the domain $\Omega$ as the set of vertices $\mathcal{V}$ given by $\mathcal{V}:= \{1, \ldots, N\}$,
		and the set of edges $E$ given by
		$E:= \left\{(i, j) \in \mathcal{V} \times \mathcal{V} \colon j \in N_i\right\}$. Furthermore, we say that the node $i \in \mathcal{V}$ is a {\rm{boundary node}} if $\partial \Omega_i \cap \partial \Omega$ is a set of measure greater than zero.
	\end{definition}
	
	The graph $\mathcal{G}$ associated with a domain $\Omega$ provides an easy and intuitive visualization of the connectivity of the domain. As we shall see, the convergence properties of the Schwarz method depend on the connectivity of the domain $\Omega$.

	\begin{definition}[Layers of a graph]
		Given a domain $\Omega= \cup_{j=1}^N \Omega_j$, consider the graph $\mathcal{G}$ associated with this domain. Then we define the layers of the graph $\mathcal{G}$ in an inductive manner as follows:
		\begin{enumerate}
			\item Layer $1$ is the set of all boundary nodes of $\mathcal{G}$. This set is denoted by $\mathcal{L}_1$.
			\item For any $j> 1$ we define the graph $\mathcal{G}_j$ iteratively as the set $\mathcal{W}_j:= \mathcal{V} - \cup_{k=1}^{j-1} \mathcal{L}_k$
			together with the associated set of edges. If $\mathcal{W}_j$ is non-empty, then Layer $j$ is the set $\mathcal{L}_j$ of all boundary nodes
			of $\mathcal{G}_j$.
		\end{enumerate}
		{Moreover, $N_{\max}$ denotes the total number of layers in the graph, and we say that the domain $\Omega$ has $N_{\max}$ layers.}
	\end{definition}

	An example of a domain and its decomposition into layers is shown in Figure \ref{fig:6}.
	\begin{figure}
		\centering
		\definecolor{zzzzzz}{rgb}{0.5,0.5,0.5}
		\definecolor{ffqqtt}{rgb}{1,0,0.2}
		\definecolor{qqqqcc}{rgb}{0,0,0.8}
		\definecolor{ttttff}{rgb}{0.2,0.2,1}
		\definecolor{ffqqqq}{rgb}{1,0,0}
		\definecolor{qqqqff}{rgb}{0,0,1}
		\begin{tikzpicture}[scale=1.7]
		\draw [dash pattern=on 3pt off 1pt,line width=0.9pt,color=zzzzzz] (0.8,1.49)-- (1.2,1.49);
		\draw [dash pattern=on 3pt off 1pt,line width=0.9pt,color=zzzzzz] (1.2,1.49)-- (1,1.84);
		\draw [dash pattern=on 3pt off 1pt,line width=0.9pt,color=zzzzzz] (1.2,1.49)-- (1.4,1.84);
		\draw [dash pattern=on 3pt off 1pt,line width=0.9pt,color=zzzzzz] (1.2,1.49)-- (1.6,1.49);
		\draw [dash pattern=on 3pt off 1pt,line width=0.9pt,color=zzzzzz] (1,1.15)-- (1.2,1.49);
		\draw [dash pattern=on 3pt off 1pt,line width=0.9pt,color=zzzzzz] (1,1.15)-- (1.4,1.15);
		\draw [dash pattern=on 3pt off 1pt,line width=0.9pt,color=zzzzzz] (1.4,1.15)-- (1.6,1.49);
		\draw [dash pattern=on 3pt off 1pt,line width=0.9pt,color=zzzzzz] (0.8,0.8)-- (1.2,0.8);
		\draw [dash pattern=on 3pt off 1pt,line width=0.9pt,color=zzzzzz] (1.2,0.8)-- (1,1.15);
		\draw [dash pattern=on 3pt off 1pt,line width=0.9pt,color=zzzzzz] (1,0.45)-- (1.2,0.8);
		\draw [dash pattern=on 3pt off 1pt,line width=0.9pt,color=zzzzzz] (1,0.45)-- (1.4,0.45);
		\draw [dash pattern=on 3pt off 1pt,line width=0.9pt,color=zzzzzz] (1.2,0.11)-- (1.4,0.45);
		\draw [dash pattern=on 3pt off 1pt,line width=0.9pt,color=zzzzzz] (1.4,0.45)-- (1.6,0.11);
		\draw [dash pattern=on 3pt off 1pt,line width=0.9pt,color=zzzzzz] (1.8,0.45)-- (1.6,0.11);
		\draw [dash pattern=on 3pt off 1pt,line width=0.9pt,color=zzzzzz] (1.8,0.45)-- (2,0.11);
		\draw [dash pattern=on 3pt off 1pt,line width=0.9pt,color=zzzzzz] (1.8,0.45)-- (2.2,0.45);
		\draw [dash pattern=on 3pt off 1pt,line width=0.9pt,color=zzzzzz] (2,0.8)-- (2.2,0.45);
		\draw [dash pattern=on 3pt off 1pt,line width=0.9pt,color=zzzzzz] (2,0.8)-- (2.4,0.8);
		\draw [dash pattern=on 3pt off 1pt,line width=0.9pt,color=zzzzzz] (1.8,1.15)-- (2.2,1.15);
		\draw [dash pattern=on 3pt off 1pt,line width=0.9pt,color=zzzzzz] (2.2,1.15)-- (2,0.8);
		\draw [dash pattern=on 3pt off 1pt,line width=0.9pt,color=zzzzzz] (1.2,0.8)-- (1.6,0.8);
		\draw [dash pattern=on 3pt off 1pt,line width=0.9pt,color=zzzzzz] (1.6,0.8)-- (1.4,0.45);
		\draw [dash pattern=on 3pt off 1pt,line width=0.9pt,color=zzzzzz] (1.6,0.8)-- (1.8,0.45);
		\draw [dash pattern=on 3pt off 1pt,line width=0.9pt,color=zzzzzz] (1.6,0.8)-- (2,0.8);
		\draw [dash pattern=on 3pt off 1pt,line width=0.9pt,color=zzzzzz] (1.6,0.8)-- (1.4,1.15);
		\draw [dash pattern=on 3pt off 1pt,line width=0.9pt,color=zzzzzz] (1.6,0.8)-- (1.8,1.15);
		\draw [dash pattern=on 3pt off 1pt,line width=0.9pt,color=zzzzzz] (1.6,1.49)-- (1.8,1.15);
		\draw [dash pattern=on 3pt off 1pt,line width=0.9pt,color=zzzzzz] (1.8,1.15)-- (2,1.49);
		\draw [dash pattern=on 3pt off 1pt,line width=0.9pt,color=zzzzzz] (2,1.49)-- (2.2,1.15);
		\draw [dash pattern=on 3pt off 1pt,line width=0.9pt,color=zzzzzz] (2.4,1.49)-- (2,1.49);
		\draw [dash pattern=on 3pt off 1pt,line width=0.9pt,color=zzzzzz] (2.2,1.84)-- (2.4,1.49);
		\draw [dash pattern=on 3pt off 1pt,line width=0.9pt,color=zzzzzz] (2.4,1.49)-- (2.6,1.84);
		\draw [color=ttttff] (0.8,0.8) circle (0.26cm);
		\draw [color=ttttff] (1,0.45) circle (0.26cm);
		\draw [color=ttttff] (1.2,0.11) circle (0.26cm);
		\draw [color=ffqqqq] (1.4,0.45) circle (0.26cm);
		\draw [color=ttttff] (1.6,0.11) circle (0.26cm);
		\draw [color=ttttff] (2,0.11) circle (0.26cm);
		\draw [color=ffqqqq] (1.8,0.45) circle (0.26cm);
		\draw [color=ttttff] (2.2,0.45) circle (0.26cm);
		\draw [color=ttttff] (2.4,0.8) circle (0.26cm);
		\draw [color=ffqqqq] (2,0.8) circle (0.26cm);
		\draw(1.6,0.8) circle (0.26cm);
		\draw [color=ffqqqq] (1.2,0.8) circle (0.26cm);
		\draw [color=ttttff] (1,1.15) circle (0.26cm);
		\draw [color=ttttff] (0.8,1.49) circle (0.26cm);
		\draw [color=ffqqqq] (1.2,1.49) circle (0.26cm);
		\draw [color=ttttff] (1,1.84) circle (0.26cm);
		\draw [color=ttttff] (1.4,1.84) circle (0.26cm);
		\draw [color=ttttff] (1.6,1.49) circle (0.26cm);
		\draw [color=ffqqqq] (1.4,1.15) circle (0.26cm);
		\draw [color=ffqqqq] (1.8,1.15) circle (0.26cm);
		\draw [color=ttttff] (2,1.49) circle (0.26cm);
		\draw [color=ttttff] (2.4,1.49) circle (0.26cm);
		\draw [color=ttttff] (2.2,1.84) circle (0.26cm);
		\draw [color=ttttff] (2.8,1.49) circle (0.26cm);
		\draw [color=ttttff] (2.2,1.15) circle (0.26cm);
		\draw [color=ttttff] (2.6,1.84) circle (0.26cm);
		\draw [color=ttttff] (3,1.15) circle (0.26cm);
		\draw [color=ttttff] (3.4,1.15) circle (0.26cm);
		\draw [color=ttttff] (0.4,0.8) circle (0.26cm);
		\draw [color=ttttff] (2.8,0.8) circle (0.26cm);
		\draw [line width=1.3pt,color=ttttff] (2.4,0.8)-- (2.8,0.8);
		\draw [line width=1.3pt,color=ttttff] (2.8,0.8)-- (3,1.15);
		\draw [line width=1.3pt,color=ttttff] (1,1.84)-- (0.8,1.49);
		\draw [line width=1.3pt,color=ttttff] (0.8,1.49)-- (1,1.15);
		\draw [line width=1.3pt,color=ttttff] (1,1.15)-- (0.8,0.8);
		\draw [line width=1.3pt,color=ttttff] (0.8,0.8)-- (1,0.45);
		\draw [line width=1.3pt,color=ttttff] (1,0.45)-- (1.2,0.11);
		\draw [line width=1.3pt,color=ttttff] (1.2,0.11)-- (1.6,0.11);
		\draw [line width=1.3pt,color=ttttff] (1.6,0.11)-- (2,0.11);
		\draw [line width=1.3pt,color=ttttff] (2,0.11)-- (2.2,0.45);
		\draw [line width=1.3pt,color=ttttff] (2.2,0.45)-- (2.4,0.8);
		\draw [line width=1.3pt,color=ttttff] (2.4,0.8)-- (2.2,1.15);
		\draw [line width=1.3pt,color=ttttff] (2.2,1.15)-- (2.4,1.49);
		\draw [line width=1.3pt,color=ttttff] (2.4,1.49)-- (2.8,1.49);
		\draw [line width=1.3pt,color=ttttff] (2.8,1.49)-- (2.6,1.84);
		\draw [line width=1.3pt,color=ttttff] (2.6,1.84)-- (2.2,1.84);
		\draw [line width=1.3pt,color=ttttff] (2.2,1.84)-- (2,1.49);
		\draw [line width=1.3pt,color=ttttff] (2,1.49)-- (1.6,1.49);
		\draw [line width=1.3pt,color=ttttff] (1.6,1.49)-- (1.4,1.84);
		\draw [line width=1.3pt,color=ttttff] (1.4,1.84)-- (1,1.84);
		\draw [line width=1.3pt,color=ttttff] (0.4,0.8)-- (0.8,0.8);
		\draw [line width=1.3pt,color=ttttff] (2.8,1.49)-- (3,1.15);
		\draw [line width=1.3pt,color=ttttff] (3,1.15)-- (3.4,1.15);
		\draw [line width=1.3pt,color=ffqqqq] (1.4,1.15)-- (1.2,0.8);
		\draw [line width=1.3pt,color=ffqqqq] (1.2,0.8)-- (1.4,0.45);
		\draw [line width=1.3pt,color=ffqqqq] (1.4,0.45)-- (1.8,0.45);
		\draw [line width=1.3pt,color=ffqqqq] (1.8,0.45)-- (2,0.8);
		\draw [line width=1.3pt,color=ffqqqq] (2,0.8)-- (1.8,1.15);
		\draw [line width=1.3pt,color=ffqqqq] (1.8,1.15)-- (1.4,1.15);
		\draw [line width=1.3pt,color=ffqqqq] (1.2,1.49)-- (1.4,1.15);
		
		\begin{footnotesize}
		\draw [<-] (-0.21,1.29)-- (-0.03,1.72);
		\draw[black] (-0.03,1.92) node {pending};
		\draw[black] (-0.03,1.80) node {subdomain};
		\draw [<-] (-3.2,0.63)-- (-3.16,0.28);
		\draw[black] (-3.16,0.21) node {pending};
		\draw[black] (-3.16,0.09) node {subdomain};
		\draw [<-] (-0.97,1.08)-- (-0.74,0.28);
		\draw[black] (-0.74,0.20) node {a ``hole''};
		\end{footnotesize}
		
		\draw(-0.8,0.8) circle (0.26cm);
		\draw(-2.8,0.8) circle (0.26cm);
		\draw(-2.6,0.45) circle (0.26cm);
		\draw(-2.4,0.11) circle (0.26cm);
		\draw(-2.2,0.45) circle (0.26cm);
		\draw(-2,0.11) circle (0.26cm);
		\draw(-1.6,0.11) circle (0.26cm);
		\draw(-1.8,0.45) circle (0.26cm);
		\draw(-1.4,0.45) circle (0.26cm);
		\draw(-1.2,0.8) circle (0.26cm);
		\draw(-1.6,0.8) circle (0.26cm);
		\draw(-2,0.8) circle (0.26cm);
		\draw(-2.4,0.8) circle (0.26cm);
		\draw(-2.6,1.15) circle (0.26cm);
		\draw(-2.8,1.49) circle (0.26cm);
		\draw(-2.4,1.49) circle (0.26cm);
		\draw(-2.6,1.84) circle (0.26cm);
		\draw(-2.2,1.84) circle (0.26cm);
		\draw(-2,1.49) circle (0.26cm);
		\draw(-2.2,1.15) circle (0.26cm);
		\draw(-1.8,1.15) circle (0.26cm);
		\draw(-1.6,1.49) circle (0.26cm);
		\draw(-1.2,1.49) circle (0.26cm);
		\draw(-1.4,1.84) circle (0.26cm);
		\draw(-0.8,1.49) circle (0.26cm);
		\draw(-1.4,1.15) circle (0.26cm);
		\draw(-1,1.84) circle (0.26cm);
		\draw(-0.6,1.15) circle (0.26cm);
		\draw(-0.2,1.15) circle (0.26cm);
		\draw(-3.2,0.8) circle (0.26cm);
		\draw (-3.2,0.8)-- (-2.8,0.8);
		\draw (-0.8,1.49)-- (-0.6,1.15);
		\draw (-0.6,1.15)-- (-0.2,1.15);
		\draw (-2.4,1.49)-- (-2.2,1.15);
		\draw (-2.8,1.49)-- (-2.4,1.49);
		\draw (-2.4,1.49)-- (-2.6,1.84);
		\draw (-2.4,1.49)-- (-2.2,1.84);
		\draw (-2.4,1.49)-- (-2,1.49);
		\draw (-2.6,1.15)-- (-2.4,1.49);
		\draw (-2.6,1.15)-- (-2.2,1.15);
		\draw (-2.2,1.15)-- (-2,1.49);
		\draw (-2.8,0.8)-- (-2.4,0.8);
		\draw (-2.4,0.8)-- (-2.6,1.15);
		\draw (-2.6,0.45)-- (-2.4,0.8);
		\draw (-2.6,0.45)-- (-2.2,0.45);
		\draw (-2.4,0.11)-- (-2.2,0.45);
		\draw (-2.2,0.45)-- (-2,0.11);
		\draw (-1.8,0.45)-- (-2,0.11);
		\draw (-1.8,0.45)-- (-1.6,0.11);
		\draw (-1.8,0.45)-- (-1.4,0.45);
		\draw (-1.6,0.8)-- (-1.4,0.45);
		\draw (-1.6,0.8)-- (-1.2,0.8);
		\draw (-1.8,1.15)-- (-1.4,1.15);
		\draw (-1.4,1.15)-- (-1.6,0.8);
		\draw (-2.4,0.8)-- (-2,0.8);
		\draw (-2,0.8)-- (-2.2,0.45);
		\draw (-2,0.8)-- (-1.8,0.45);
		\draw (-2,0.8)-- (-1.6,0.8);
		\draw (-2,0.8)-- (-2.2,1.15);
		\draw (-2,0.8)-- (-1.8,1.15);
		\draw (-2,1.49)-- (-1.8,1.15);
		\draw (-1.8,1.15)-- (-1.6,1.49);
		\draw (-1.6,1.49)-- (-1.4,1.15);
		\draw (-1.2,1.49)-- (-1.6,1.49);
		\draw (-1.4,1.84)-- (-1.2,1.49);
		\draw (-1.2,1.49)-- (-1,1.84);
		\draw (-2.6,1.15)-- (-2.8,1.49);
		\draw (-2.8,1.49)-- (-2.6,1.84);
		\draw (-2.6,1.84)-- (-2.2,1.84);
		\draw (-2.2,1.84)-- (-2,1.49);
		\draw (-2,1.49)-- (-1.6,1.49);
		\draw (-1.6,1.49)-- (-1.4,1.84);
		\draw (-1.4,1.84)-- (-1,1.84);
		\draw (-1,1.84)-- (-0.8,1.49);
		\draw (-0.8,1.49)-- (-1.2,1.49);
		\draw (-1.4,1.15)-- (-1.2,1.49);
		\draw (-1.4,1.15)-- (-1.2,0.8);
		\draw (-1.2,0.8)-- (-1.4,0.45);
		\draw (-1.4,0.45)-- (-1.6,0.11);
		\draw (-1.6,0.11)-- (-2,0.11);
		\draw (-2.8,0.8)-- (-2.6,0.45);
		\draw (-2.8,0.8)-- (-2.6,1.15);
		\draw (-2.4,0.8)-- (-2.2,1.15);
		\draw (-2.2,1.15)-- (-1.8,1.15);
		\draw (-1.8,1.15)-- (-1.6,0.8);
		\draw (-1.6,0.8)-- (-1.8,0.45);
		\draw (-2.6,0.45)-- (-2.4,0.11);
		\draw (-2.4,0.11)-- (-2,0.11);
		\draw (-2.2,0.45)-- (-1.8,0.45);
		\draw (-2.4,0.8)-- (-2.2,0.45);
		\draw (-1.2,0.8)-- (-0.8,0.8);
		\draw (-0.8,0.8)-- (-0.6,1.15);
		\begin{scriptsize}
		\fill [color=qqqqff] (0.8,0.8) circle (1.8pt);
		\fill [color=ffqqqq] (1.2,0.8) circle (1.8pt);
		\fill [color=qqqqff] (1,1.15) circle (1.8pt);
		\fill [color=black] (1.6,0.8) circle (1.8pt);
		\fill [color=ffqqqq] (1.4,1.15) circle (1.8pt);
		\fill [color=ffqqqq] (1.2,1.49) circle (1.8pt);
		\fill [color=ffqqqq] (1.8,1.15) circle (1.8pt);
		\fill [color=qqqqff] (1.6,1.49) circle (1.8pt);
		\fill [color=ffqqqq] (2,0.8) circle (1.8pt);
		\fill [color=qqqqff] (2.2,1.15) circle (1.8pt);
		\fill [color=qqqqff] (2,1.49) circle (1.8pt);
		\fill [color=qqqqff] (2.4,0.8) circle (1.8pt);
		\fill [color=qqqqff] (2.4,1.49) circle (1.8pt);
		\fill [color=qqqqff] (2.2,1.84) circle (1.8pt);
		\fill [color=qqqqff] (1.4,1.84) circle (1.8pt);
		\fill [color=qqqqff] (2.8,1.49) circle (1.8pt);
		\fill [color=qqqqff] (2.6,1.84) circle (1.8pt);
		\fill [color=qqqqff] (3,1.15) circle (1.8pt);
		\fill [color=qqqqff] (1,0.45) circle (1.8pt);
		\fill [color=ffqqqq] (1.4,0.45) circle (1.8pt);
		\fill [color=ffqqqq] (1.8,0.45) circle (1.8pt);
		\fill [color=qqqqff] (1.2,0.11) circle (1.8pt);
		\fill [color=qqqqff] (0.8,1.49) circle (1.8pt);
		\fill [color=qqqqff] (1,1.84) circle (1.8pt);
		\fill [color=qqqqff] (1.6,0.11) circle (1.8pt);
		\fill [color=qqqqff] (2.2,0.45) circle (1.8pt);
		\fill [color=qqqqff] (2,0.11) circle (1.8pt);
		\fill [color=qqqqff] (3.4,1.15) circle (1.8pt);
		\fill [color=qqqqff] (0.4,0.8) circle (1.8pt);
		\fill [color=black] (-2.8,0.8) circle (1.5pt);
		\fill [color=black] (-2.4,0.8) circle (1.5pt);
		\fill [color=black] (-2.6,1.15) circle (1.5pt);
		\fill [color=black] (-2,0.8) circle (1.5pt);
		\fill [color=black] (-2.2,1.15) circle (1.5pt);
		\fill [color=black] (-2.4,1.49) circle (1.5pt);
		\fill [color=black] (-1.8,1.15) circle (1.5pt);
		\fill [color=black] (-2,1.49) circle (1.5pt);
		\fill [color=black] (-1.6,0.8) circle (1.5pt);
		\fill [color=black] (-1.4,1.15) circle (1.5pt);
		\fill [color=black] (-1.6,1.49) circle (1.5pt);
		\fill [color=black] (-1.2,0.8) circle (1.5pt);
		\fill [color=black] (-1.2,1.49) circle (1.5pt);
		\fill [color=black] (-1.4,1.84) circle (1.5pt);
		\fill [color=black] (-2.2,1.84) circle (1.5pt);
		\fill [color=black] (-0.8,1.49) circle (1.5pt);
		\fill [color=black] (-1,1.84) circle (1.5pt);
		\fill [color=black] (-0.6,1.15) circle (1.5pt);
		\fill [color=black] (-2.6,0.45) circle (1.5pt);
		\fill [color=black] (-2.2,0.45) circle (1.5pt);
		\fill [color=black] (-1.8,0.45) circle (1.5pt);
		\fill [color=black] (-2.4,0.11) circle (1.5pt);
		\fill [color=black] (-2.8,1.49) circle (1.5pt);
		\fill [color=black] (-2.6,1.84) circle (1.5pt);
		\fill [color=black] (-2,0.11) circle (1.5pt);
		\fill [color=black] (-1.4,0.45) circle (1.5pt);
		\fill [color=black] (-1.6,0.11) circle (1.5pt);
		\fill [color=black] (-0.2,1.15) circle (1.5pt);
		\fill [color=black] (-3.2,0.8) circle (1.5pt);
		\fill [color=black] (-0.8,0.8) circle (1.5pt);
		\fill [color=black] (-0.6,1.15) circle (1.5pt);
		\fill [color=qqqqff] (2.8,0.8) circle (1.8pt);
		\fill [color=qqqqff] (3,1.15) circle (1.8pt);
		\end{scriptsize}
		\end{tikzpicture}
		\caption{Left: An example of a domain {$\Omega= \cup_{j=1}^N \Omega_j$ (as a collection of disks)} and its corresponding
			graph (black nodes and edges). Notice the presence of two ``pending subdomains''
			and a ``hole''. Right: Layers corresponding to the left figure. In particular,
			the blue nodes represent $\mathcal{L}_1$ (Layer 1), the red nodes $\mathcal{L}_2$ (Layer 2), and the black node $\mathcal{L}_3$ (Layer 3).}
		\label{fig:6}
	\end{figure}
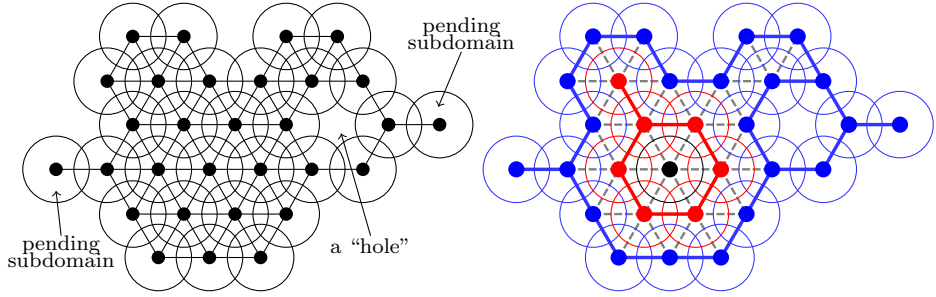
	
	\begin{definition}
		Let $\Omega= \cup_{j=1}^N \Omega_j$ be a domain with $N_{\max}$ layers. Then for any natural number $n\leq N_{\max}$ we define the sets $\mathcal{V}_n$ and $\mathcal{C}_n$ as follows:
		\begin{equation*}
		\begin{split}
		\mathcal{V}_n &:= \left\{ \bold{v} \in \Pi_{j=1}^N C^0(\mathcal{S}^{\Int}_j) \colon \Vert \bold{v}\Vert_{\infty}\leq1 \text{ and }\begin{cases}
		\bold{v}_j(x) < 1 \; &\forall x \in \mathcal{S}_j^{\Int}, \forall j \in \cup_{k=1}^n\mathcal{L}_k,\\
		\bold{v}_j =1 &~\quad \text{on } \mathcal{S}_j,\forall j \in \cup_{k=n+1}^{N_{\max}}\mathcal{L}_k.
		\end{cases}
		\right\}, \\
		\mathcal{C}_n &:= \Big\{ \bold{v} \in \Pi_{j=1}^N C^0(\mathcal{S}^{\Int}_j) \colon \Vert \bold{v}\Vert_{\infty}\leq1 \text{ and } \text{ess}\sup_{\mathcal{S}_j} \bold{v}_j < 1, ~\forall j \in \cup_{k=1}^n\mathcal{L}_k\Big\}.
		\end{split}
		\end{equation*}
	\end{definition}
	
	Intuitively, a non-negative function $\bold{v} \in \Pi_{j=1}^N C^0(\mathcal{S}_j)$ is in the set $\mathcal{V}_n$ if $\bold{v}_j(x) < 1$ for all points $x$ on the \emph{interior skeleton} of all subdomains $\Omega_j$ that belong to the first $n$ layers of the domain. Notice that there is no constraint on the behaviour of the function $\bold{v}_j$ at the endpoints of the skeleton $\mathcal{S}_j$. Therefore, in some sense, the function $\bold{v}$ has \emph{begun to experience a contraction} on all subdomains in the first $n$ layers, but it cannot yet be claimed that the function $\bold{v}$ has infinity norm strictly smaller than one on the first $n$ layers.
	
	On the other hand, a non-negative function $\bold{w}$ in $\Pi_{j=1}^N C^0(\mathcal{S}^{\rm int}_j)$ is in the set $\mathcal{C}_n$ if $\text{ess}\sup_{\mathcal{S}_j} \bold{w}_j < 1$ on all subdomains $\Omega_j$ that belong to the first $n$ layers of the domain. Therefore, the function $\bold{w}$ has \emph{already contracted} on all subdomains in the first $n$ layers and thus has infinity norm less than one on these subdomains. 
	
	\subsubsection{Partition of unity, extension and restriction operators}
	Let $j \in \{1, \ldots, N\}$, we define for each $k \in N_j$ a function $\chi_j^{k} \colon \partial \Omega_j \rightarrow \mathbb{R}$, continuous on $\text{int} \big(\Gamma_j^{\Int}\big)$, with the property that
	\begin{equation}\label{eq:PU}
	\chi_j^{k}:= \begin{cases}
	1 \quad &\text{on } \overline{\Gamma_j^{k, 0}} \setminus \Gamma_j^{\Ext},~~~~~~\\
	\in [0, 1] \quad &\text{on } \overline{{\Gamma_j^{k, i}}} \setminus \Gamma_j^{\Ext} \quad ~~~~~~\text{for } i \in N_{j, k} \text{ with } i\neq 0,\\
	0 \quad &\text{otherwise},
	\end{cases}
	\end{equation}
	and such that
	\begin{equation}\label{eq:new}
	\sum_{k \in {N}_j} \chi_j^{k}(x)=1 \quad \text{ for all } x \in \text{int}\big(\Gamma_{j}^{{\Int}}\big).
	\end{equation}
	We say that $\{\chi_j^{k}\}_{\substack{k \in {N}_j}}$ are the partition of unity functions on $\Gamma_{j}^{{\Int}}$.
	
	{
		\begin{remark}\label{new:remark}
			Let us consider a general domain $\Omega= \cup_{i=1}^N \Omega_j$. If for all $j =1, \ldots, N$, no set $\Gamma_j^{\rm ext}$ has an isolated point, then we can readily use Equation \eqref{eq:PU} to define the partition of unity functions. In the pathological case where some set $\Gamma_j^{\rm ext}$ has an isolated point (see Figure \ref{fig:forHassan}), Equation \eqref{eq:PU} does not provide a correct function definition. In this pathological situation, we can modify Definition \eqref{eq:PU} by setting at least one partition of unity function to be non-zero at this isolated point. This modification preserves the continuity requirement on the partition of unity functions.
		\end{remark}
	}
	\begin{figure}[h]
		\centering
		\begin{tikzpicture}[scale=0.75]
		\draw(29.5,12) circle (1.56cm);
		\draw(27.5,12) circle (1.56cm);
		\draw(28.5,11.64) circle (1.56cm);
		\begin{scriptsize}
		\fill (28.5,13.2) circle (1.5pt);
		\draw (28.5,13.4) node {$P$};
		\draw (28.5,10.4) node {$\Omega_1$};
		\draw (30.5,12.5) node {$\Omega_2$};
		\draw (26.5,12.5) node {$\Omega_3$};
		\end{scriptsize}
		
		\end{tikzpicture}
		\caption{ An example of a pathological geometry where the point $P$ is an isolated point of $\Gamma_1^{\rm ext}$. In this case, one must slightly modify the definitions of the partition of unity functions $\chi_1^2$ and $\chi_1^3$ at the point $P$ in order to ensure that both functions are well-defined and continuous on \text{int}$\big(\Gamma_1^{\rm int}\big)$.}%that is excluded by imposing constraint A3). We show in Remark \ref{rem:Hassan} that our formulation of the parallel Schwarz method for continuous boundary data in this case leads to an ill-posed problem.}
		\label{fig:forHassan}
	\end{figure}
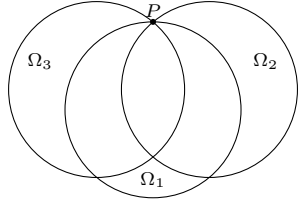
	
	\begin{remark}\label{rem:1}
		Consider a subdomain $\Omega_j$ and the partition of unity functions $\{\chi_j^{k}\}_{\substack{k \in {N}_j}}$. It follows from the definition of these functions that
		\begin{equation*}
		\sum_{k \in {N}_j} \chi_j^{k}(x)=0 \quad \text{ for all } x \in \Gamma_{j}^{\text{{\rm ext}}}.
		\end{equation*}
		
		Consequently, there is a subset of the boundary $\partial \Omega_j$ with non-zero Lebesgue measure on which all partition of unity functions are zero if and only if $\partial \Omega_j \cap \partial \Omega$ is a set of measure greater than zero.
	\end{remark}
	
	\begin{remark}\label{rem:0}
		Consider the partition of unity functions defined through Equation \eqref{eq:PU}. Two natural questions arise: 1) Is the assumption of continuity that we have imposed on the partition of unity functions truly necessary? 2) Does {the choice} of the partition of unity functions affect the iterates of the Parallel Schwarz method and the asymptotic contraction factor? It can indeed be verified that the proofs of all lemmas in Section \ref{sec:Technical} still hold if we drop the continuity assumption on the partition of unity functions. On the other hand, numerical experiments in Section \ref{sec:num} reveal that different choices of the partition of unity functions can lead to quantitatively slightly different Schwarz iterates.
	\end{remark}
	
	Next, we define our harmonic extension operators. To this end, let $\Gamma_j \subseteq \partial \Omega_j$ be a non-empty open set. Then we define the mapping $\mathcal{E}_j \colon L^2(\Gamma_j) \rightarrow C^0(\Omega_j)$ as the operator with the property that given any input function $u_j \in L^2(\Gamma_j)$, the output function $w_j:=\mathcal{E}_j (u_j) \in C^0(\Omega_j)$ is the unique solution to the {Dirichlet problem}
	\begin{align*}
	\Delta w_j &=0 &&\text{ in $\Omega_j$},\\
	w_j&=u_j &&\text{ on $\Gamma_j$},\\
	w_j&=0 &&\text{ on $\partial \Omega_j \setminus\Gamma_j$}.
	\end{align*}

	Finally, we define our restriction operators. Let $C^{0, \text{pc}}(\overline{\Omega_j})$ denote the set of functions on $\overline{\Omega_j}$ that are continuous on $\Omega_j$ and \emph{piecewise continuous} on the boundary $\partial \Omega_j$. {Then, we define the mapping $\mathcal{R}_j \colon C^{0, \text{pc}}(\overline{\Omega_j})\rightarrow C^0(\mathcal{S}_j^{\text{int}}) \cap L^{\infty}(\mathcal{S}_j)$ as}
	\begin{equation*}
	\mathcal{R}_j \colon u \mapsto \mathcal{R}_j(u) = u \vert_{\mathcal{S}_j},
	\end{equation*}
	and we say that $\mathcal{R}_j$ is the restriction operator on $\Omega_j$.
	
	Intuitively, the map $\mathcal{R}_j$ takes as input a continuous function ${v}_j$ defined on $\Omega_j$ with piecewise continuous values on the boundary $\partial \Omega_j$ and gives as output the restriction of this function on the skeleton $\mathcal{S}_j$. Notice that we cannot claim a priori that the output function $\mathcal{R}_j({v}_j)$ is continuous on the entire skeleton $\mathcal{S}_j$ since it is possible that there is a jump discontinuity at the endpoints of the skeleton $\mathcal{S}_j$ which lie on the boundary $\partial \Omega_j$.
	
	In what follows, we will often consider the composition $\mathcal{R}_j \big(\mathcal{E}_j (\lambda)\big)$ where $\lambda \in L^2(\Gamma_j)$ and $\Gamma_j \subset \partial \Omega_j$. We remark that this composition is well defined if $\lambda$ is a piecewise continuous function. A schematic representation of the mappings $\mathcal{E}_j$ and $\mathcal{R}_j$ is given in Figure \ref{fig:4}.
	\begin{figure}[h]
		\centering
		\begin{tikzpicture}[scale=0.6]
		\draw [fill=black,fill opacity=0.4] (-1,5) circle (1.56cm);
		\draw [dash pattern=on 3pt off 3pt] (5,5) circle (1.56cm);
		\draw [shift={(2.85,6.27)}] plot[domain=-1.18:0.11,variable=\t]({1*1.56*cos(\t r)+0*1.56*sin(\t r)},{0*1.56*cos(\t r)+1*1.56*sin(\t r)});
		\draw [shift={(2.83,3.77)}] plot[domain=-0.13:1.16,variable=\t]({1*1.56*cos(\t r)+0*1.56*sin(\t r)},{0*1.56*cos(\t r)+1*1.56*sin(\t r)});
		\draw [shift={(4.98,2.51)}] plot[domain=0.92:2.21,variable=\t]({1*1.56*cos(\t r)+0*1.56*sin(\t r)},{0*1.56*cos(\t r)+1*1.56*sin(\t r)});
		\draw [shift={(7.16,3.74)}] plot[domain=1.97:3.25,variable=\t]({1*1.56*cos(\t r)+0*1.56*sin(\t r)},{0*1.56*cos(\t r)+1*1.56*sin(\t r)});
		\draw [shift={(7.18,6.24)}] plot[domain=3.02:4.3,variable=\t]({1*1.56*cos(\t r)+0*1.56*sin(\t r)},{0*1.56*cos(\t r)+1*1.56*sin(\t r)});
		\draw [shift={(5.02,7.51)}] plot[domain=4.06:5.34,variable=\t]({1*1.56*cos(\t r)+0*1.56*sin(\t r)},{0*1.56*cos(\t r)+1*1.56*sin(\t r)});
		\draw [shift={(-7,5)},line width=2.8pt,color=gray]  plot[domain=1.15:1.98,variable=\t]({1*1.56*cos(\t r)+0*1.56*sin(\t r)},{0*1.56*cos(\t r)+1*1.56*sin(\t r)});
		\draw(-7,5) circle (1.56cm);
		\draw[black] (-4,5.3) node {$\mathcal{E}_j$};
		\draw[|->] (-5,5)-- (-3,5);
		\draw[black] (2,5.3) node {$\mathcal{R}_j$};
		\draw[|->] (1,5)-- (3,5);
		\draw[black] (-7,3) node {$\partial \Omega_j$};
		\draw[black] (-1,3) node {$\overline{\Omega}_j$};
		\draw[black] (5,3) node {$\mathcal{S}_j$};
		\draw[black] (-7,6) node {$u_j$};
		\draw[black] (-7,7) node {$\Gamma_j$};
		\draw[black] (-1,5) node {$w_j:=\mathcal{E}_j(u_j)$};
		\draw[black] (5,5) node {$\mathcal{R}_j(w_j)$};
		\fill [color=black] (3.45,5.2) circle (2.1pt);
		\fill [color=black] (3.45,4.83) circle (2.1pt);
		\fill [color=black] (3.77,5.02) circle (2.1pt);
		\fill [color=black] (4.05,3.76) circle (2.1pt);
		\fill [color=black] (4.38,3.95) circle (2.1pt);
		\fill [color=black] (4.38,3.57) circle (2.1pt);
		\fill [color=black] (5.61,3.56) circle (2.1pt);
		\fill [color=black] (5.61,3.94) circle (2.1pt);
		\fill [color=black] (5.93,3.75) circle (2.1pt);
		\fill [color=black] (6.55,4.81) circle (2.1pt);
		\fill [color=black] (6.24,5) circle (2.1pt);
		\fill [color=black] (6.55,5.18) circle (2.1pt);
		\fill [color=black] (5.94,6.24) circle (2.1pt);
		\fill [color=black] (5.63,6.43) circle (2.1pt);
		\fill [color=black] (5.63,6.07) circle (2.1pt);
		\fill [color=black] (4.4,6.44) circle (2.1pt);
		\fill [color=black] (4.4,6.08) circle (2.1pt);
		\fill [color=black] (4.08,6.26) circle (2.1pt);
		\end{tikzpicture}
		\caption{Representation of the maps $\mathcal{E}_j$ and $\mathcal{R}_j$:
			a function $u_j: \Gamma_j \rightarrow \mathbb{R}$ is extended
			harmonically in $\Omega_j$ by $\mathcal{E}_j$.
			The harmonic extension $w_j:=\mathcal{E}_j(u_j)$ is then restricted
			to the skeleton $\mathcal{S}_j$ by the map $\mathcal{R}_j$.}
		\label{fig:4}
	\end{figure}
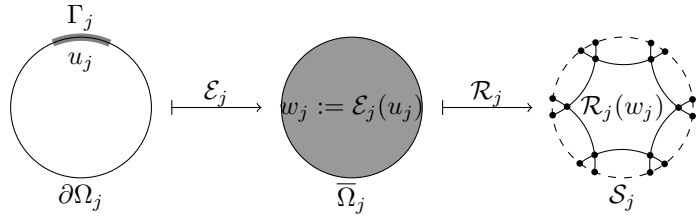
	
	\newpage
	\subsection{The Schwarz method and the operator formulation}\label{sec:2.2}
	We wish to apply the {parallel Schwarz method (PSM)} in order to obtain the solution to Equation \eqref{eq:1}. To this end, we consider the natural decomposition of the domain $\Omega$ into the $N$ subdomains~$\{\Omega_i\}_{i=1}^N$. Then for each $j \in \{1, \ldots, N\}$ we find the function $u_j \in L^{\infty}(\Omega_j)$ such that 
	\begin{align}\nonumber
	%\begin{split}
	\Delta u_j &=0 \quad &&\text{in }\Omega_j,\\ 
	u_j&=g \quad && \text{on } \Gamma_j^{\text{ext}},\label{eq:2} \\ 
	u_j&=\sum_{i \in N_j}u_i\chi_j^{i} \quad &&\text{in } \partial \Omega_j \setminus \Gamma_j^{{\Ext}}. \nonumber
	\end{align}
	Intuitively {and according to \eqref{eq:2},} we look for harmonic functions $u_j$ on each subdomain 
	$\Omega_j$ that satisfy appropriate boundary conditions. {On the exterior boundary we impose the true
		data $g$, while on the interior boundary we impose Dirichlet data, which is traced from the neighbouring subdomain solutions and multiplied by our partition of unity functions.}
	
	{The equivalence between the global Laplace problem \eqref{eq:1} and the domain decomposition problem \eqref{eq:2} will be proved later in Section \ref{sec:ConvergenceResults} Theorem \ref{lemma:2.1}.}
	
	Using {\eqref{eq:2}}, we obtain the following implementation of the PSM: Let $u_0 \colon \overline{\Omega} \rightarrow \mathbb{R}$ be some continuous initialization. For each $j \in \{1, \ldots, N\}$ find a sequence of functions $\{u^n_j\}_{n \in \mathbb{N}} \in L^{\infty}(\Omega_j)$ such that $u_j^0= u_0\vert_{\Omega_j}$ and for each $n \in \mathbb{N}$ it holds that
	\begin{align}
	%\begin{split}
	\nonumber \Delta u^{n+1}_j &=0 \quad && \text{in } \Omega_j,\\ 
	u^{n+1}_j&=g \quad &&\text{on }\Gamma_j^{\text{ext}}, \label{eq:3} \\ \nonumber
	u^{n+1}_j&=\sum_{i \in N_j}u^n_i\chi_j^{i} \quad &&\text{in } \partial \Omega_j \setminus \Gamma_j^{{\Ext}}.
	\end{align}
	
	In order to analyse the sequence of functions $\{u^n_j\}_{n \in \mathbb{N}}$ we consider the error equation associated with \eqref{eq:3}. To this end, we define for each $j \in \{1, \ldots, N\}$ and each $n \in \mathbb{N}$ the error functions $e_j^n:= u_j-u_j^{n}$. It follows that for each $j \in \{1, \ldots, N\}$ the sequence of error functions $\{e_j^n\}_{n \in \mathbb{N}}$ satisfies the equation
	\begin{align}
	%\begin{split}
	\nonumber \Delta e^{n+1}_j &=0 \quad &&\text{in }\Omega_j,\\ 
	e^{n+1}_j&=0 \quad &&\text{on } \Gamma_j^{\text{ext}}, \label{eq:4} \\ \nonumber
	e^{n+1}_j&=\sum_{i \in N_j}e^n_i\chi_j^{i} \quad &&\text{in }\partial \Omega_j \setminus \Gamma_j^{{\Ext}}.%\Gamma_j^{i, 0}, ~~ \forall ~i \in \mathcal{N}_j. \nonumber
	%\end{split}
	\end{align}
	
	The key step is now to recognize that a convergence analysis of \eqref{eq:4} is easier to perform using the notion of the skeleton that we introduced earlier. {To do so,} for each $n \in \mathbb{N}$ we define an $N$-dimensional vector $\bold{e}^n$ of error functions as
	\begin{equation}\label{def:Error}
	\bold{e}^n:= \begin{bmatrix}
	\bold{e}^n_1\\
	\bold{e}^n_2\\
	\vdots\\
	\bold{e}^n_N
	\end{bmatrix}.
	\end{equation}
	Here, each element $\bold{e}^n_j$ is defined as
	\begin{equation*}
	\bold{e}^n_j:= e_j^n \vert_{\mathcal{S}_j}.
	\end{equation*}
	In other words, the entry $\bold{e}^n_j$ of the vector $\bold{e}^n$ is the restriction of the error function $e_j^n$ on the skeleton $\mathcal{S}_j$ of the subdomain $\Omega_j$. 
	
	Next, we introduce the iteration operator corresponding to the above choice of the error vector $\bold{e}^n$. To do so, we define the $N \times N$ matrix $T$ by setting
	\begin{equation}\label{Def:Iteration}
	T_{ij}:= \begin{cases}
	P_{ij} \quad &\text{if } j \in N_i,\\
	0 \quad &\text{otherwise. }
	\end{cases}
	\end{equation}
	
	The entries $P_{ij}$ of the iteration operator $T$ are operators defined as follows: For each fixed $i \in \{1, \ldots, N\}$ and $j \in N_i$ the mapping $P_{ij} \colon C^0(\mathcal{S}_{j, i}) \rightarrow C^0(\mathcal{S}^{\Int}_i) \cap L^{\infty}(\mathcal{S}_i)$ is a linear operator such that for all ${v}_{j, i} \in C^0(\mathcal{S}_{j, i})$ it holds that
	\begin{equation*}
	P_{ij} \bold{v} = \mathcal{R}_i\Bigg(\mathcal{E}_i\Big( {v}_{j, i}\chi_i^{j}\vert_{\mathcal{S}_{j, i}}\Big)\Bigg).
	\end{equation*}
	In other words, for a given subdomain $\Omega_i$ and a given neighbour $\Omega_j$, the mapping $P_{ij}$
	
	\begin{enumerate}
		\item takes as input some function ${v}_{j, i}$ defined on the skeleton $\mathcal{S}_{j, i}$, i.e., the part of the interior boundary of subdomain $\Omega_i$ that is {contained in} $\Omega_j$, and multiplies it with the partition of unity function $\chi_i^j$,
		
		\item extends the function ${v} \chi_i^j \vert_{\mathcal{S}_{j, i}}$ harmonically inside the domain $\Omega_i$,
		
		\item and then yields as output the restriction of this harmonic extension on the skeleton $\mathcal{S}_i$.
	\end{enumerate}
	
	We remark that the definition of the iteration operator $T$ implies that it is block-sparse. Indeed, the $i^{\text{th}}$ row of $T$ contains non-zero entries exactly at columns $j \in N_i$. 
	
	\begin{example}
		Let us consider the situation of a domain $\Omega$ consisting of seven subdomains, i.e., $\Omega= {\cup_{j=1}^7} \Omega_j$ as {shown} in Figure \ref{fig:3}.
		In this setting the iteration operator $T$ is a $7 \times 7$ matrix given by
		\begin{equation*}
		T= \begin{bmatrix}
		0& P_{12}& 0& 0& 0 & P_{16} &P_{17}\\
		P_{21}& 0& P_{23}& 0& 0 & 0 &P_{27}\\
		0& P_{32}& 0& P_{34}& 0 & 0 &P_{37}\\
		0& 0& P_{43}& 0& P_{45} & 0 &P_{47}\\
		0& 0& 0& P_{54}& 0 & P_{56} &P_{57}\\
		P_{61}& 0& 0& 0& P_{65} & 0 &P_{67}\\
		P_{71}& P_{72}& P_{73}& P_{74}& P_{75} & P_{76} &0\\
		\end{bmatrix}.
		\end{equation*}
	\end{example}
	
	{Using \eqref{def:Error} and \eqref{Def:Iteration} we can rewrite \eqref{eq:4} as follows}
	\begin{align}\label{eq:5}
	\bold{e}^{n+1} = T \bold{e}^n \qquad \text{for each } n \in \mathbb{N}.
	\end{align}
	
	\vspace{1cm}
	\begin{lemma}\label{lem:2.8}
		Equation \eqref{eq:4} is equivalent to Equation \eqref{eq:5}.
	\end{lemma}
	\begin{proof}
		We show that \eqref{eq:5} can equivalently be rewritten as \eqref{eq:4}. To this end, let $n \in \mathbb{N}$. Then Equation \eqref{eq:5} can be written as
		\begin{align*}
		\bold{e}^n_j&= \sum_{k=1}^N T_{jk} \bold{e}^{n-1}_k \qquad \forall  j \in \{1, \ldots, N\},\\
		\intertext{which is equivalent to}
		\bold{e}^n_j&= \sum_{k \in N_j} P_{jk} \bold{e}^{n-1}_k \qquad \forall  j \in \{1, \ldots, N\}.
		\end{align*}
		Using the definition of the operators $P_{ik}$ we obtain for all $j \in \{1, \ldots, N\}$
		\begin{equation}\label{eq:2bb}
		\bold{e}^n_j
		= \sum_{k \in N_j} \mathcal{R}_j\Bigg(\mathcal{E}_j\Big(\bold{e}^{n-1}_k\vert_{\mathcal{S}_{k, j}}\, \chi_j^{k} \vert_{\mathcal{S}_{k, j}}\Big)\Bigg) 
		= \mathcal{R}_j\Bigg(\mathcal{E}_j\Big( \sum_{k \in N_j}\bold{e}^{n-1}_k \vert_{\mathcal{S}_{k, j}} \,\chi_j^{k} \vert_{\mathcal{S}_{k, j}}\Big)\Bigg) .
		\end{equation}
		A direct inspection now shows that Equation \eqref{eq:2bb} is equivalent to Equation \eqref{eq:4}.
	\end{proof}
	
	It therefore follows that in order to analyse the convergence of the sequence of error functions $\{e^n_j\}_{n \in \mathbb{N}}, ~~ j \in \{1, \ldots, N\}$ defined through Equation \eqref{eq:4}, we must {study} the convergence of the error vectors $\{\bold{e}^n\}_{n \in \mathbb{N}}$ defined through the matrix equation \eqref{eq:5}. More specifically, we must analyse the structure of $T$. This is the subject of the next Subsection \ref{sec:conv}.\\
	
	We conclude this section by returning to the geometric constraint A3) that we imposed on the subdomains $\Omega_i, ~i=1, \ldots, N$ in Section \ref{sec:2.1.1}. The following remark demonstrates the necessity of this constraint.

	\subsection{Convergence analysis}\label{sec:conv}

	\subsubsection{Technical lemmas}\label{sec:Technical}
	The goal of this section is to introduce some technical results that we use to prove our main theorems. We begin with the following first result.
	
	\begin{lemma}\label{lem:1}
		Let $j \in \{1, \ldots, N\}$ and consider the subdomain $\Omega_j$, let $\widetilde{u} \in L^{\infty}(\partial \Omega_j)$ be a non-negative, piecewise continuous function such that  $\text{\rm ess}\sup_{\partial \Omega_j} \widetilde{u} \leq 1$.
		Assume that there exists at least one neighbouring index $k \in N_j$ of the subdomain $\Omega_j$ such that $\widetilde{u}< 1 ~\text{\rm in }\mathcal{S}^{\rm{int}}_{k, j}$,
		and define the function $u^{\rm{new}} \in L^{\infty}(\mathcal{S}_j)$ as
		$u^{\rm{new}}:= \mathcal{R}_j \left(\mathcal{E}_j (\widetilde{u})\right)$.
		Then it holds that
		\begin{equation*}
		0 \leq \text{\rm ess}\sup_{\mathcal{S}_j} u^{\rm{new}} \leq 1, \quad
		u^{\rm{new}}(x) < 1 ~\forall x\in\mathcal{S}_j^{{\Int}}.
		\end{equation*}
	\end{lemma}
	
	An example of the application of Lemma \ref{lem:1} is given in Figure \ref{fig:5}.
	\begin{figure}
		\centering
		\definecolor{zzzzzz}{rgb}{0.6,0.6,0.6}
		\begin{tikzpicture}[scale=0.5]
		\draw [shift={(23.85,11.26)},line width=2.8pt,color=zzzzzz]  plot[domain=1.15:1.98,variable=\t]({1*1.56*cos(\t r)+0*1.56*sin(\t r)},{0*1.56*cos(\t r)+1*1.56*sin(\t r)});
		\draw [dash pattern=on 5pt off 5pt] (26.02,12.5) circle (1.56cm);
		\draw [dash pattern=on 5pt off 5pt] (26,10) circle (1.56cm);
		\draw [dash pattern=on 5pt off 5pt] (23.82,8.77) circle (1.56cm);
		\draw [dash pattern=on 5pt off 5pt] (21.67,10.03) circle (1.56cm);
		\draw [dash pattern=on 5pt off 5pt] (21.69,12.53) circle (1.56cm);
		\draw [dash pattern=on 5pt off 5pt] (23.86,13.77) circle (1.56cm);
		\draw(23.84,11.26) circle (1.56cm);
		\draw[black] (23.8,11.2) node {$\partial \Omega_j$};
		\draw[black] (23.8,13.9) node {$\Omega_k$};
		\draw[black] (29.25,11.7) node {$\mathcal{R}_j\circ\mathcal{E}_j$};
		\draw[black] (33.0,11.5) node {$\mathcal{S}_j$};
		\draw [<-] (30.4,11.4)-- (28.2,11.4);
		\draw [dash pattern=on 5pt off 5pt] (33.02,11.44) circle (1.56cm);
		\draw [shift={(30.87,12.71)},line width=2.8pt,color=zzzzzz]  plot[domain=-0.95:0.10,variable=\t]({1*1.56*cos(\t r)+0*1.56*sin(\t r)},{0*1.56*cos(\t r)+1*1.56*sin(\t r)});
		\draw [shift={(30.85,10.21)},line width=2.8pt,color=zzzzzz]  plot[domain=0.10:0.93,variable=\t]({1*1.56*cos(\t r)+0*1.56*sin(\t r)},{0*1.56*cos(\t r)+1*1.56*sin(\t r)});
		\draw [shift={(33,8.95)},line width=2.8pt,color=zzzzzz]  plot[domain=1.15:1.98,variable=\t]({1*1.56*cos(\t r)+0*1.56*sin(\t r)},{0*1.56*cos(\t r)+1*1.56*sin(\t r)});
		\draw [shift={(35.18,10.18)},line width=2.8pt,color=zzzzzz]  plot[domain=2.16:3.02,variable=\t]({1*1.56*cos(\t r)+0*1.56*sin(\t r)},{0*1.56*cos(\t r)+1*1.56*sin(\t r)});
		\draw [shift={(35.2,12.68)},line width=2.8pt,color=zzzzzz]  plot[domain=3.02:4.07,variable=\t]({1*1.56*cos(\t r)+0*1.56*sin(\t r)},{0*1.56*cos(\t r)+1*1.56*sin(\t r)});
		\draw [shift={(33.04,13.95)},line width=2.8pt,color=zzzzzz]  plot[domain=4.29:5.10,variable=\t]({1*1.56*cos(\t r)+0*1.56*sin(\t r)},{0*1.56*cos(\t r)+1*1.56*sin(\t r)});
		
		\draw [shift={(30.87,12.71)},color=black]  plot[domain=-1.18:0.11,variable=\t]({1*1.56*cos(\t r)+0*1.56*sin(\t r)},{0*1.56*cos(\t r)+1*1.56*sin(\t r)});
		\draw [shift={(30.85,10.21)},color=black]  plot[domain=-0.13:1.16,variable=\t]({1*1.56*cos(\t r)+0*1.56*sin(\t r)},{0*1.56*cos(\t r)+1*1.56*sin(\t r)});
		\draw [shift={(33,8.95)},color=black]  plot[domain=0.92:2.21,variable=\t]({1*1.56*cos(\t r)+0*1.56*sin(\t r)},{0*1.56*cos(\t r)+1*1.56*sin(\t r)});
		\draw [shift={(35.18,10.18)},color=black]  plot[domain=1.97:3.25,variable=\t]({1*1.56*cos(\t r)+0*1.56*sin(\t r)},{0*1.56*cos(\t r)+1*1.56*sin(\t r)});
		\draw [shift={(35.2,12.68)},color=black]  plot[domain=3.02:4.3,variable=\t]({1*1.56*cos(\t r)+0*1.56*sin(\t r)},{0*1.56*cos(\t r)+1*1.56*sin(\t r)});
		\draw [shift={(33.04,13.95)},color=black]  plot[domain=4.06:5.34,variable=\t]({1*1.56*cos(\t r)+0*1.56*sin(\t r)},{0*1.56*cos(\t r)+1*1.56*sin(\t r)});
		
		\fill [color=black] (22.29,11.46) circle (2.4pt);
		\fill [color=black] (22.29,11.09) circle (2.4pt);
		\fill [color=black] (22.61,11.28) circle (2.4pt);
		\fill [color=black] (22.89,10.02) circle (2.4pt);
		\fill [color=black] (23.22,10.21) circle (2.4pt);
		\fill [color=black] (23.22,9.83) circle (2.4pt);
		\fill [color=black] (24.45,9.82) circle (2.4pt);
		\fill [color=black] (24.45,10.2) circle (2.4pt);
		\fill [color=black] (24.77,10.01) circle (2.4pt);
		\fill [color=black] (25.39,11.07) circle (2.4pt);
		\fill [color=black] (25.08,11.26) circle (2.4pt);
		\fill [color=black] (25.39,11.44) circle (2.4pt);
		\fill [color=black] (24.78,12.5) circle (2.4pt);
		\fill [color=black] (24.47,12.69) circle (2.4pt);
		\fill [color=black] (24.47,12.33) circle (2.4pt);
		\fill [color=black] (23.24,12.7) circle (2.4pt);
		\fill [color=black] (23.24,12.34) circle (2.4pt);
		\fill [color=black] (22.92,12.52) circle (2.4pt);
		\fill [color=black] (31.47,11.64) circle (2.4pt);
		\fill [color=black] (31.47,11.27) circle (2.4pt);
		\fill [color=black] (31.79,11.46) circle (2.4pt);
		\fill [color=black] (32.07,10.2) circle (2.4pt);
		\fill [color=black] (32.4,10.39) circle (2.4pt);
		\fill [color=black] (32.4,10.01) circle (2.4pt);
		\fill [color=black] (33.63,10) circle (2.4pt);
		\fill [color=black] (33.63,10.38) circle (2.4pt);
		\fill [color=black] (33.95,10.19) circle (2.4pt);
		\fill [color=black] (34.57,11.25) circle (2.4pt);
		\fill [color=black] (34.26,11.44) circle (2.4pt);
		\fill [color=black] (34.57,11.62) circle (2.4pt);
		\fill [color=black] (33.96,12.68) circle (2.4pt);
		\fill [color=black] (33.65,12.87) circle (2.4pt);
		\fill [color=black] (33.65,12.51) circle (2.4pt);
		\fill [color=black] (32.42,12.88) circle (2.4pt);
		\fill [color=black] (32.42,12.52) circle (2.4pt);
		\fill [color=black] (32.1,12.7) circle (2.4pt);
		\end{tikzpicture}
		\caption{Example of the application of Lemma \ref{lem:1}.
			A function $v: \partial \Omega_j \rightarrow \mathbb{R}$ such that
			$v(x) \in [0,1]$ for all $x \in \partial \Omega_j$ with $v(x)<1$
			for $x \in \Gamma_j^{k,0}$ for some $k$ (thick gray arc in the left picture)
			is mapped by the operator $\mathcal{R}_j\circ\mathcal{E}_j$
			on to a function $w:=\mathcal{R}_j(\mathcal{E}_j(v))$
			defined on $\mathcal{S}_j$ and satisfying $w(x) \in [0,1]$ for all
			$x \in \mathcal{S}_j$, $w(x) < 1$ for all $x \in \mathcal{S}_j^{\Int}$.}
		\label{fig:5}
	\end{figure}
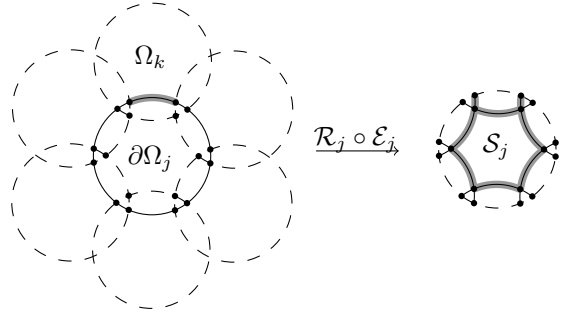
	
	\begin{proof}
		The proof follows from the maximum principle for harmonic functions (see, e.g., \cite{CiaramellaGander2}) combined with the definition of the partition of unity functions introduced above.
	\end{proof}

	Next, we {prove a fundamental} lemma concerning the norm of the iteration operator $T$.
	
	\begin{lemma}\label{lem:Hassan}
		Let $\bold{1} \in \Pi_{j=1}^N C^0\left(\mathcal{S}_j\right)$ denote an $N$-dimensional vector with the property that every $j^{\text{th}}$ entry of the vector is a function identically equal to 1 on the skeleton $\mathcal{S}_j$. Then for any natural number $n \in \mathbb{N}$ it holds that
		\begin{equation*}
		\Vert T^n\Vert_{\rm{OP}, \infty}:= \sup_{\substack{\bold{v} \in \Pi_{j=1}^N C^0(\mathcal{S}^{\rm int}_j)\\ \Vert \bold{v}\Vert_{\infty}=1}}\Vert T^n \bold{v}\Vert_{\infty} = \Vert T^n \bold{1}\Vert_{\infty}.
		\end{equation*}
	\end{lemma}
	\begin{proof}
		
		We first prove an intermediate result. Let $\bold{v}, \bold{u} \in \Pi_{j=1}^N C^0(\mathcal{S}^{\rm int}_j)$ be {two} functions such that 
		\begin{equation*}
		{\bold{v} \leq \bold{u}}  ~\text{ on } \cup_{j=1}^N\mathcal{S}_j.
		\end{equation*}
		
		We prove that
		\begin{equation}\label{eq:hassan}
		(T \bold{v}) \leq (T \bold{u}) \qquad \text{on } \cup_{j=1}^N \mathcal{S}_j.
		\end{equation}
		
		Define ${\bold{w}}:= T (\bold{u}-\bold{v})$. Then by definition of $T$ we obtain that for each $j \in \{1, \ldots, N\}$ it holds that
		\begin{align*}
		\bold{w}_{j}= \sum_{i \in N_j} P_{ji} (\bold{u}_i -\bold{v}_i)&=\sum_{i \in N_j}\mathcal{R}_j\Bigg(\mathcal{E}_j\Big( \big(\bold{u}_i-\bold{v}_i\big)\vert_{\mathcal{S}_{j, i}} \, \chi_j^{i}\vert_{\mathcal{S}_{j, i}} \Big)\Bigg)\\
		&=\mathcal{R}_j\Bigg(\mathcal{E}_j\Big(\sum_{i \in N_j} \big(\bold{u}_i-\bold{v}_i\big)\vert_{\mathcal{S}_{j, i}}\,\chi_j^{i} \vert_{\mathcal{S}_{j, i}}\Big)\Bigg).
		\end{align*}
		
		In other words, the function ${\bold{w}}_{j}$ is simply the restriction onto the skeleton $\mathcal{S}_j$ of the harmonic solution to a Dirichlet problem with boundary data from $\bold{u}-\bold{v}$. 
		
		Recall that we have by hypothesis that for each $j \in \{1, \ldots, N\}$ it holds that $\bold{u}_j-\bold{v}_j \geq 0$ on $\mathcal{S}_j$. It follows from the maximum principle that for each $j \in \{1, \ldots, N\}$ we have
		\begin{equation*}
		{\bold{w}}_j \geq 0 \qquad \text{on } \mathcal{S}_j.
		\end{equation*}
		Therefore, we obtain that
		\begin{equation*}
		{\bold{w}}= \left(T(\bold{u}-\bold{v})\right)\geq 0 \qquad \text{on }\cup_{j=1}^N \mathcal{S}_j.
		\end{equation*}
		This completes the proof for Equation \eqref{eq:hassan}.
		
		We now proceed to the proof of Lemma \ref{lem:Hassan}. 
		Note that it suffices to prove that for all natural numbers $n \in \mathbb{N}$ and all functions $\bold{v} \in \Pi_{j=1}^N C^0(\mathcal{S}^{\rm int}_j)$ such that $\Vert \bold{v} \Vert_{\infty}=1$ it holds that
		\begin{equation}\label{eq:induction1}
		(T^n \bold{v}) \leq (T^n \bold{1}) \qquad \text{on } \cup_{j=1}^N \mathcal{S}_j.
		\end{equation}
		
		The proof follows easily by induction from Equation \eqref{eq:hassan}. Indeed, the base case follows immediately by picking $\bold{u}=\bold{1}$. Next, assume that \eqref{eq:induction1} holds for some $k \in \mathbb{N}$. We define functions $\widetilde{\bold{v}}:=T^k \bold{v}$ and $\widetilde{\bold{u}}:= T^k \bold{1}$ and recognize that the induction hypothesis implies that
		\begin{equation*}
		\widetilde{\bold{v}} \leq \widetilde{\bold{u}} \qquad \text{on } \cup_{j=1}^N \mathcal{S}_j.
		\end{equation*}
		
		Therefore, applying once again Equation \eqref{eq:hassan} to the functions $\widetilde{\bold{v}}$  and $\widetilde{\bold{u}}$ yields the required result. This completes the proof.
	\end{proof}
	
	We are now ready to state our second fundamental result. 
	\begin{lemma}\label{lem:Vr}
		Let $\Omega= \cup_{j=1}^N \Omega_j$ be a domain with $N_{\max}$ layers. Then for all natural numbers $n \leq N_{\max}$ it holds that
		$T^n \bold{1} \in \mathcal{V}_n$.
	\end{lemma}
	\begin{proof}
		The proof proceeds by induction.
		We consider the base case $n=1$ and show that $T \bold{1} \in \mathcal{V}_1$.
		To this end, let $\bold{v}= T \bold{1}$. We first consider a subdomain $\Omega_j$ such that the node $j$ is in $\mathcal{L}_1$. Clearly, we have that $\partial \Omega_j \cap \partial \Omega \neq \emptyset$. Therefore applying Lemma \ref{lem:1} yields that $\bold{v}_j(x) < 1$ for all $x \in \mathcal{S}_j^{\Int}$.
		Next, consider a subdomain $\Omega_j$ such that node $j$ is in $\mathcal{L}_2$ or higher. It follows from the definition of the layers that $\Omega_j$ does not correspond to a boundary node, and thus $\partial \Omega_j \cap \partial \Omega = \emptyset$. Using the definition of the iteration operator $T$, Equation \eqref{eq:new}, and Remark \ref{rem:1} we have that
		\begin{equation} \label{eq:calc}
		\begin{split}
		{\bold{v}}_{j}
		&= \sum_{i \in N_j} P_{ji} (\bold{1})_i
		=\sum_{i \in N_j}\mathcal{R}_j\Bigg(\mathcal{E}_j\Big((\bold{1})_i \vert_{\mathcal{S}_{j, i}}\, \chi_j^{i} \vert_{\mathcal{S}_{j, i}}\Big)\Bigg)\\
		&=\sum_{i \in N_j}\mathcal{R}_j\Bigg(\mathcal{E}_j\Big(\chi_j^{i} \vert_{\mathcal{S}_{j, i}}\Big)\Bigg)
		=\mathcal{R}_j\Bigg(\mathcal{E}_j\Big(\sum_{i \in N_j}\chi_j^{i} \vert_{\mathcal{S}_{j, i}}\Big)\Bigg)\\
		&=\mathcal{R}_j\big(\mathcal{E}_j (1)\big)=(\bold{1})_j.
		\end{split}
		\end{equation}
		We therefore obtain that ${\bold{v}}_j= (\bold{1})_j$. This completes the proof for the base case.
		
		Assume now that the result holds for some natural number $n=k < N_{\max}$, i.e., that $T^k \bold{1} \in \mathcal{V}_k$. We must prove that $T^{k+1} \bold{1} \in \mathcal{V}_{k+1}$.
		
		Let $\bold{v}:=T^{k}\bold{1}$, let $\bold{w}:= T^{k+1}\bold{1}$ and consider an arbitrary subdomain $\Omega_j$. If node $j$ is in $\mathcal{L}_1$, then Lemma \ref{lem:1} yields that $\bold{v}_j(x) < 1$ for all $x \in \mathcal{S}_j^{\Int}$. On the other hand, if node $j$ is in $\mathcal{L}_m$ where $m \in \{2, \ldots, k+1\}$, then it follows from the definition of the layers that there must exist a neighbouring node $i^\prime \in N_j$ such that node $i^\prime$ is in $\mathcal{L}_{m-1}$. The definition of the iteration operator $T$ therefore yields
		\begin{align*}
		{\bold{w}}_{j}=\sum_{i \in N_j} P_{ji} \bold{v}_i&= \mathcal{R}_j\Bigg(\mathcal{E}_j\Big(\bold{v}_{i^\prime}\vert_{\mathcal{S}_{j, i^\prime}}\, \chi_j^{i^\prime} \vert_{\mathcal{S}_{j, i^\prime}}\Big)\Bigg)+ \sum_{\substack{i \in N_j\\ i \neq i^\prime}}\mathcal{R}_j\Bigg(\mathcal{E}_j\Big( \bold{v}_i\vert_{\mathcal{S}_{j, i}}\,\chi_j^{i} \vert_{\mathcal{S}_{j, i}}\Big)\Bigg)\\
		&=\mathcal{R}_j\Bigg(\mathcal{E}_j\Big(\bold{v}_{i^\prime}\vert_{\mathcal{S}_{j, i^\prime}}\,\chi_j^{i^\prime}\vert_{\mathcal{S}_{j, i^\prime}}+\sum_{\substack{i \in N_j\\ i \neq i^\prime}}\bold{v}_i\vert_{\mathcal{S}_{j, i}}\,\chi_j^{i}\vert_{\mathcal{S}_{j, i}} \Big)\Bigg).
		\end{align*}
		
		The induction hypothesis implies that $\bold{v}_{i^\prime}(x) < 1$ for all $x \in \mathcal{S}_{i^{\prime}}^{\Int}$. This implies in particular that $\bold{v}_{i^\prime}(x) < 1$ for all $x \in \mathcal{S}_{j, i^{\prime}}^{\Int}$. Since the function $\bold{w}_j$ is simply the restriction on to the skeleton $\mathcal{S}_j$ of the harmonic extension of boundary data from the function $\bold{v}$, we can apply Lemma \ref{lem:1} to obtain that $\bold{w}_j (x) < 1$ for all $x \in \mathcal{S}_j^{{\Int}}$.

		If $k+1 = N_{\max}$ then we are done. If not, then consider a subdomain $\Omega_j$ such that the node $j$ is in $\mathcal{L}_{\widetilde{m}}$ where $\widetilde{m} \in \{k+2, \ldots, N_{\max}\}$. We must show that $\bold{w}_j =1$ on $\mathcal{S}_j$. 
		
		It follows from the definition of the layers that all neighbouring nodes $\ell \in N_j$ must belong to $\mathcal{L}_{m^\prime}$ where $m^\prime \in \{k+1, \ldots, N_{\max}\}$. Therefore, the induction hypothesis implies that for each $\ell \in N_j$ it holds that $\bold{v}_\ell = 1$ on $\mathcal{S}_{j, \ell}$. 
		
		Using the definition of the iteration operator $T$ and Remark \ref{rem:1}, and proceeding in the same manner as in Equation \eqref{eq:calc}
		we obtain that  ${\bold{w}_j}=(\bold{1})_j$, and we have thus shown that $T^{k+1} \bold{1} = \bold{w} \in \mathcal{V}_{k+1}$. This completes the proof by induction.
	\end{proof}
	
	Our next goal is to obtain an analogous result for the set $\mathcal{C}_n$. To this end, we require the use of the following key lemma.
	
	\begin{lemma}\label{lem:2}
		Let $\Omega= \cup_{i=1}^N\Omega_i$ be a domain with $N_{\max}$ layers, let $\bold{v} \in \Pi_{j=1}^N C^0\big(\mathcal{S}^{\rm int}_j\big)$ be such that $\Vert \bold{v}_j \Vert_{\infty} \leq 1$, and let $\bold{w}=T\bold{v}$. Consider a subdomain $\Omega_j$ such that for all neighbouring indices $i \in N_j$ it holds that $\bold{v}_i(x) < 1$  for all $x \in \mathcal{S}_i^{\Int}$. Then it holds that
		\begin{align*}
		\Vert \bold{w}_j\Vert_{L^{\infty}(\mathcal{S}_j)}:=\text{\rm ess}\sup_{\mathcal{S}_j} \bold{w}_j < 1.
		\end{align*}
	\end{lemma}
	\begin{proof}
		By the definition of the iteration operator we have that
		\begin{align*}
		{\bold{w}}_{j}= \sum_{i \in N_j} P_{j i} \bold{v}_i=\sum_{i \in N_j}\mathcal{R}_j\Bigg(\mathcal{E}_j\Big(\bold{v}_i\vert_{\mathcal{S}_{j, i}}\,\chi_j^{i} \vert_{\mathcal{S}_{j, i}}\Big)\Bigg)=\mathcal{R}_j\Bigg(\mathcal{E}_j\Big(\sum_{i \in N_j}\bold{v}_i\vert_{\mathcal{S}_{j, i}}\, \chi_j^{i} \vert_{\mathcal{S}_{j, i}}\Big)\Bigg).
		\end{align*}
		Therefore, we define the function $h \colon \partial \Omega_j\rightarrow \mathbb{R}$ as 
		\begin{align*}
		h(x):= \begin{cases} \sum_{i \in N_j}\bold{v}_i\vert_{\mathcal{S}_{j, i}}(x)\, \chi_j^{i} \vert_{\mathcal{S}_{j, i}}(x) \qquad &\text{if } x \in \text{int}(\Gamma_j^{\text{int}}),\\
		0 \qquad &\text{otherwise}.
		\end{cases}
		\end{align*}
		
		Next, we observe that the definition of the partition of unity functions and the assumptions of Lemma~\ref{lem:2} imply two key properties of the function $h$. 
		\begin{enumerate}[(i)]
			\item It holds that
			\begin{align*}
			h(x)=0 \text{ for all } x \in \Gamma_j^{\Ext}, \qquad h(x)\leq 1 \text{ for all } x \in \Gamma_j^{\Int}.
			\end{align*}
			
			\item It holds that $h$ is a continuous function on $\text{int}\big(\Gamma_j^{\Int}\big)$. 
		\end{enumerate}
		\vspace{0.2cm}
		As a first step, we prove a third key property of the function $h$:
		\begin{enumerate}
			\item[(iii)] It holds that $h(x) < 1$ for all $x \in \text{int}\big(\Gamma_j^{\Int}\big)$.
		\end{enumerate}
		
		\vspace{2mm}
		To this end, let $x \in \text{int}(\Gamma_j^{\Int}) = \text{int}\Big(\overline{\cup_{k \in N_j} \cup_{i \in N_{jk}} \Gamma_{j}^{k, i}}\Big)$. We distinguish two cases:
		\begin{enumerate}
			
			\item $x \in \overline{\Gamma_j^{k, 0}}$ for some neighbouring index $k \in N_j$. We again have two cases:
			
			\begin{itemize} 
				\item $x \in \Gamma_j^{k, 0}$. Recalling the definition of the interior skeleton, we obtain also that $x \in \mathcal{S}_{k, j}^{\Int} \subset \mathcal{S}_k^{\Int}$. Thus it holds that $h(x)= \bold{v}_k(x) \chi_j^k(x)$. We recall that by the definition of the partition of unity functions, it holds that $\chi_{j}^k(x) =1$ for $x \in \Gamma_j^{k, 0}$.  Furthermore, we have by assumption that $\bold{v}_k(x) < 1 $ for $x \in \mathcal{S}_{k, j}^{\Int} \subset \mathcal{S}_k^{\Int}$. We therefore conclude that $h(x) < 1$.
				
				\item $x \notin \Gamma_j^{k, 0}$, i.e., $x$ is a boundary point of the closed set $\overline{\Gamma_j^{k, 0}}$. Now, either $x \in {\Gamma_j^{\Ext}}$ or there exists some neighbouring index $\ell \in N_{jk}$ such that $x \in \overline{\Gamma_j^{k, \ell}}$. In other words, there are exactly two possibilities: either $x$ is a boundary point of the exterior boundary of $\Omega_j$ or $x$ is a boundary point of some triple intersection. Since $x \in \text{int}(\Gamma_j^{\Int})$, we have excluded the first case. The second case $x \in \overline{\Gamma_j^{k, \ell}}$ is covered below. 
			\end{itemize}
			\item $x \in \overline{\Gamma_j^{k, \ell}}$ for some neighbouring indices $k, \ell \in N_j$. Recalling once again the definition of the skeletons, we obtain that $x \in \mathcal{S}_{k, j} \subset \mathcal{S}_{k}$ and $x \in \mathcal{S}_{\ell, j} \subset \mathcal{S}_{\ell}$. It therefore follows that
			\begin{align}\label{eq:newnew}
			h(x)=\bold{v}_k(x)\chi_j^{k}(x) + \bold{v}_{\ell}(x)\chi_j^{\ell}(x).
			\end{align}
			We again have two cases:
			\begin{itemize}
				\item $x \in \mathcal{S}_{k}^{\Int}$. We have by assumption that $\bold{v}_k(x) < 1 $ for $x \in \mathcal{S}_k^{\Int}$. On the other hand, we also know that $\bold{v}_{\ell}(x) \leq 1 $ for $x \in \mathcal{S}_{\ell}$. Therefore, Equation \eqref{eq:newnew} implies that $h(x) < 1$.
				\item $x \notin \mathcal{S}_{k}^{\Int}$. Since $x \in \mathcal{S}_{k}$, we must have that $x \in 
				\partial \Omega_k$. It is readily seen that this in turn implies that $x \in \mathcal{S}_{\ell}^{\Int}$. Therefore, we obtain that $\bold{v}_{\ell}(x) < 1$, and $\bold{v}_k(x) \leq 1 $. Hence, Equation \eqref{eq:newnew} again implies that $h(x) < 1$.
			\end{itemize}
		\end{enumerate}
		
		We conclude that $h(x) < 1$ for all $x \in \text{int}(\Gamma_j^{\Int})$ and therefore Property (iii) of the function $h$ also holds.\\

		Consider now the skeleton $\mathcal{S}_j$ and let $\mathcal{S}_{j, k}, ~ k \in N_j$ be an arbitrary arc of the skeleton. We must show that $\text{ess}\sup_{\mathcal{S}_{j, k}} \bold{w}_j < 1$. This is a slightly delicate argument since the function $\bold{w}_j$ need not be continuous on the closed set $\mathcal{S}_{j, k}$ due to possible jump discontinuities at the endpoints. We therefore proceed in two steps:
		
		\begin{enumerate}
			\item First we show that $\bold{w}_j(x) < 1$ for all $x \in \mathcal{S}_{j, k}^{\Int}$, i.e., that the function $\bold{w}_j$ is strictly smaller than one in the \emph{interior} of the skeleton arc $\mathcal{S}_{j, k}$. Property (iii) of the function $h$ yields that $h(x) < 1 ~ \forall x \in \Gamma_j^{\Int} \supset \mathcal{S}_{k, j}$. Therefore Lemma \ref{lem:1} yields that $\bold{w}_j(x) < 1$ for all $x \in \mathcal{S}_j^{\Int} \supset \mathcal{S}_{j, k}^{\Int}$.
			
			\item Next we show that $\lim_{\substack{x \in \mathcal{S}_{j, k} \\ x \to \partial \Omega_j} }\bold{w}_j(x) < 1$. In other words we must show that the limit of the function $\bold{w}_j$ along the skeleton arc $\mathcal{S}_{j, k}$, as one approaches the endpoints \emph{is strictly smaller than one}. We emphasise that this step is necessary since the function $\bold{w}_j$, a priori, may contain a jump discontinuity at the endpoints of the skeleton \footnote{This situation can arise precisely when the subdomains weakly overlap (see \cite{Lions2})}. 
			
			To this end, let $\hat{x} \in \partial \Omega_j$ denote any endpoint of the skeleton arc $\mathcal{S}_{j, k}$. Once again we have two cases:
			\begin{itemize}
				\item Suppose $\hat{x} \in \Gamma_j^{\Ext}$. Recall that $h=0$ on $\Gamma_j^{\Ext}$ from Property (i) and $h(x) < 1$ for all $x \in \text{int}(\Gamma_j^{\Int})$ from Property (iii). Thus, the Schwarz lemma (see, e.g., \cite[Pages 632-635]{Krylov}, \cite{CiaramellaGander2} and \cite[Section 3]{Lions2}) implies that 
				\begin{align*}
				\lim_{\substack{x \in \mathcal{S}_{j, k} \\ x \to \hat{x}} }\bold{w}_j(x) = \alpha< 1,
				\end{align*}
				where $\alpha$ is a constant that depends on the angle at which the skeleton arc $ \mathcal{S}_{j, k}$ intersects the boundary $\partial \Omega_j$ at the point $\hat{x}$. 
				
				\item Suppose $\hat{x} \notin \Gamma_j^{\Ext}$. Then $\hat{x} \in \text{int}\big(\Gamma_j^{\Int}\big)$. Property (ii) of the function $h$ implies that $h$ is a continuous function on $\text{int}\big(\Gamma_j^{\Int}\big)$. This in turn implies that the harmonic extension $\mathcal{E}_j(h)$ is continuous in a neighbourhood of the point $\hat{x}$. This yields in particular that
				\begin{align*}
				\lim_{\substack{x \in \mathcal{S}_{j, k} \\ x \to \hat{x}} }\bold{w}_j(x) = \bold{w}_j(\hat{x})= h(\hat{x}) <1.
				\end{align*}
			\end{itemize}
			
			Hence, the claim holds in both cases.
		\end{enumerate}
		
		It therefore follows that $\text{ess}\sup_{\mathcal{S}_{j, k}} \bold{w}_j< 1$. Since the skeleton arc $\mathcal{S}_{j, k}$ was arbitrary, we obtain that
		\begin{align*}
		\text{ess}\sup_{\mathcal{S}_{j}} \bold{w}_j < 1.
		\end{align*}
	\end{proof}
	
	Consider the setting of Lemma \ref{lem:2}. The careful reader will observe that the proof of Lemma \ref{lem:2} required the use of all three key properties of the boundary data $h$. In particular, we explicitly used the fact that $h$ is a continuous function on the interior boundary $\Gamma_j^{\text{int}}$. The continuity of the function $h$ is itself a consequence of our earlier assumption that the partition of unity functions are continuous on the interior boundary $\Gamma_j^{\text{int}}$. One might therefore wonder if the proof of Lemma \ref{lem:2} still holds if the continuity of the partition of unity functions is not imposed. It turns out that this continuity assumption is not necessary to prove Lemma \ref{lem:2}, and we may instead use the Schwarz lemma. However, as we demonstrate in Section \ref{sec:num} using some numerical examples, a choice of discontinuous partition of unity functions may lead to a quantitatively slightly worse contraction of the error at each iteration. We remark that the ddCOSMO implementation uses continuous partition of unity functions in practice.
	
	We conclude this subsection by observing that Lemma \ref{lem:2} has the following consequence.
	
	\begin{lemma}\label{lem:3}
		Let $\Omega= \cup_{j=1}^N \Omega_j$ be a domain with $N_{\max}$ layers, let $n \leq N_{{\max}}-2$ be a natural number, and let $\bold{u} \in \mathcal{V}_{n}$. Then it holds that
		\begin{equation}\label{eq:Aachen1}
		T^2 \bold{u} \in \mathcal{C}_{n}.
		\end{equation}
	\end{lemma}
	\begin{proof}
		Let $\bold{v}=T\bold{u}$ and let $\bold{w}=T^2\bold{u}$. Since $\bold{u} \in \mathcal{V}_n$ and $\Vert T \Vert_{\rm{OP}, \infty} \leq 1$, we obtain that $\Vert \bold{w}\Vert_{\infty} \leq \Vert \bold{u}\Vert_{\infty} \leq 1$. Therefore, we need only show that $\text{ess}\sup_{\mathcal{S}_j}\bold{w}_j < 1$ for all $j \in \cup_{k=1}^{n}\mathcal{L}_k$.
		
		Let $j \in \cup_{k=1}^{n}\mathcal{L}_k$ and let $\ell \in N_j$ be any neighbouring index. By the definition of the layers, we know that $\ell \in \cup_{k=1}^{n+1}\mathcal{L}_k$. On the other hand, since $\bold{v} \in \mathcal{V}_{n+1}$, we know that $\bold{v}_{\ell}(x) < 1 $ for all $x \in \mathcal{S}_{\ell}^{\Int}$. Applying Lemma \ref{lem:2} immediately yields that $\text{ess}\sup_{\mathcal{S}_j}\bold{w}_j < 1$. Since $j \in \cup_{k=1}^{n}\mathcal{L}_k$ was arbitrary, we conclude that $\bold{w} \in C_n$.
	\end{proof}

	\subsubsection{Convergence results}\label{sec:ConvergenceResults}
	
	We are now ready to state our main results.
	
	\begin{theorem}\label{thm:1}
		Let $\Omega= \cup_{j=1}^N \Omega_j$ be a domain with $N_{\max}$ layers and let $\bold{u} \in \mathcal{V}_{N_{\max}}$. Then it holds that
		\begin{equation}\label{eq:Main1}
		T \bold{u} \in \mathcal{C}_{N_{\max}}.
		\end{equation}
	\end{theorem}
	\begin{proof}
		Let $\bold{w}=T\bold{u}$. We must show that for all $j \in \{1, \ldots, N\}$ it holds that $\text{ess}\sup_{\mathcal{S}_j}\bold{w}_j < 1$.
		
		To this end, let $j \in \{1, \ldots, N\}$. Since $\bold{u} \in \mathcal{V}_{N_{\max}}$ it follows that that for all neighbouring indices $i \in N_j$ it holds that $\bold{u}_i(x) < 1$ for all $x \in \mathcal{S}_i^{\Int}$. Applying Lemma \ref{lem:2} immediately yields that $\text{ess}\sup_{\mathcal{S}_j}\bold{w}_j < 1$. Since $j \in \{1, \ldots, N\}$ was arbitrary, this completes the proof.
	\end{proof}

	Theorem \ref{thm:1} has the following important consequence.
	\begin{corollary}\label{cor:1.1}
		Let $\Omega= \cup_{j=1}^N \Omega_j$ be a domain with $N_{\max}$ layers. Then it holds that
		\begin{equation*}
		\Vert T^{N_{\max}+1} \Vert_{\rm{OP}, \infty} < 1.
		\end{equation*}
	\end{corollary}
	\begin{proof}
		Lemma \ref{lem:Hassan} implies that $\Vert T^{N_{\max}+1} \Vert_{\rm{OP}, \infty}= \Vert T^{N_{\max}+1} \bold{1}\Vert_{\infty}$.
		
		Theorem \ref{thm:1} implies that $T^{N_{\max}+1} \bold{1} \in \mathcal{C}_{N_{\max}}$. By definition of the set $\mathcal{C}_{N_{\max}}$ we obtain that $\Vert T^{N_{\max}+1} \Vert_{\rm{OP}, \infty}= \Vert T^{N_{\max}+1} \bold{1}\Vert_{\infty} < 1$.
	\end{proof}
	
	\begin{remark}
		Consider the setting of Theorem \ref{thm:1}. The relation \eqref{eq:Main1} is sharp for $N_{\text{max}} >1$. See also Example \ref{ex:chain} for the case $N_{\max}=1$.
	\end{remark}
	
	\begin{remark}
		Consider the error equation \eqref{eq:5}. Corollary \ref{cor:1.1} implies that 
		\begin{equation*}
		\lim_{n \to \infty} \Vert\bold{e}^{n+1}\Vert_{\infty}= \lim_{n \to \infty} \Vert T \bold{e}^{n}\Vert_{\infty}= \lim_{n \to \infty} \Vert T^{n+1} \bold{e}^0 \Vert_{\infty}\leq \lim_{n \to \infty} \Vert T^{n+1}\Vert_{\rm{OP}, \infty} \Vert\bold{e}^0\Vert_{\infty}= 0.
		\end{equation*}
	\end{remark}
	
	Theorem \ref{thm:1} also allows us to prove that the global Laplace problem \eqref{eq:1} and the domain decomposition problem \eqref{eq:2} are indeed equivalent.
	\newpage
	\begin{theorem}\label{lemma:2.1}
		Equations \eqref{eq:1} and \eqref{eq:2} are equivalent. Therefore, the PSM converges to the solution of the global Laplace problem \eqref{eq:1}.
	\end{theorem}
	\begin{proof}
		It is well known that there exists a unique solution $u \in C^0(\Omega)$ to Equation \eqref{eq:1}. A direct calculation then shows that the restrictions of this function $u\vert_{\Omega_j}, ~j \in \{1, \ldots, N\}$ on each subdomain $\Omega_j$ also satisfy Equation \eqref{eq:2}. 
		Therefore, it suffices to show that Equation \eqref{eq:2} must have a unique solution. 
		
		We argue by contradiction. For each $j \in \{1, \ldots, N\}$, let $v_j, \widetilde{v}_j \in C^0(\Omega_j)$ be {two distinct} solutions to
		\eqref{eq:2} on the subdomain $\Omega_j$ and let $w_j:= v_j - \widetilde{v}_j$. It follows that
		\begin{align} \nonumber
		\Delta w_j &=0 \quad \hspace{2.2cm} \text{in }\Omega_j,\\
		w_j&=0 \quad \hspace{2.2cm}\text{on } \Gamma_j^{\text{ext}}, \label{eq:equivalence}\\ 
		w_j&=\sum_{i \in N_j}w_i\chi_j^{i} \quad \hspace{0.85cm}\text{on }\Gamma_j^{{\Int}}. \nonumber
		\end{align}
		
		Next, we define the function $\bold{w} \in \Pi_{j=1}^N C^0(\mathcal{S}^{\rm int}_j)$ by taking, for each $j=1, \ldots, N$, the restriction of the function $w_j$ on the skeleton $\mathcal{S}_j$. It follows from Equation \eqref{eq:equivalence} and Lemma \ref{lem:2.8} that $\bold{w}$ satisfies the fixed-point equation $\bold{w}= T \bold{w}$. This implies in particular that 
		\begin{align}\label{eq:fixed}
		\bold{w}= T^{N_{\max}+1} \bold{w},
		\end{align}
		where $N_{\max}$ is the number of layers in the domain $\Omega$. Equation \eqref{eq:fixed} now yields that $\Vert T^{N_{\max}+1} \Vert_{\text{OP}, \infty} \geq 1$, which contradicts Theorem \ref{thm:1}. Therefore, we must have that $\bold{w} \equiv0$ which implies that $v_j= \widetilde{v}_j$ for all $j=1, \ldots, N$, and thus Equation \eqref{eq:2} has a unique solution.
	\end{proof}
	
	\subsubsection{Discussion on convergence results and related examples} \label{sec:Examples}
	
	In this section, we discuss the convergence results obtained in Section \ref{sec:conv} and use two examples to explain the heuristic behind them and demonstrate how they can be considered an extension of existing results given in the literature \cite{CiaramellaGander2,CiaramellaGander}.
	In particular, in Example \ref{ex:globular} we show how our results are capable of precisely tracking the propagation of the contraction across the different layers comprising the domain $\Omega$ in the course of the iterations. In Example \ref{ex:chain}, we consider a problem defined on a linear chain of collinear subdomains and demonstrate how our results can be considered an extension of existing convergence results in the literature.
	
	\begin{example}\label{ex:globular}
		We first consider a domain $\Omega$ consisting of the union of 51 subdomains grouped in 4 layers as shown in Figure \ref{fig:ex_glob_1} (left). We will describe visually the results obtained in Section \ref{sec:conv} as they apply to this particular choice of domain.
		
		\begin{figure}
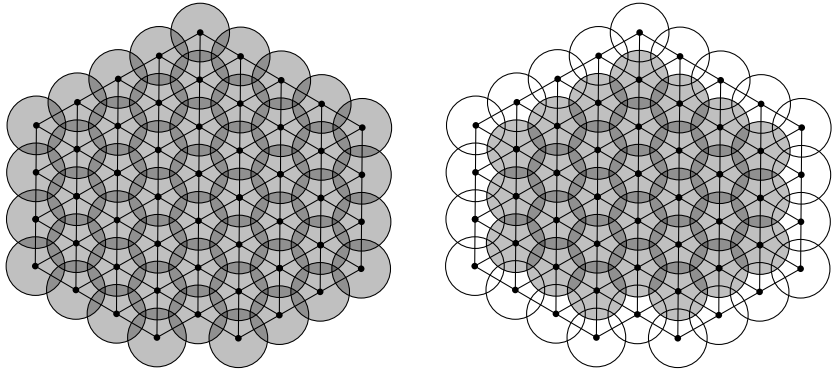

			\centering
			\input{fig1.tex}
			\input{fig2.tex}
			\caption{Left: Domain $\Omega$ and corresponding graph.
				Initial guess $\bold{u}^0=1$ (grey) on $\partial \Omega_j$ for $j=1,\dots,N$.
				Right: Result after the first iteration. The subdomains in white (layer 1)
				are the ones that began to experience a contraction, that is
				$\bold{u}^1 \in \mathcal{V}_1$.}
			\label{fig:ex_glob_1}
		\end{figure}
		
		To do so, we consider as initialization the function $\bold{u}^0 = \bold{1}$
		and follow the propagation of the contraction of the error through the course of the first five iterations of the PSM.
		We represent in grey the subdomains where the current approximation satisfies $\bold{u}^n_j=\bold{1}_j$.
		Moreover, we depict in white all the subdomains where the current approximation satisfies
		$\bold{u}^n_j(x)<1$ for all $x \in \mathcal{S}_j^{\Int}$
		and in red the subdomains such that $\text{ess}\sup_{\mathcal{S}_j}\bold{u}^n_j<1$.
		Since the initialisation is $\bold{u}^0 = \bold{1}$, at the iteration $0$ all the subdomains are grey;
		Figure \ref{fig:ex_glob_1} (left). After the first iteration, the current approximation
		is given by $\bold{u}^1 = T \bold{u}^0$. Thanks to the external Dirichlet boundary condition,
		the subdomains in the first layer $\mathcal{L}_1$ have begun to experience a contraction.
		Hence, $\bold{u}^1 \in \mathcal{V}_1$ and the first layer is then white, while all the other layers 
		are still grey; see Figure \ref{fig:ex_glob_1} (right).
		
		After the second iteration, the current approximation is $\bold{u}^2 = T \bold{u}^1$. The contraction propagates and thus the subdomains in layer 2 are now white while the ones in layers 3 and 4 are still grey. Hence, $\bold{u}^2 \in \mathcal{V}_2$ in agreement with Lemma \ref{lem:Vr}; see Figure \ref{fig:ex_glob_2} (left).
		\begin{figure}
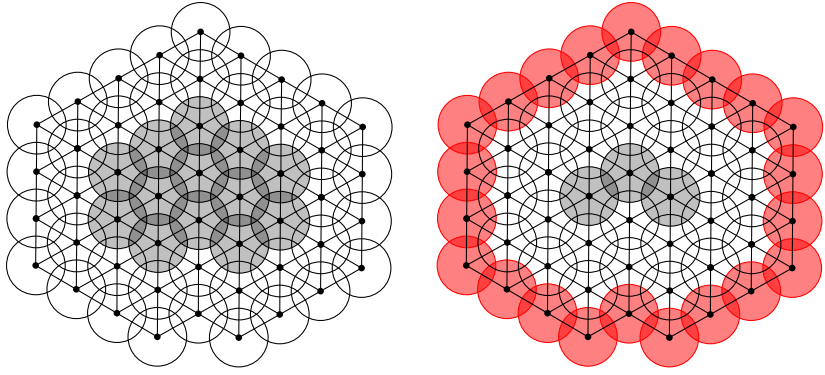

			\centering
			\input{fig3.tex}
			\input{fig4.tex}
			\caption{Left: Result after the second iteration. The subdomains in layer 2 have begun to experience a contraction and hence are represented in white so that $\bold{u}^2 \in \mathcal{V}_2$.
				Right: Result after the third iteration. The contraction has now propagated to the third layer, which is displayed in white so that $\bold{u}^3 \in \mathcal{V}_3$.
				Furthermore, the subdomains in Layer 1 are represented in red
				since they are contracting everywhere on the skeleton, i.e., $\text{ess}\sup_{\mathcal{S}_j}\bold{u}^3_j<1$ for all $j \in \mathcal{L}_1$.}
			\label{fig:ex_glob_2}
		\end{figure}
		
		The third iteration yields $\bold{u}^3 = T \bold{u}^2$ and reveals new behaviour. On the one hand, the contraction is still propagating towards the last layer. Layer 3 has thus begun to experience a contraction and is displayed in white indicating that $\bold{u}^3 \in \mathcal{V}_3$. On the other hand, since all the subdomains in Layer 1 were surrounded by white subdomains, Layer 1 is now displayed in red in accordance with Lemma \ref{lem:2}. Hence $\bold{u}^3 \in \mathcal{C}_1$.	Only the three subdomains in layer 4 are still grey; see Figure \ref{fig:ex_glob_2} (right).
		
		The fourth iteration yields $\bold{u}^4 = T \bold{u}^3$.
		The contraction has propagated to Layer 4, which is displayed in white, and $\bold{u}^4 \in \mathcal{V}_4$ in accordance with Lemma \ref{lem:Vr}.
		Furthermore, all the subdomains in Layer 2 which were previously surrounded by white or red subdomains at iteration 3 are now displayed in red in accordance with Lemma \ref{lem:2}. Thus, $\bold{u}^4 \in \mathcal{C}_2$; see Figure \ref{fig:ex_glob_3} (left).
		\begin{figure}
			\centering
			\input{fig5.tex}
			\input{fig6.tex}
			\caption{Left: Result after the fourth iteration.
				The contraction has propagated to the fourth layer, which is displayed in white. Since all subdomains have begun to experience a contraction, we have $\bold{u}^4 \in \mathcal{V}_4$. Moreover, the subdomains in Layer 1 and Layer 2 are contracting everywhere on the skeleton, i.e.,, $\text{ess}\sup_{\mathcal{S}_j}\bold{u}^4_j<1$ for all $j \in \mathcal{L}_1\cap\mathcal{L}_2$.
				Thus, Layers 1 and 2 are displayed in red.
				Right: Result after the fifth iteration. All the subdomains are
				contracting everywhere in their skeletons and are therefore displayed in red. It therefore holds that $\bold{u}^5 \in \mathcal{C}_4$ and after $N_{\max}+1=5$ iterations
				we finally observe a contraction in the infinity norm.}
			\label{fig:ex_glob_3}
		\end{figure}
		
		Finally, we consider the fifth iteration, which yields $\bold{u}^5 = T \bold{u}^4$.
		Since all the subdomains in Layers 3 and 4 were white at the previous iteration, they are now displayed in red (see Figure \ref{fig:ex_glob_3} (right)) indicating that $\bold{u}^5 \in \mathcal{C}_5$. In other words,  we observe a contraction in the infinity norm after precisely $N_{\max}+1=4+1=5$ iterations. This is in agreement with Theorem \ref{thm:1}.
	\end{example}
	
	\begin{example}\label{ex:chain}
		Next, we consider a domain $\Omega$ consisting of a linear chain of collinear subdomains (unit disks) as shown in Figure \ref{fig:linear}.
		\begin{figure}
			\centering
			\begin{tikzpicture}[scale=0.6, every node/.style={scale=0.85}]
			% Subdomains with overlap
			\draw[black] (-2.75,0.75) circle (1);
			\node at (-2.75,0.75) {$\Omega_1$};
			
			\draw[black,dashed] (-1.25,0.75) circle (1);
			\node at (-1.25,0.75) {$\cdots$};
			\draw[black,dashed] (0.25,0.75) circle (1);
			\node at (0.25,0.75) {$\cdots$};
			
			\draw[black] (1.75,0.75) circle (1);
			\node at (1.75,0.75) {$\Omega_j$};
			
			\draw[black,dashed] (3.25,0.75) circle (1);
			\node at (3.25,0.75) {$\cdots$};
			\draw[black,dashed] (4.75,0.75) circle (1);
			\node at (4.75,0.75) {$\cdots$};
			
			\draw[black] (6.25,0.75) circle (1);
			\node at (6.30,0.75) {$\Omega_N$};
			\end{tikzpicture}
			\caption{Example of a linear chain of $N$ collinear subdomains (unit disks).}
			\label{fig:linear}
		\end{figure}
		The Schwarz operator $T$ corresponding to this geometry is the block-tridiagonal operator
		\begin{equation*}
		T = \begin{bmatrix}
		0& P_{12}& & &  & \\
		P_{21}& 0& P_{23}& &  & \\
		& P_{32}& 0& P_{34}&  & \\
		& & \ddots & \ddots & \ddots &\\
		& & & P_{N-1N-2}& 0 & P_{N-1N}\\
		& & & & P_{NN-1} & 0 \\
		\end{bmatrix}.
		\end{equation*}
		Notice how this operator resembles the Schwarz iteration operator constructed in Fourier space in the literature \cite[Section 3.1]{CiaramellaGander4}; see also~\cite{CiaramellaGander}.
		In \cite{CiaramellaGander4} the Schwarz operator is denoted by $T_{1D}$ and it has been
		shown that $\| T_{1D} \|_{\infty}=1$, where $\| \cdot \|_{\infty}$ is the usual infinity matrix norm. This is in agreement with Lemmas \ref{lem:Hassan} and \ref{lem:Vr}. 
		
		It is clear from Figure \ref{ex:chain} that the considered chain is composed of a single layer since $\partial \Omega_j \cap \partial \Omega \neq \emptyset$ for all $j=1,\dots,N$. Lemma \ref{lem:Vr} implies that $T \bold{1} \in \mathcal{V}_1$ and Theorem \ref{thm:1} guarantees that $T^2 \bold{1} \in \mathcal{C}_1$, which implies that
		one observes a contraction in the infinity norm after two iterations. This result is conservative for chains of subdomains composed of only one layer. In fact, it has been proved in \cite{CiaramellaGander2}
		that the PSM for the solution of one-layer chains of subdomains contracts at each iteration;
		see also~\cite{CiaramellaGander3}.
		This result can be obtained also using our framework. Indeed, using Lemma \ref{lem:Hassan}, we can compute
		\begin{equation}\label{eq:chain1}
		\medmuskip=0.5mu
		\thinmuskip=0.5mu
		\thickmuskip=0.5mu
		\nulldelimiterspace=1.2pt
		\scriptspace=1.2pt    
		\arraycolsep1.2em
		\begin{split}
		&\Vert T \Vert_{\rm{OP}, \infty} 
		= \| T \bold{1} \|_{\infty} 
		= \max_{j = 1 , \dots , N} \text{ess }\sup_{\mathcal{S}_j} | (T \bold{1})_j| \\
		&=\max \Bigl\{ \, \text{ess }\sup_{\mathcal{S}_1} | (P_{12} \bold{1}_2) | \, , \, 
		\max_{j = 2 , \dots , N-1} \text{ess }\sup_{\mathcal{S}_j} | (P_{jj-1} \bold{1}_{j-1} + P_{jj+1} \bold{1}_{j+1}) |  \, , \, 
		\text{ess }\sup_{\mathcal{S}_N} | (P_{NN-1} \bold{1}_{N-1}) | \, \Bigr\}.
		\end{split}
		\end{equation}
		The definition of the operators $P_{j,k}$ implies that the functions
		\begin{equation*}
		\begin{split}
		v_1 &:= (P_{12} \bold{1}_{2}), \\
		v_j &:= (P_{jj-1} \bold{1}_{j-1} + P_{jj+1} \bold{1}_{j+1}), \; j=2,\dots,N-1,\\
		v_N &:= (P_{N-1N} \bold{1}_{N}), \\
		\end{split}
		\end{equation*}
		solve the problems
		\begin{equation*}
		\Delta v_j =0 \; \text{in $\Omega_j$}, \quad
		v_j =1 \; \text{on $\Gamma_j^{\Int}$}, \quad
		v_j =0 \; \text{on $\Gamma_j^{{\rm ext}}$},
		\end{equation*}
		for each $j=1,\dots,N$.
		Notice that \eqref{eq:chain1} is exactly given in \cite[Theorem 7, formula (24)]{CiaramellaGander2}.
		Following the same arguments as in \cite{CiaramellaGander2}, we notice that
		\begin{equation*}
		\text{ess}~\sup_{\mathcal{S}_1} v_1 = \text{ess }\sup_{\mathcal{S}_N} v_N \leq \text{ess }\sup_{\mathcal{S}_j} v_j < 1,
		\end{equation*}
		for any $j=1,\dots,N$, where the strict inequality follows by the maximum principle. 
		The value $\text{ess }\sup_{\mathcal{S}_j} v_j$ can then be computed explicitly as
		in \cite[Sections 4 and 5.1]{CiaramellaGander2}.
		We therefore obtain that $\Vert T \Vert_{\text{OP}, \infty} = \rho(\alpha) < 1$,
		where $\rho(\alpha)$ is exactly the estimate of the contraction factor given in \cite[Section 5.1]{CiaramellaGander2}.
		
		We also remark that it is possible in the same way to analyse other chains of fixed-sized subdomains characterized by only one layer such as the examples given in~\cite{CiaramellaGander2,CiaramellaGander3}.
	\end{example}
	
	\section{Extensions to more general settings}\label{sec:Extensions}
	
	Recall that we began Section \ref{sec:Geom} by imposing constraints A1) and A2) on the geometry of the domain $\Omega$. The preceding analysis confirms our claim that the constraint A1) is imposed purely to assist clarity of exposition, and can be dropped by adopting minor changes in, for instance, the definition of the index set $N_{jk}, j =1, \ldots, N, k \in N_j$.
	
	On the other hand, the constraint A2) which limits the number of simultaneously intersecting disks to three, significantly restricts the generality of our analysis. A second important limitation of our analysis is the fact that it is, a priori, restricted to disks in two dimensions and it is not yet clear if our results can be extended to domains in three dimensions consisting of the union of intersecting balls. The goal of this section is to address these shortcomings and extend the preceding analysis.
	
	\subsection{Extension to arbitrary number of intersections}\label{sec:Extensions2}
	We assume the geometric setting introduced in Section \ref{sec:Geom} and remove the constraints A1) and A2). In other words, we have assumed that the disks $\Omega_i, ~ i=1, \ldots, N$ may have any type and any number of simultaneous intersections. Our goal is now to outline step-by-step the changes that must be made to the analysis of Section \ref{sec:2} so that the convergence results still hold.

	\subsubsection{Notation}
	
	Let $j \in \{1, \ldots, N\}$. Given the disk $\Omega_j$ we define as before the sets
	\begin{equation*}
	\Gamma_j^\text{ext}:= \partial \Omega_j \cap \partial \Omega,
	\quad
	\Gamma_j^{\Int}:= \overline{\partial \Omega_j \setminus \Gamma_j^\text{ext}}.
	\end{equation*}
	
	The first complication arises when we attempt to define index sets characterising the number and types of intersections. In Section \ref{sec:Geom}, our task was straightforward because we assumed that the domain $\Omega$ consisted of at most triple intersecting subdomains but we have no longer imposed this constraint. In order to deal with the current more general setting, we define these index sets in an inductive manner.
	
	\begin{enumerate}
		\item[Step 1)] For each ${j_0} =1, \ldots, N$, we define the index set of neighbours ${N}_{j_0}$ as
		\begin{equation*}
		{N}_{j_0}:= \{j_1 \in \mathbb{N} \colon ( j_1\neq j_0) \land (\Omega_{j_1} \cap \Omega_{j_0} \neq \emptyset)\}.
		\end{equation*}
		
		\item[Step 2)] For each ${j_0} =1, \ldots, N$ and each $j_1 \in N_{j_0}$, we define the index set of triple intersections as
		\begin{align*}
		{N}_{j_0 j_1} &:= \Big\{j_2\in \{1, \ldots, N\} \colon\\
		& (j_2\neq j_1, {j_0}) \land (\Omega_{j_2}\cap \Omega_{j_1}  \cap \Omega_{j_0} \text{ is a set of non-zero measure}) \Big\}.
		\end{align*}
		If $N_{j_0j_1}$ is an empty set for every $ j_1 \in N_{j_0}$ then we stop the procedure for this specific choice of $j_0$, and we set the natural number $M_{j_0}=1$.
		
		\item[Step m)] For each $\big(j_0, j_1, \ldots, j_{m-1}\big) \in \{1, \ldots, N\} \times N_{j_0} \times N_{j_0j_1}\times \ldots \times N_{j_0j_1\ldots j_{m-2}}$ we define the index set of $m+1$-ple intersections{, i.e., the set of all intersections generated by $m+1$ simultaneously intersecting subdomains} as
		\begin{align*}
		{N}_{j_0 j_1\ldots j_{m-1}} &:= \Big\{j_m\in \{1, \ldots, N\} \colon \\
		&( j_m\neq j_{m-1}, \ldots, j_1, j_0) \land \Big(\bigcap\limits_{i=0}^{i=m} \Omega_{j_i} \text{ is a set of non-zero measure}\Big)\Big\}.
		\end{align*}
		
		If $N_{j_0 j_1\ldots j_{m-1}}$ is an empty set for every $ j_{m-1} \in N_{j_0j_1\ldots j_{m-2}}$, then we stop the procedure for this specific choice of $j_0$, and we set the natural number $M_{j_0}=m-1$.
	\end{enumerate}
	
	In addition, we define the constant $M=\max_{j_0} M_{j_0}$. Intuitively, the natural number $M_{j_0}+1$ denotes the maximum number of subdomains that simultaneously intersect the boundary of subdomain $\Omega_{j_0}$. Moreover, the natural number $M+1$ denotes the maximum number of simultaneously intersecting subdomains in the entire domain $\Omega$. We say that our domain consists of at most $M+1$-ple intersections, i.e., the intersections in our domain are generated by at most $M+1$ simultaneously intersecting subdomains.
	
	\vspace{4mm}
	The next complication arises now that we attempt to define decompositions of the interior boundary $\Gamma_{j_0}^{\Int}, ~ {j_0}=1, \ldots, N$. Let $j_0\in \{1, \ldots, N\}$. Then we define for every natural number $n \in \{1, \ldots, M_{j_0} \}$ the set $X^n_{j_0}$ of $M-$tuples as
	\begin{align*}
	X^n_{j_0}= \Big\{\alpha \in \mathbb{N}^M \colon \alpha \in \big(\Pi_{i=0}^{n-1} N_{j_0\ldots j_i} \big)\times \{0\}^{M-n} \Big\}.
	\end{align*}
	
	Intuitively, the set $X^1_{j_0}$ characterises the simple intersections of subdomain $\Omega_{j_0}$, the set $X^2_{j_0}$ characterises the triple intersections of subdomain $\Omega_{j_0}$ and so on. In particular, the set $X^{M_{j_0}}_{j_0}$ characterises the intersections of the highest order, i.e., the $M_{j_0}+1$-ple intersections of subdomain $\Omega_{j_0}$. It is natural to define also for every ${j_0} \in \{1, \ldots, N\}$ the set $Z_{j_0}:= \bigcup_{i=1}^{M_{j_0}} X_{j_0}^{i}$.
	
	\vspace{4mm}
	Let $j \in \{1, \ldots, N\}$. For any multi-index $\alpha \in Z_j$, we define the natural number $M_{\alpha}$ as 
	\begin{align*}
	M_{\alpha}= \text{card}\left\{ \alpha_i \neq 0\right\}.
	\end{align*}
	
	In other words, $M_{\alpha}$ denotes the number of non-zero entries in a given multi-index $\alpha \in Z_j$. We can now define for all multi-indices $\alpha \in Z_{j}$ the set $\Gamma_{j}^{\alpha} \subset \partial \Omega_j$ as
	\begin{align*}
	\Gamma_{j}^{\alpha}:=
	\text{int} \Big\{ x \in \partial \Omega_j \colon \big(x \in \Omega_{\alpha_j} \forall j=1, \ldots, M_{\alpha}\big) \land \big(x \notin \Omega_{\ell} \forall \ell \in N_{j} \text{ with } \ell \notin \alpha\big)\Big\},
	\end{align*}
	
	where we use an abuse of notation to write $\ell \notin \alpha$, which simply means that $\ell \neq \alpha_i$ for all $i =1, \ldots, M$.
	
	These definitions now imply that 
	\begin{equation*}
	\Gamma_{j}^{{\Int}}= \overline{\bigcup\limits_{\alpha \in Z_{j}} \Gamma_{j}^{\alpha}}.
	\end{equation*}

	It is now straightforward to define also the skeletons associated with the subdomain $\Omega_j$. 
	Indeed, let $j \in \{1, \ldots, N\}$ and $k \in {N}_j$ be fixed. Then we define the sets $\mathcal{S}_{j, k}^{{\Int}}$ and $\mathcal{S}_{j, k}$ as
	\begin{equation*}
	\mathcal{S}_{j, k}:= \overline{\bigcup_{\substack{\alpha \in Z_{j}, \\ \alpha_1=j}}\Gamma_k^{\alpha}}, 
	\qquad
	\mathcal{S}^{\Int}_{j, k}:= \mathcal{S}_{j, k}\setminus \partial \Omega_{j}.
	\end{equation*}
	
	Similarly, we define the skeleton $\mathcal{S}_{j}$ and the interior skeleton $\mathcal{S}^{\text{int}}_{j}$ as
	\begin{equation*}
	\mathcal{S}_{j}:= \bigcup_{k \in {N}_{j}} \mathcal{S}_{j, k},
	\qquad
	\mathcal{S}_{j}^{{\Int}}:= \bigcup_{k \in {N}_{j}} \mathcal{S}_{j, k}^{{\Int}}.
	\end{equation*}
	
	It can be observed that with this notation
	
	\begin{itemize}
		\item The definition of the graph and layers of a domain remain unchanged.
		
		\item The definition of the sets $\mathcal{V}_n$ and $\mathcal{C}_n$ remains unchanged.
	\end{itemize}
	
	\vspace{4mm}
	Of course, it becomes necessary to modify the definition of the partition of unity functions. Indeed, let $j \in \{1, \ldots, N\}$. Then we define for each $k \in N_{j}$ a function $\chi_{j}^{k} \colon \partial \Omega_j \rightarrow \mathbb{R}$, continuous on $\text{int} \big(\Gamma_{j}^{\Int}\big)$, with the property that
	\begin{equation}
	\chi_{j}^{k}:= \begin{cases}
	1 \quad &\text{on } \overline{\Gamma_{j}^{\alpha}} \setminus \Gamma_{j}^{\Ext}~~~ \text{for } \alpha \in Z_{j} \text{ such that } \alpha= (k, 0, \ldots, 0),\\[0.5em]
	\in [0, 1] \quad &\text{on } \overline{{\Gamma_{j}^{\alpha}}} \setminus \Gamma_{j}^{\Ext}  ~~~\text{for } \alpha \in Z_{j} \text{ such that } (\alpha_1=k) \land  \big(\alpha \neq (k,0, \ldots, 0)\big),\\[0.5em]
	0 \quad &\text{otherwise},
	\end{cases}
	\end{equation}
	and such that
	\begin{equation*}
	\sum_{k \in {N}_{j}} \chi_{j}^{k}(x)=1 \quad \text{ for all } x \in \text{int}\big(\Gamma_{j}^{{\Int}}\big).
	\end{equation*}
	
	The harmonic extension operator $\mathcal{E}_{j}$ and the restriction operator $\mathcal{R}_{j}$ can now be defined analogously to Section \ref{sec:2}, keeping in mind the notation we have just introduced.

	\subsubsection{Operator formulation and convergence analysis}
	
	Let us first consider the developments of Section \ref{sec:2.2}. The definition of the error functions $\bold{e}^n, ~ n \in \mathbb{N}$ in Equation \ref{def:Error} remains unchanged. Indeed, each element $\bold{e}_{j}^n, ~ j =1 , \ldots, N, ~ n \in \mathbb{N}$ is now defined as
	\begin{align*}
	\bold{e}_{j}^n:= e^n_j \vert_{\mathcal{S}_j}.
	\end{align*}
	
	Finally, we consider the lemmas and theorems we have stated in Sections \ref{sec:conv} and \ref{sec:ConvergenceResults}. Once again, we observe that the majority of the technical lemmas still hold under the current setting. In fact only the proof of Lemma \ref{lem:2} needs to be modified but the changes are minor and occur at only one point.
	
	Indeed, consider the setting and proof of Lemma \ref{lem:2}. We must modify the proof of Property (iii) of the boundary function $h$: i.e., the proof that $h(x) < 1$ for all $x \in \text{int} \big(\Gamma_j^{\Int}\big)$.
	
	\vspace{3mm}
	Let $x \in \text{int}(\Gamma_j^{\Int})$. As before, we distinguish two cases.
	\begin{enumerate}
		
		\item Suppose $x \in \overline{\Gamma_{j}^{\alpha}}$ where $\alpha \in Z_{j}$ satisfies $\alpha_2=\alpha_3, \ldots, =\alpha_{M}=0$ and $\alpha_1 = k$ for some $k \in N_j$. In other words, we consider the case where $x$ belongs to the part of the boundary $\partial \Omega_j$ that is a simple intersection. We distinguish two cases.
		
		\begin{itemize} 
			\item $x \in \Gamma_j^{\alpha}$. Recalling the definition of the interior skeleton, we obtain also that $x \in \mathcal{S}_{k, j}^{\Int} \subset \mathcal{S}_k^{\Int}$. Thus it holds that $h(x)= \bold{v}_k(x) \chi_j^k(x)$. Using the fact that the partition of unity functions are all bounded by one, and the fact that we have by assumption $\bold{v}_k(x) < 1 $ for $x \in \mathcal{S}_{k, j}^{\Int} \subset \mathcal{S}_k^{\Int}$, we conclude that $h(x) < 1$.
			
			\item $x \notin \Gamma_j^{\alpha}$, i.e., $x$ is a boundary point of the closed set $\overline{\Gamma_j^{\alpha^{j}}}$. Now, either $x \in {\Gamma_j^{\Ext}}$ or there exists some neighbouring index $\ell \in N_{jk}$ such that $x \in \overline{\Gamma_j^{\beta}}$ where $\beta \in Z_{j}$ satisfies $\beta_1= k$ and $\beta_2=\ell$.  Since $x \in \text{int}(\Gamma_j^{\Int})$, we have excluded the first case. The second case is covered below. 
		\end{itemize}
		
		\item Suppose $x \in \overline{\Gamma_{j}^{\alpha}}$ where $\alpha \in Z_{j}$ satisfies $\alpha_1 = k, \alpha_2=\ell$ for some $k \in N_j$ and some $\ell \in N_{jk}$. In other words, we consider the case where $x$ belongs to the part of the boundary $\partial \Omega_j$ that is a triple intersection or higher.  It now holds that
		\begin{align}\label{eq:extension1}
		h(x)=\bold{v}_{k}\chi_j^{k}(x) + \bold{v}_{\ell}(x)\chi_j^{\ell}(x) + \sum_{i=3}^{M_{\alpha}} \bold{v}_{\alpha_i}(x)\chi_j^{\alpha_i}(x).
		\end{align}
		
		Notice now that $\text{for all } x \in \overline{\Gamma_{j}^{\alpha}}$ the partition of unity functions satisfy $\sum_{i=1}^{M_{\alpha}} \chi_j^{\alpha_i}(x)=1 $, and we have by the assumptions of Lemma \ref{lem:2} that $\bold{v}_{\alpha_i}(x) \leq 1$ for all $i=1, \ldots, M_{\alpha}$. In fact, even more is true. Since $x \in \overline{\Gamma_{j}^{\alpha^{j}}} \cap \text{int}\big(\Gamma_{j}^{{\Int}}\big)$ it is not difficult to see that either $x \in \Omega_k$ or $x \in \Omega_{\ell}$, i.e., $x$ belongs to the interior of one of the two neighbouring subdomains $\Omega_k$ and $\Omega_{\ell}$. This being the case, we obtain from Lemma \ref{lem:1} that either $\bold{v}_k(x) < 1$ or $\bold{v}_{\ell}(x) < 1$. Equation \eqref{eq:extension1} then implies that $h(x) < 1$. This completes the claim.
	\end{enumerate}
	
	We therefore conclude that $h(x) < 1$ for all $x \in \text{int}(\Gamma_j^{\Int})$ and therefore Property (iii) of the function $h$ holds also in the current more general setting. The remainder of the proof is identical.

	It can now be readily checked that the proofs of the remaining lemma and theorems in Section \ref{sec:conv} remain essentially unchanged. We conclude this sub-section by remarking that the reader can now understand why we chose to begin our analysis with certain simplifying assumptions: unfortunately, in a completely general setting, the notation turns increasingly complex and it becomes easy to lose focus on the main ideas of our analysis.

	\subsection{Extensions to three dimensions}\label{sec:Extensions3}
	
	We have assumed throughout this article that the domain $\Omega$ is a subset of $\mathbb{R}^2$ and is decomposed into subdomains $\{\Omega_i\}_{i=1}^N$ consisting of disks. It is pertinent to recall however, that the original ddCOSMO model involves solving the Laplace equation on a computational domain in $\mathbb{R}^3$, which is decomposed into subdomains $\{\Omega_i\}_{i=1}^N$ consisting of balls. We therefore briefly discuss the possibility of extending the preceding analysis to this three dimensional setting.
	
	Let us first assume that the domain $\Omega \subset \mathbb{R}^3$ can be decomposed as the union of intersecting balls $\Omega= \cup_{i=1}^N \Omega_i$. As in Section \ref{sec:Extensions2}, we assume that no two subdomains intersect at a single point. A careful study of Sections \ref{sec:2} and \ref{sec:Extensions2} now reveals that this three-dimensional geometric setting does not require any change in notation. In fact, the only lemma that uses the fact that each $\Omega_i, ~ i=1, \ldots, N \subset \mathbb{R}^2$ is Lemma \ref{lem:2}. 
	
	To be more concrete, assume the setting of Lemma \ref{lem:2}. The final step of the proof consists of showing that
	\begin{align*}
	\lim_{\substack{x \in \mathcal{S}_{j, k}\\ x \to \partial \Omega_j}} \bold{w}_j(x) < 1.
	\end{align*}
	
	In order to prove this result, we use the Schwarz lemma. % (see, e.g., \cite[Pages 632-635]{Krylov}, \cite{CiaramellaGander2}, and \cite{Lions2}). 
	This lemma is referenced by name in \cite[Pages 632-635]{Krylov}, and a detailed proof is provided in the case of two dimensions. Another proof of this lemma, which uses complex analysis and is therefore also valid only in two dimensions is given in \cite{CiaramellaGander2}. Additionally, the Schwarz lemma is also referenced- this time in higher dimensions- by P.L. Lions in his classical paper on domain decomposition in \cite{Lions2}[Section 3]. Lions claims that ``... this follows easily from potential theory'' but omits providing a proof.  Unfortunately, apart from Lion's single comment, we have been unable to find any proof for the Schwarz lemma in three dimensions. Moreover, the fact that the proof in \cite{CiaramellaGander2} uses tools from complex analysis prevents a straightforward extension to three dimensions.
	
	We therefore emphasise that the only bottleneck in extending the analysis of Sections \ref{sec:2} and \ref{sec:Extensions2} to the case of a domain $\Omega \subset \mathbb{R}^3$ composed of intersecting balls is the lack of a proof of the Schwarz lemma in three dimensions.

	\subsection{Extensions to other types of subdomains}
	The analysis we have presented thus far assumes that the domain $\Omega$ consists either of a union of intersecting disks in $\mathbb{R}^2$ or a union of intersecting balls in $\mathbb{R}^3$ (see Section \ref{sec:Extensions3}). In this subsection, we briefly discuss the possibility of extending our analysis to domains composed of more general types of subdomains. 
	
	Indeed, the results presented in this manuscript will hold for a different choice of subdomains if
	\begin{itemize}
		\item[(B1)] For this new choice of subdomains, it is possible to introduce definitions and notations analogous to the ones introduced in Section \ref{sec:2} such that the operator formulation of the Schwarz method as stated in Section \ref{sec:2.2} is well-posed.
		
		\item[(B2)] For this new choice of subdomains, the \emph{technical results} presented in Section \ref{sec:conv} still hold.
	\end{itemize}
	
	Clearly, the well-posedness of the operator formulation of the Schwarz method as given in Section \ref{sec:2.2} requires that the boundary value problem \eqref{eq:3} be well-posed on each subdomain $\Omega_j, ~j=1, \ldots, N$. This condition is satisfied if the subdomains $\{\Omega_j\}_{j=1}^N$ are, for instance, assumed to be \emph{Lipschitz}. Moreover, it is not difficult to see that the definitions and notations introduced in Section \ref{sec:2} involving, for instance, interior and exterior boundary sets, skeletons and partition of unity functions, can easily be extended to Lipschitz subdomains $\{\Omega_j\}_{j=1}^N$. Consequently, condition (B1) holds for the case of Lipschitz subdomains.
	
	Unfortunately, the \emph{analysis} we have presented in Section \ref{sec:conv} does not hold for general Lipschitz subdomains. In order to illustrate this point, let us consider the geometry displayed in Figure \ref{fig:forHassan2}.
	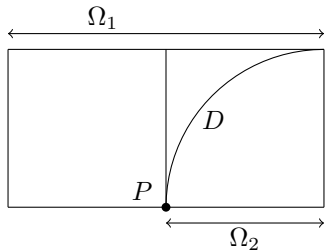
\begin{figure}[h]
		\centering
		\begin{tikzpicture}[scale=2.1]
		\draw (0.5,0.5)-- (0.5,1.5);
		\draw (0.5,1.5)-- (2.5,1.5);
		\draw (2.5,1.5)-- (2.5,0.5);
		\draw (2.5,0.5)-- (0.5,0.5);
		\draw (1.5,0.5)-- (1.5,1.5);
		\draw [domain=0:90] plot ({2.5-cos(\x)}, {0.5+sin(\x)});
		\draw[<->] (0.5,1.6)-- (2.5,1.6);
		\draw[<->] (1.5,0.4)-- (2.5,0.4);
		\draw (2.0,0.3) node {$\Omega_2$};
		\draw (1.1,1.7) node {$\Omega_1$};
		\fill [black] (1.5,0.5) circle (0.8pt);
		\draw (1.35,0.6) node {$P$};
		\draw (1.8,1.05) node {$D$};
		\end{tikzpicture}
		\caption{A domain $\Omega$ consisting of the union of two subdomains $\Omega_1$ (left) and $\Omega_2$ (right). Notice that although $\Omega_1$ and $\Omega_2$ are both Lipschitz, the overlap $\Omega_1 \cap \Omega_2$ is non-Lipschitz. Lemma \ref{lem:2} fails for this choice of subdomains.}
		\label{fig:forHassan2}
	\end{figure}

	We define the sets $\Omega_1 \subset \mathbb{R}^2$ and $\Omega_2\subset \mathbb{R}^2$ as
	\begin{align*}
	\Omega_1&:= (-1,0] \times (0,1) \cup \left\{ (x, y) \in \mathbb{R}^2 \colon x \in (0, 1), y\in (\sqrt{x}, 1)\right\},\\
	\Omega_2&:= (0, 1) \times (0, 1),
	\end{align*}
	and we define the domain $\Omega:=\Omega_1 \cup \Omega_2$. Additionally, we denote by \[D= \left\{ (x, y) \in \mathbb{R}^2 \colon x \in (0, 1), y=\sqrt{x}\right\},\] the curve that constitutes the right-side boundary of the domain $\Omega_1$.
	Let us now consider Lemma \ref{lem:2} for this choice of geometry and domain decomposition. Essentially, Lemma \ref{lem:2} claims the following: Let the function $h \colon \partial \Omega_2 \rightarrow \mathbb{R}$ be defined as
	\begin{align*}
	h(x)= \begin{cases}
	1 & \quad \text{if } x \in \text{int}\big(\Gamma_2^{\rm int}\big),\\
	0 & \quad \text{otherwise}.
	\end{cases}
	\end{align*}
	Let $w_2 \in L^2(\Omega_2)$ be the harmonic extension in $\Omega_1$ of $h$ and let the function $w_1 \in L^2(\Omega_1)$ be the solution to the BVP
	\begin{align*}
	-\Delta w_1 &= 0 \hspace{1cm}\text{in } \Omega_1,\\
	w_1 &= w_2\hspace{0.75cm} \text{on } \text{int } \Gamma_1^{\rm int},\\
	w_1 &= 0\hspace{1cm}\text{on } \partial \Omega_1 \setminus \text{int } \Gamma_1^{\rm int}.
	\end{align*}
	Then according to Lemma \ref{lem:2} we should have that
	\begin{align}\label{eq:Hassan4}
	\text{ess}\sup_{\Omega_1} w_1 < 1.
	\end{align}
	Consider now the proof of Lemma \ref{lem:2} and let $P\in \partial \Omega_1$ be the point $(0, 0)$, i.e., the lower endpoint of the curve $D$. The first case in Step (3) of the proof of Lemma \ref{lem:2} shows that 
	\[\lim_{\substack{x \in D \\ x \to P}} w_1(x)< 1.\]
	This bound on the limit is established using the Schwarz lemma (see, e.g., \cite[Pages 632-635]{Krylov}), which states that \[\lim_{\substack{x \in D \\ x \to P}} w_1(x)= \alpha,\] where $\alpha \in [0, 1]$ is a constant that depends on the angle at which the curve $D$ intersects the boundary $\Gamma_2^{\rm int}$ at the point $P$. 
	We now observe that the definition of the curve $D$ implies that it intersects $\Gamma_2^{\rm int}$ with angle \emph{zero} at the point $P$ which yields $\alpha=1$. Thus, the bound \eqref{eq:Hassan4} \underline{fails}, i.e., Lemma \ref{lem:2} no longer holds. We emphasise that this situation does not arise in the case of a domain $\Omega\subset \mathbb{R}^2$ that is composed of the union of intersecting disks since the boundaries of disks do not intersect at zero angles. 
	
	Note that the geometry displayed in Figure \ref{fig:forHassan2} is known in the domain decomposition literature as an example of a domain $\Omega=\Omega_1\cup \Omega_2$ wherein the standard proof of geometric convergence of the PSM via a contraction argument fails (see, e.g., \cite{Gabriele1}). Therefore, it is not surprising that our analysis also does not hold in this case. We remark that convergence of the PSM in this case can still be proven but requires different tools (see, e.g., \cite[Section 4]{Lions2}).
	
	We conclude that the analysis we have presented in this manuscript does not hold for general non-convex Lipschitz subdomains.
	
	\section{Numerical experiments}\label{sec:num}~\\
	The goal of this section is two-fold: First, we provide numerical evidence supporting our main convergence result Theorem \ref{thm:1}, and its analogue for the case of quadruple intersecting disks. Second, we show the effect of adopting a different choice of the partition of unity functions on the convergence rates.
	
	\begin{figure}[h]
		\begin{subfigure}{.32\textwidth}
			\centering
			\includegraphics[width=1.1\linewidth, trim={4cm 8cm 4cm 8cm}, clip=true]{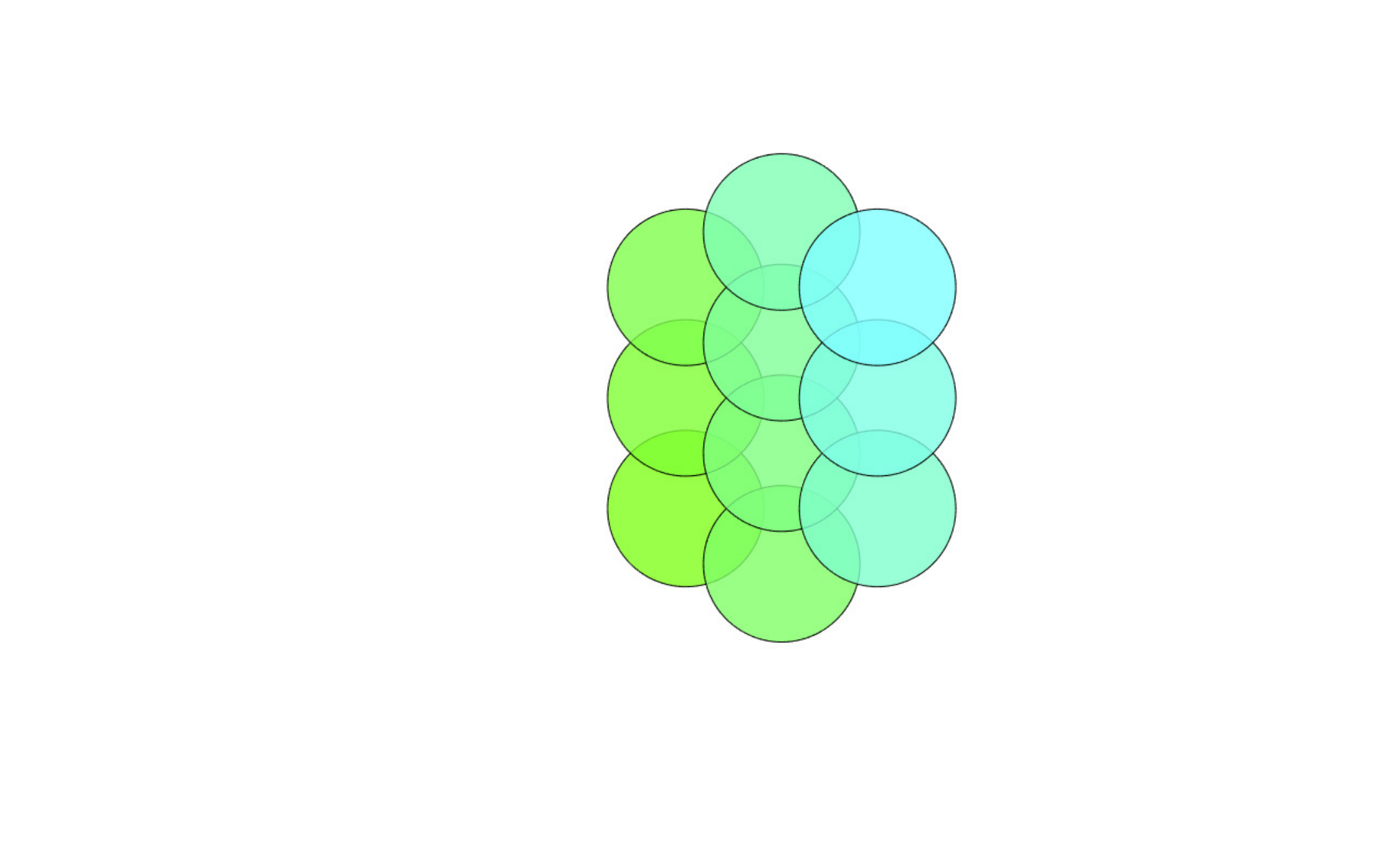}
			\vspace{0.2cm}
			\caption{$N_{\max}=2$ Layers.}
		\end{subfigure}%
		\begin{subfigure}{.32\textwidth}
			\centering
			\includegraphics[width=\linewidth, trim={3.5cm 7cm 3.5cm 7cm}, clip=true]{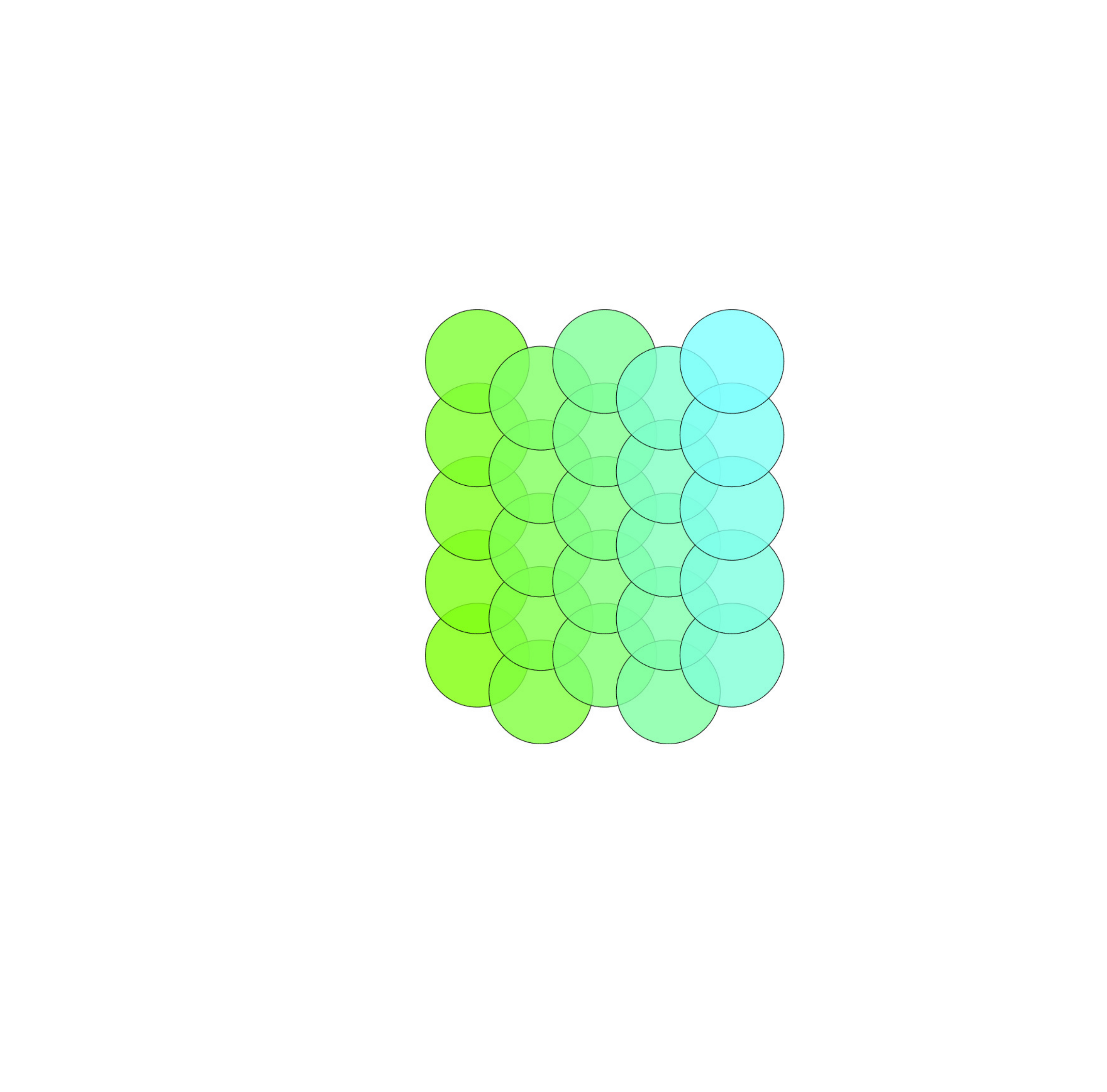}\vspace{0.5cm}
			\caption{$N_{\max}=3$ Layers.}
		\end{subfigure}
		\begin{subfigure}{.32\textwidth}
			\centering
			\includegraphics[width=1\linewidth, trim={2cm 4cm 2cm 4cm}, clip=true]{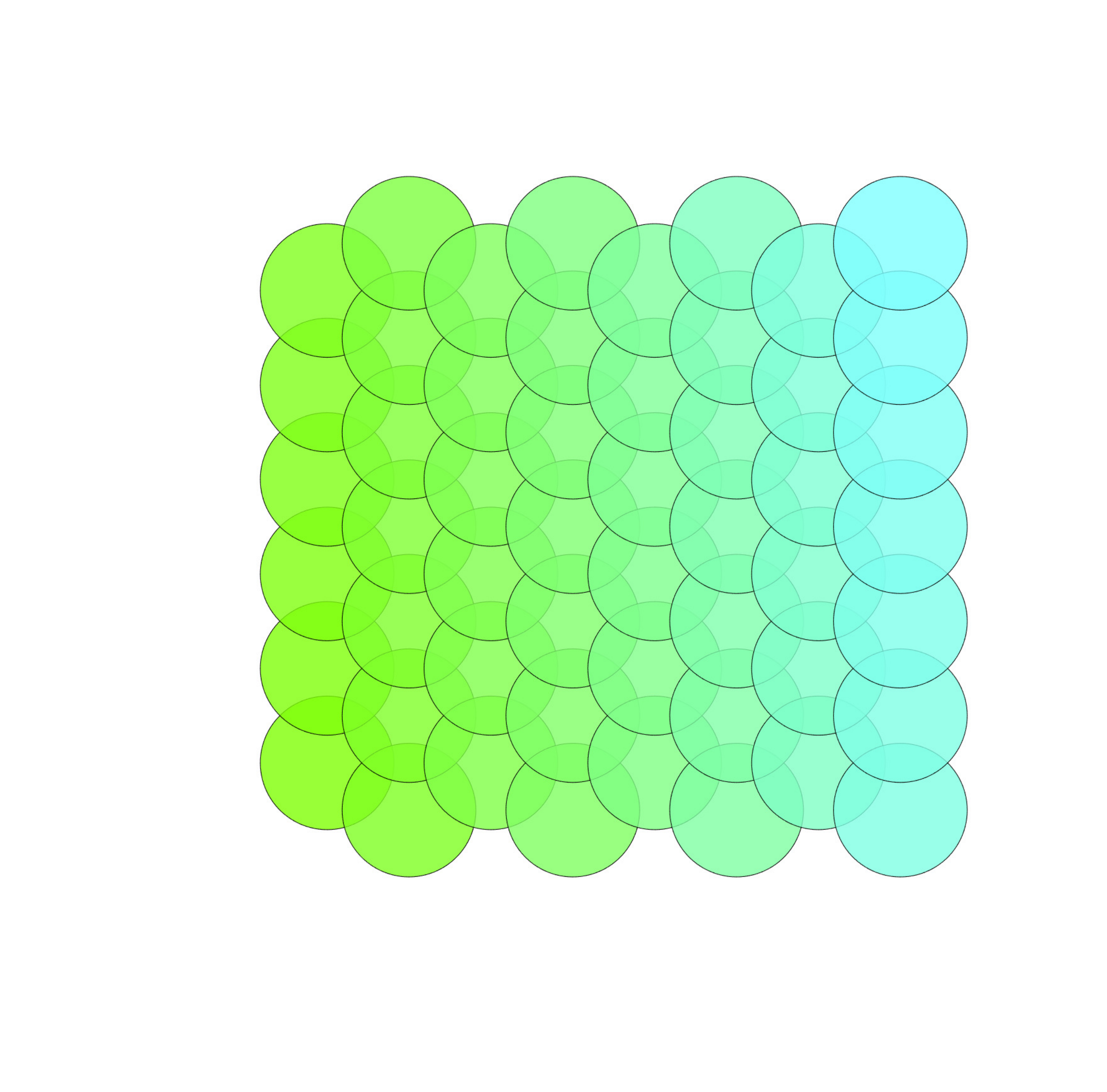}
			\caption{$N_{\max}=4$ Layers.}
		\end{subfigure}
		\caption{The different geometries used for the first set of numerical computations. The colour scheme is purely cosmetic. Note that both the radii and the distance between the centres of the subdomains in the three geometries are the same.}
		\label{fig:Numerics1a}
	\end{figure}
	
	\begin{figure}[h!]
		\centering
		\begin{subfigure}{0.45\textwidth}
			\centering
			\includegraphics[width=0.95\textwidth]{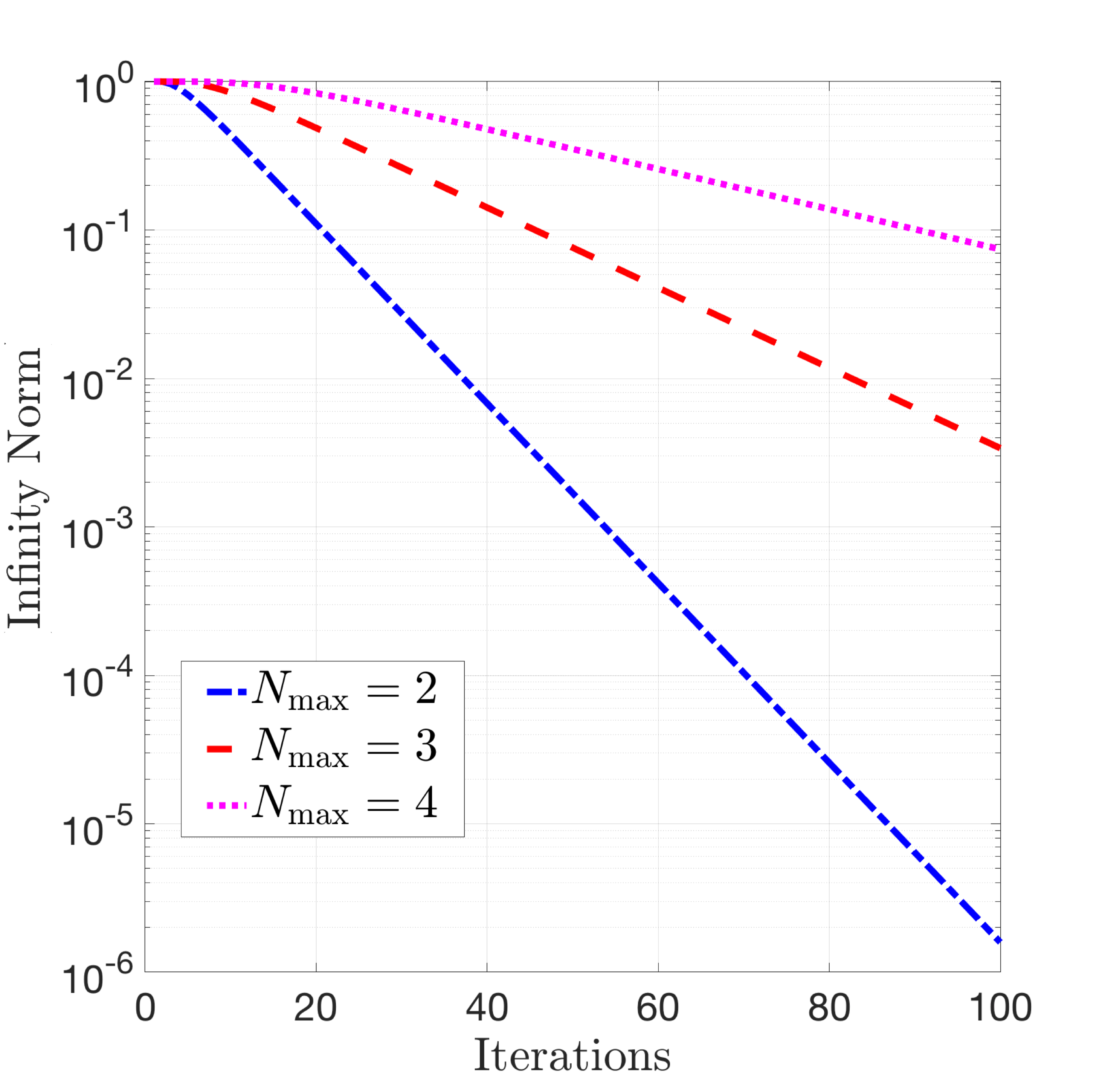} 
			\caption{Log-Lin plot of the Infinity norm of the Schwarz iterates for 100 iterations.}			
			\label{fig:Numerics1}	
		\end{subfigure} %
		\begin{subfigure}{0.45\textwidth}
			\centering
			\includegraphics[width=0.95\textwidth]{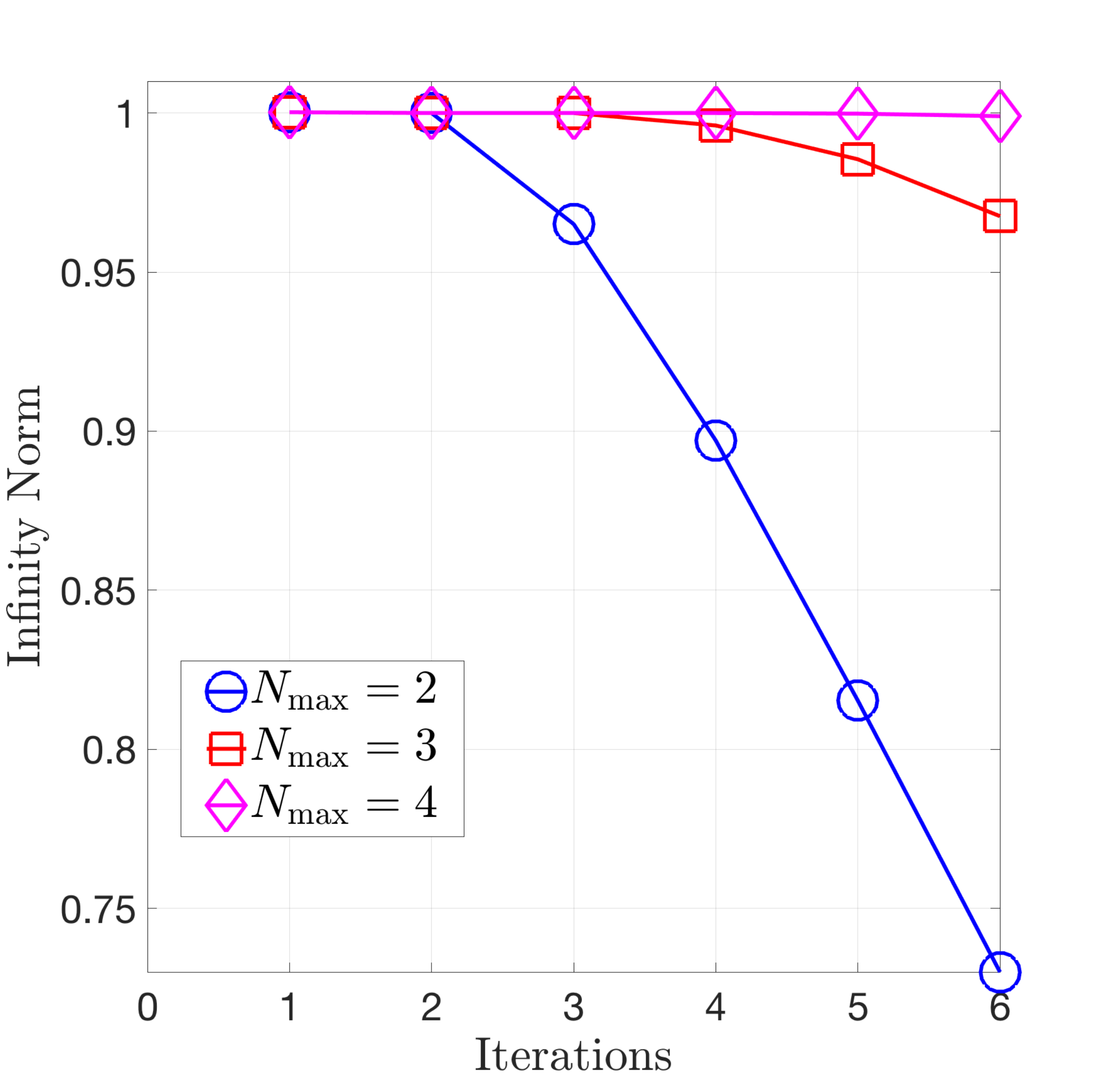} 
			\caption{Infinity norm of the Schwarz iterates for the first 6 iterations.}
			\label{fig:Numerics2}
		\end{subfigure}
		\caption{Numerical computations involving a computational domain with triple intersection subdomains.}
	\end{figure}
	
	Figures \ref{fig:Numerics1} and \ref{fig:Numerics2} display the infinity norm of the Schwarz iterates generated through the Error Equation \eqref{eq:5} using the initialisation $\bold{e}^0 \equiv 1$ on $\Pi_{j=1}^{j=N} \mathcal{S}_j$ for geometries with a different number of maximum layers. The different geometries are displayed in Figure \ref{fig:Numerics1a}. We emphasise that continuous partition of unity functions were chosen for this and all subsequent computations.
	
	Although it is slightly difficult to see this at a glance, the numerical results displayed in Figure \ref{fig:Numerics2} follow the theoretical results exactly. Indeed, the error function always displays a first contraction at iteration number $N_{\max}+1$. Furthermore, it is readily seen that the convergence rate of the error also depends on the total number of layers in the domain and degrades as $N_{\max}$ increases as shown in the log-lin plot displayed in Figure \ref{fig:Numerics1}. Indeed, we observe that the slope of the plots and consequently the asymptotic contraction factor degrades as $N_{\max}$ increases.

	\begin{figure}[h]
		\centering
		\begin{subfigure}{0.45\textwidth}
			\centering
			\includegraphics[width=0.95\textwidth]{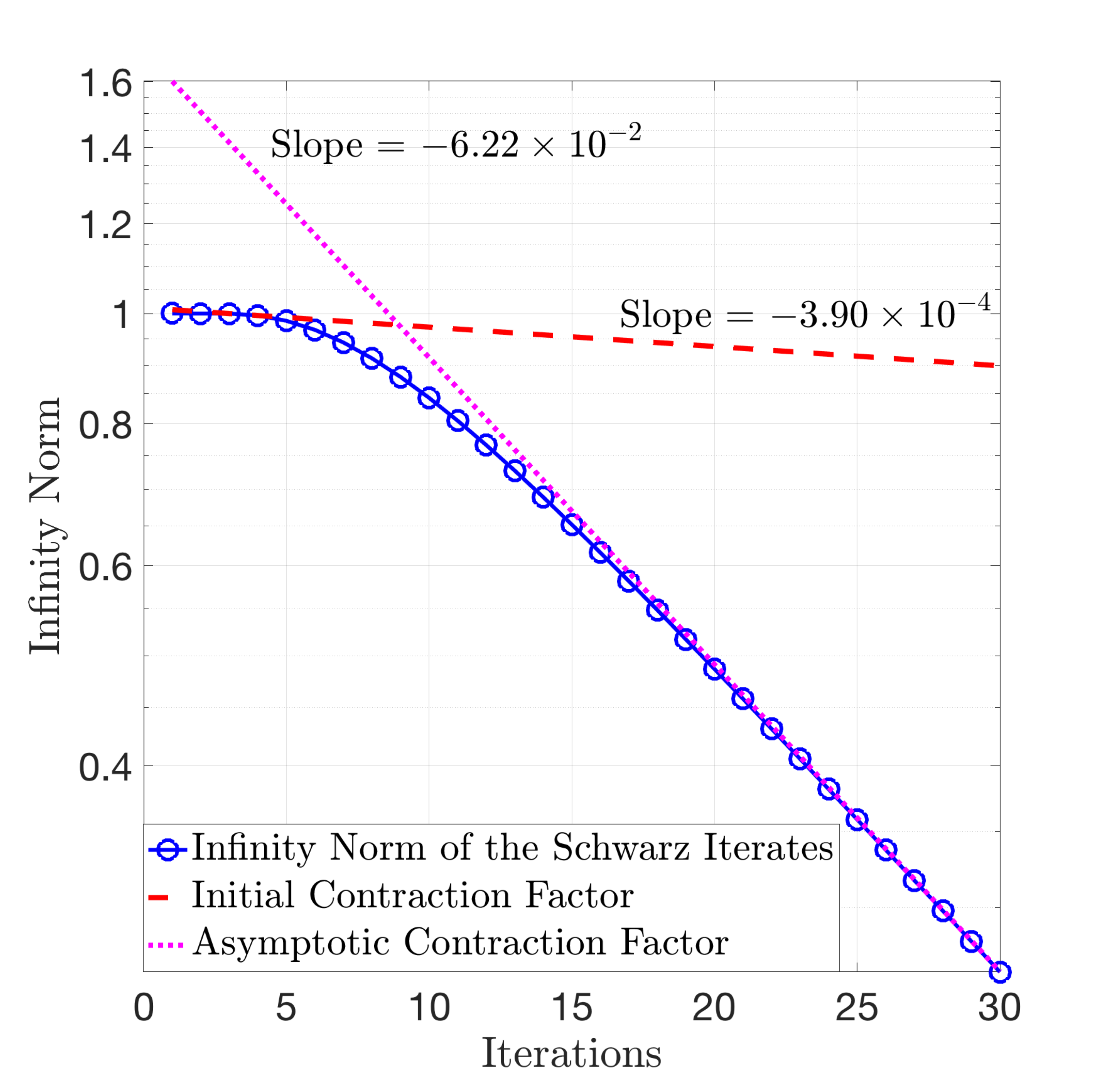} 
			\caption{Contraction Factors for $N_{\max}=3$ layers.}
			\label{fig:Numerics3}
		\end{subfigure}\hfill
		\begin{subfigure}{0.45\textwidth}
			\centering
			\includegraphics[width=0.95\textwidth]{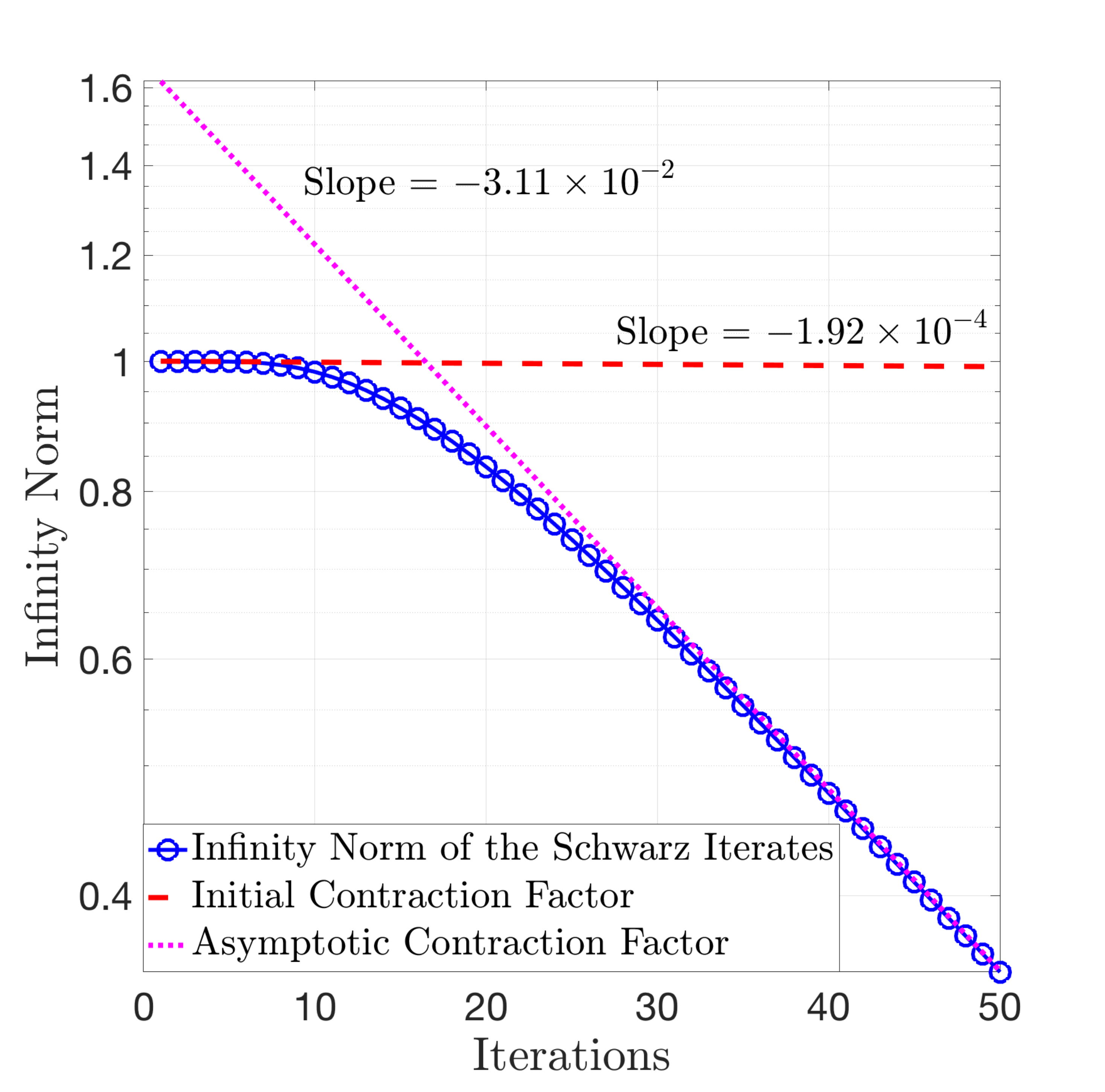} 
			\caption{Contraction Factors for $N_{\max}=4$ layers.}
			\label{fig:Numerics4}
		\end{subfigure}
		\caption{Numerical computations involving a computational domain with triple intersection subdomains.}
	\end{figure}

	One extremely complicated question that we did not address in the theoretical analysis of Section \ref{sec:2} is how to obtain a quantitive estimate of the contraction factor. Typically, one obtains an upper bound for the initial contraction factor, i.e., the initial decrease in the infinity norm of the Schwarz iterates (see, e.g., \cite{CiaramellaGander}, \cite{CiaramellaGander2} and \cite{CiaramellaGander3}), which then improves considerably in the asymptotic limit. Figures \ref{fig:Numerics3} and \ref{fig:Numerics4} display the exact infinity norm of the Schwarz iterates together with the (numerically obtained) initial and asymptotic contraction factor for geometries with $N_{\max}=3$ and $N_{\max}=4$ layers. We observe immediately that the first contraction in both cases is nearly two orders of magnitude smaller than the asymptotic contraction factor. The numerics therefore suggest that even if we were able to obtain an upper bound for the initial contraction factor, that is, an estimate of $\Vert T^{N_{\max}+1} \Vert_{\rm{OP}, \infty}$, it would be extremely conservative and not useful from a practical point of view. %We remark that one should compare the initial contraction factor computed here to the theoretical upper bound for one-dimensional geometries obtained in \cite{CiaramellaHassanStamm}.

	In order to support our claim regarding the extension of Theorem \ref{thm:1} to more complicated geometries, we next repeat the above numerical computations for a computational domain consisting of quadruple intersecting subdomains. The different geometries are displayed in Figure \ref{fig:Numerics2a}.
	
	\begin{figure}[h]
		\begin{subfigure}{.32\textwidth}
			\centering
			\includegraphics[width=1\linewidth, trim={4cm 8cm 4cm 8cm}, clip=true]{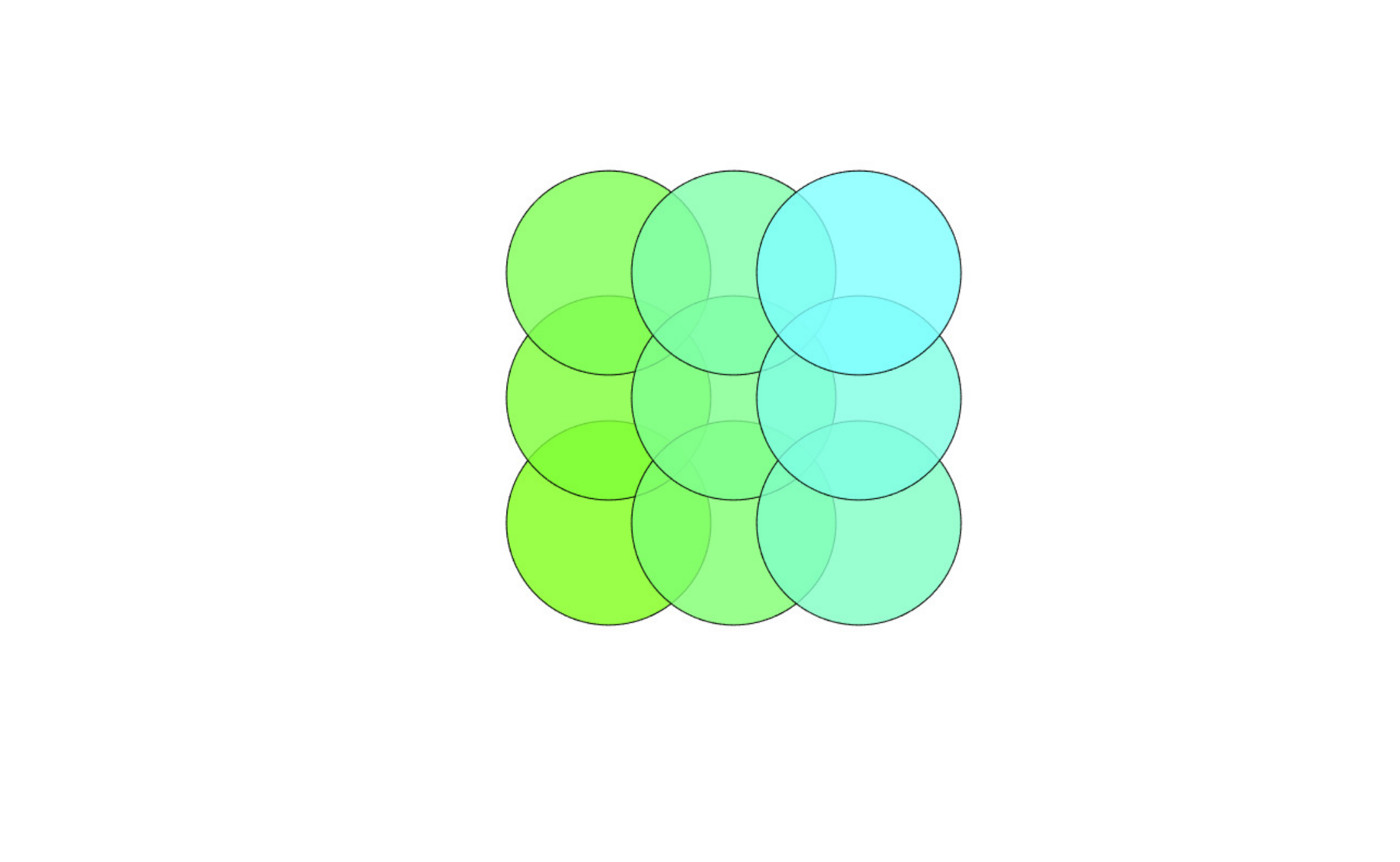}\vspace{1.02cm}
			\caption{$N_{\max}=2$ Layers.}
		\end{subfigure}%
		\begin{subfigure}{.32\textwidth}
			\centering
			\includegraphics[width=1\linewidth, trim={3.5cm 7cm 3.5cm 7cm}, clip=true]{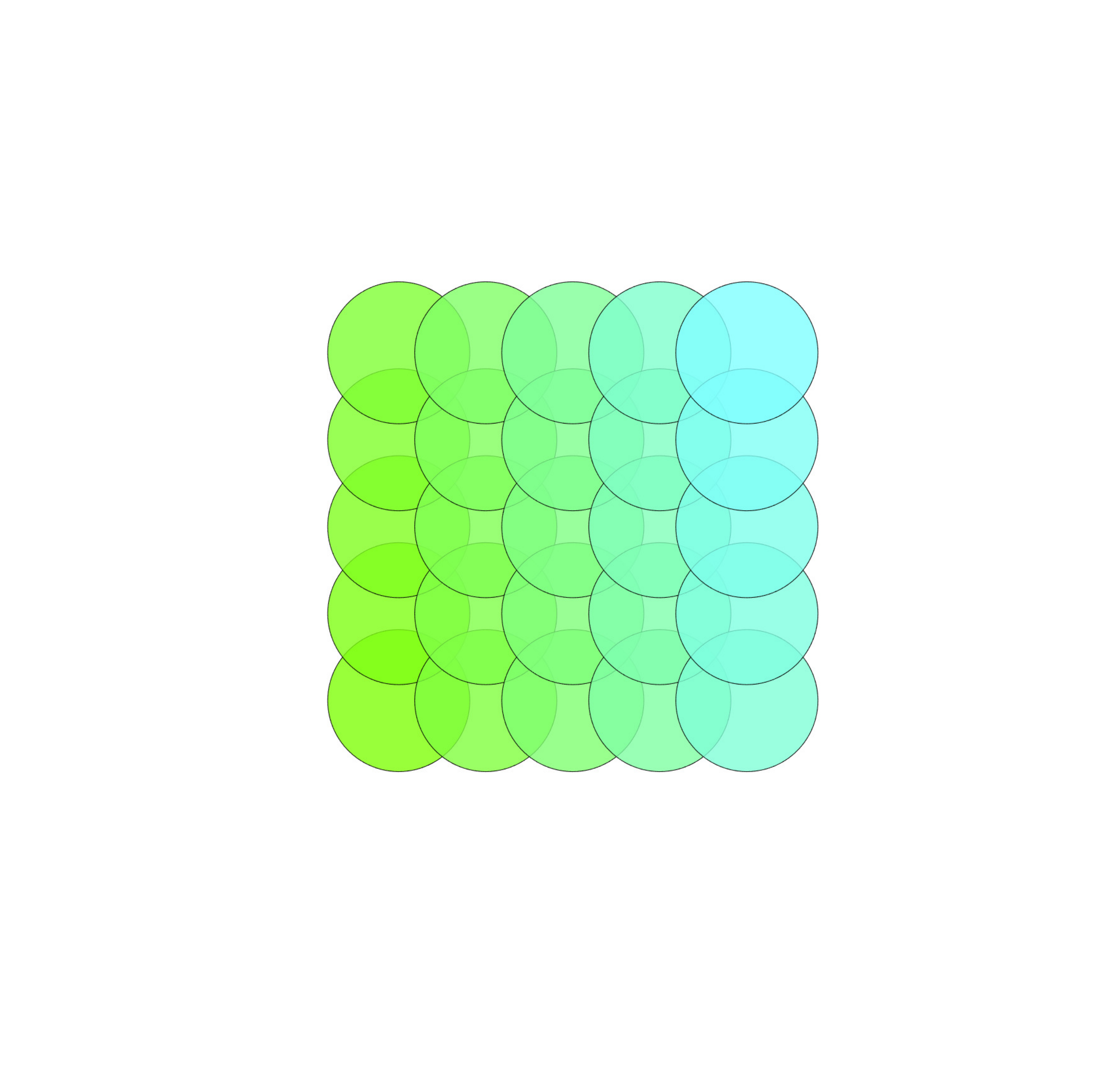}\vspace{0.75cm}
			\caption{$N_{\max}=3$ Layers.}
		\end{subfigure}
		\begin{subfigure}{.32\textwidth}
			\centering
			\includegraphics[width=1\linewidth, trim={2cm 4cm 2cm 3cm}, clip=true]{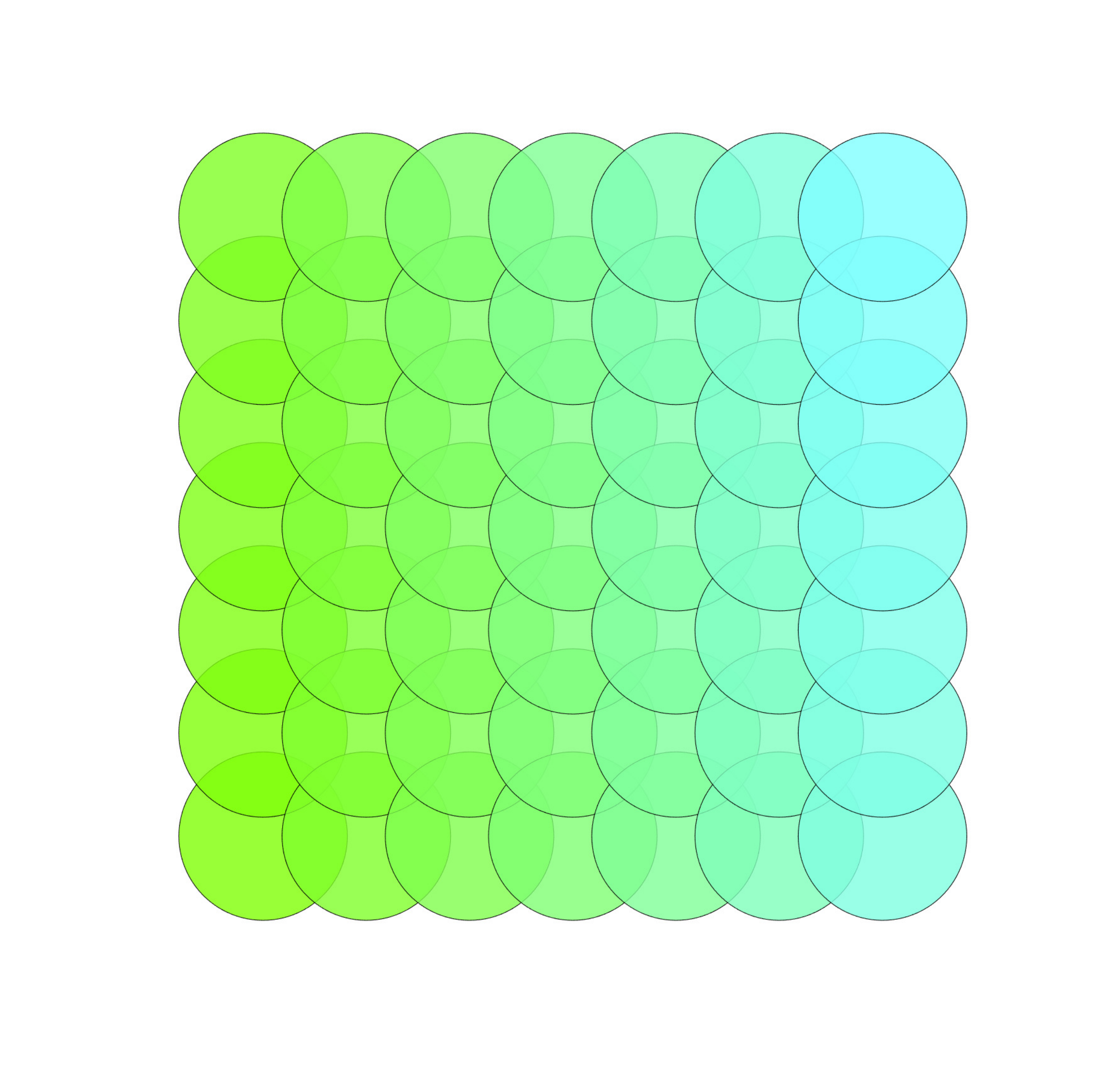}
			\caption{$N_{\max}=4$ Layers.}
		\end{subfigure}
		\caption{The different geometries used for the second set of numerical computations. The colour scheme is purely cosmetic. Note that both the radii and the distance between the centres of the subdomains in the three geometries is the same.}
		\label{fig:Numerics2a}
	\end{figure}
	
	Figures \ref{fig:Numerics22} displays the infinity norms of the Schwarz iterates generated through the Error Equation \eqref{eq:5} using the initialisation $\bold{e}^0 \equiv 1$ on $\Pi_{j=1}^{j=N} \mathcal{S}_j$ for quadruple intersecting geometries with a different number of maximum layers.  Once again, although it might be difficult to tell at a glance, the numerical results follows the theoretical results exactly and we see a first contraction of the infinity norm at iteration number $N_{\max}+1$.

	\begin{figure}[h]
		\centering
		\begin{subfigure}{0.45\textwidth}
			\centering
			\includegraphics[width=0.95\textwidth]{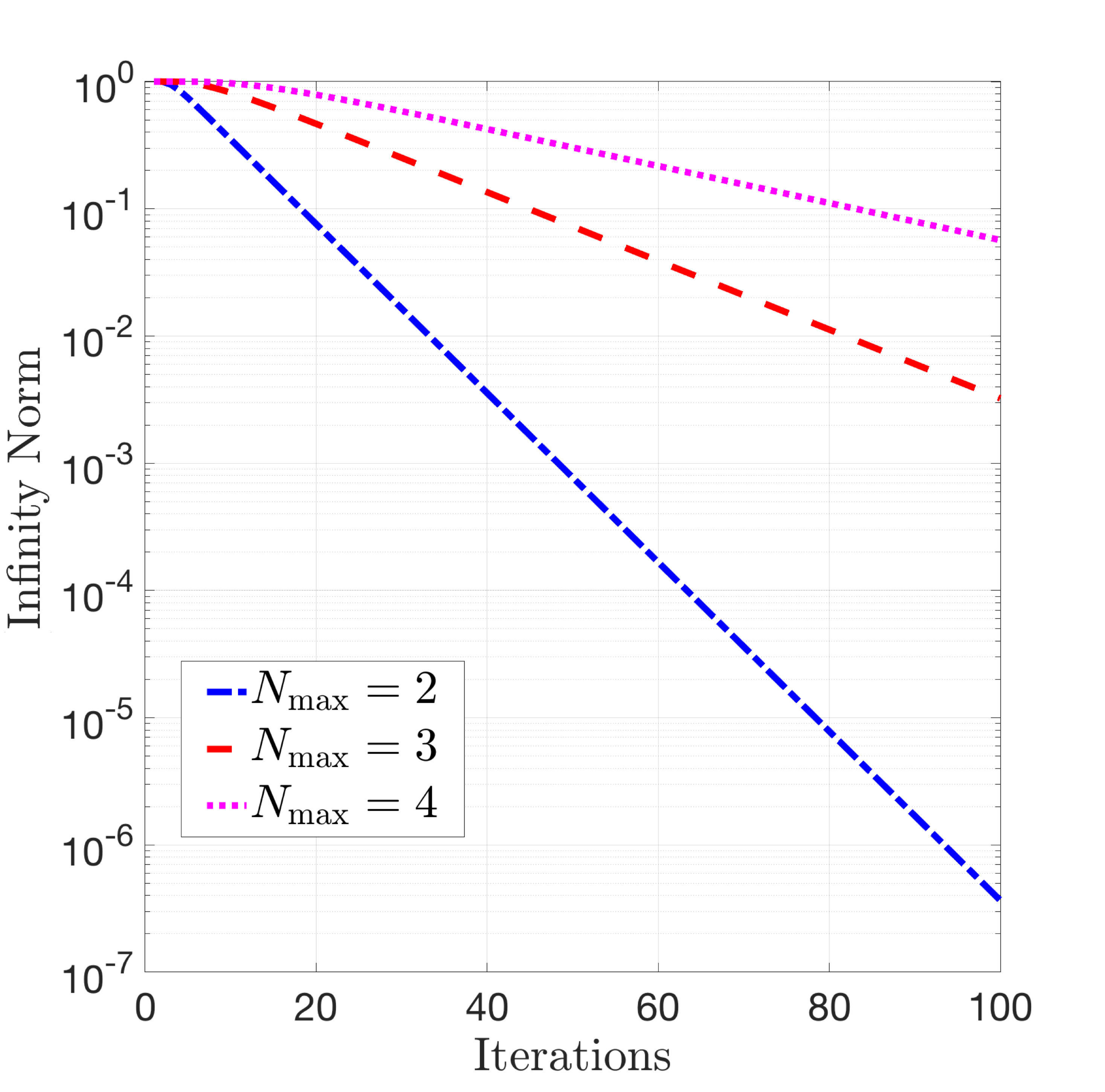} 
			\caption{Log-Lin plot of the Infinity norm of the Schwarz iterates for 100 iterations.}			
			\label{fig:Numerics21}	
		\end{subfigure} %
		\begin{subfigure}{0.45\textwidth}
			\centering
			\includegraphics[width=0.95\textwidth]{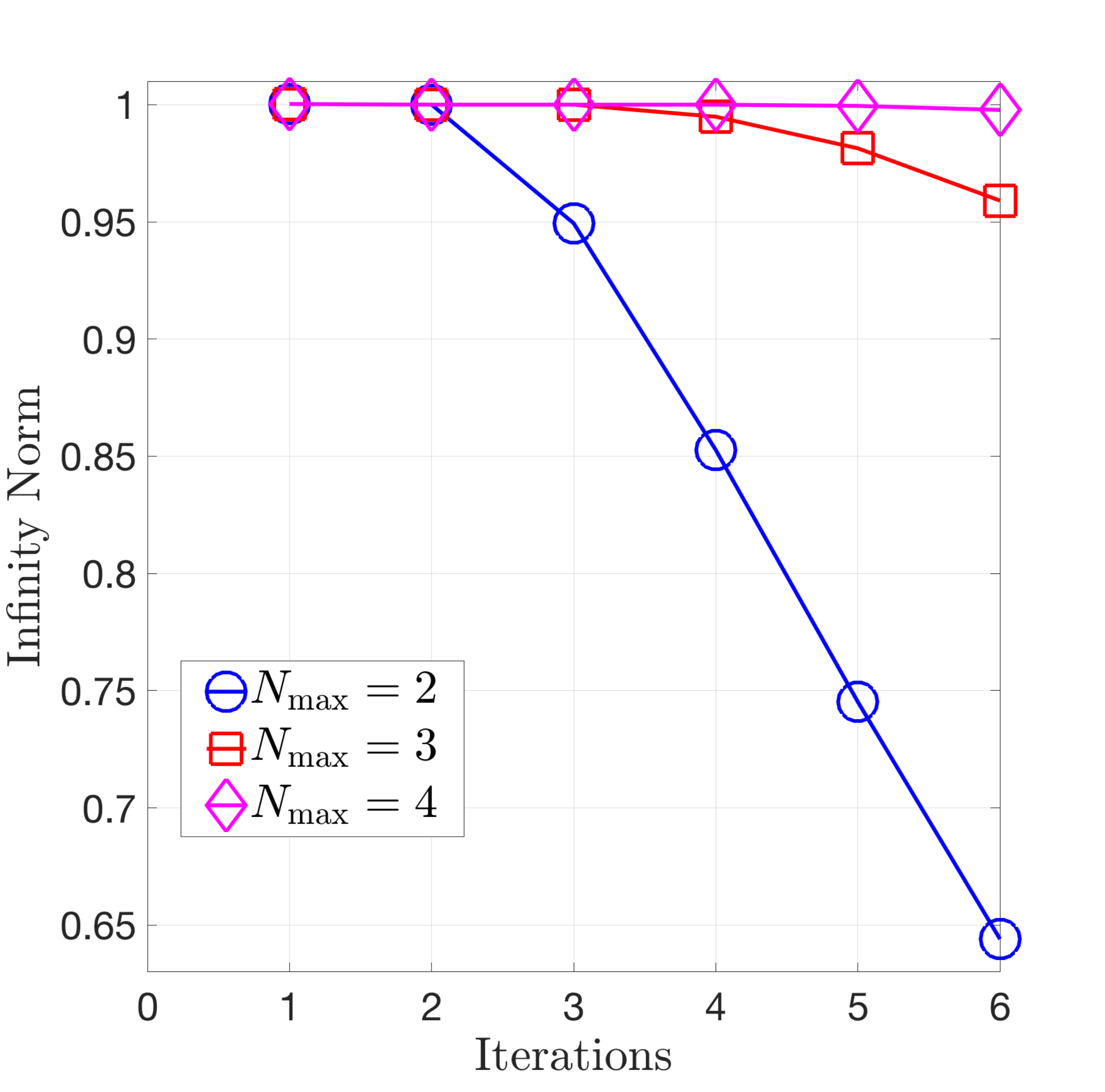} 
			\caption{Infinity norm of the Schwarz iterates for the first 6 iterations.}
			\label{fig:Numerics22}
		\end{subfigure}
		\caption{Numerical computations involving a computational domain with quadruple intersection subdomains.}
	\end{figure}

	\begin{figure}[h]
		\centering
		\begin{subfigure}{0.45\textwidth}
			\centering
			\includegraphics[width=0.95\textwidth]{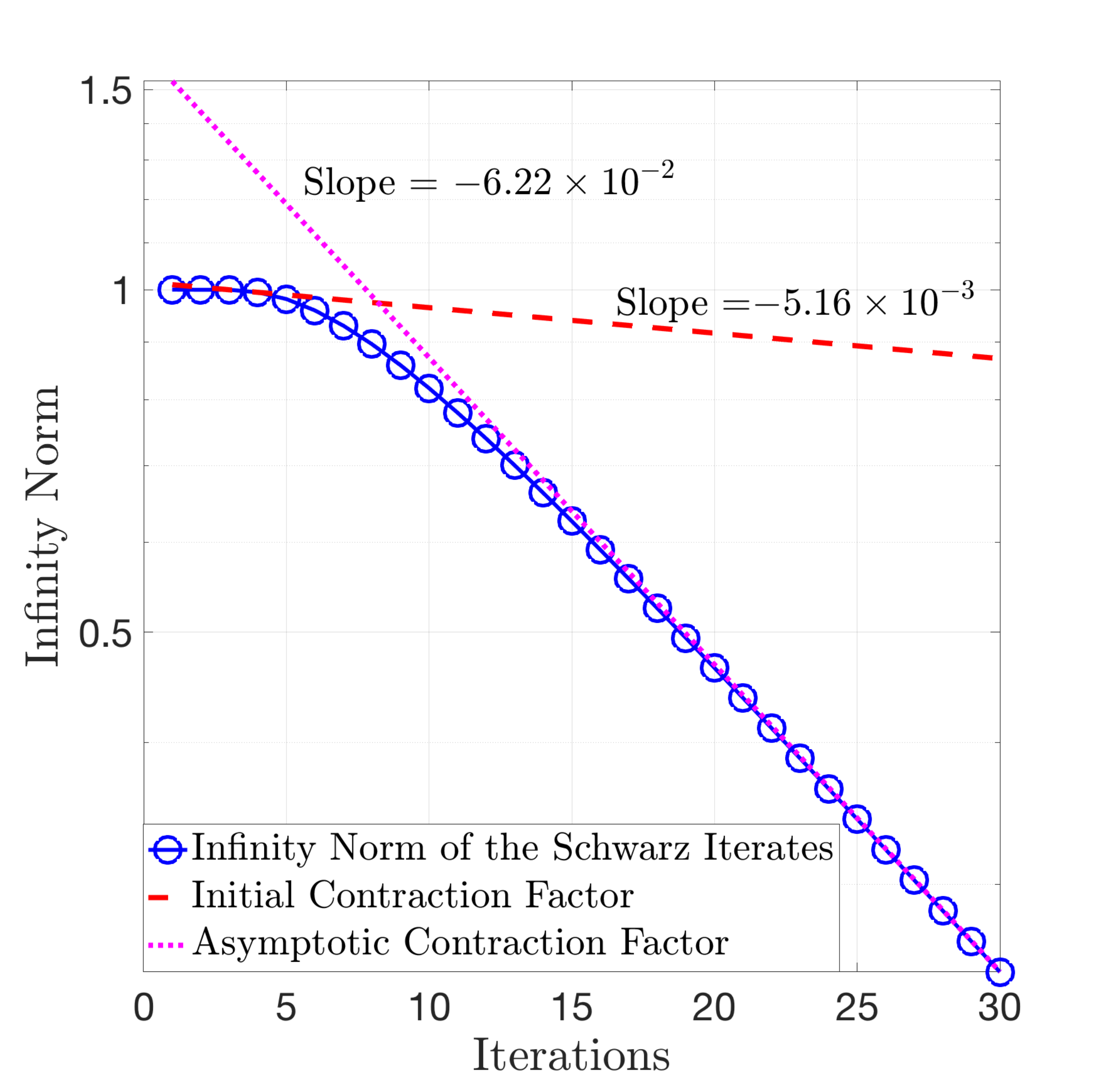} 
			\caption{Contraction Factors for $N_{\max}=3$ layers.}
			\label{fig:Numerics23}
		\end{subfigure}\hfill
		\begin{subfigure}{0.45\textwidth}
			\centering
			\includegraphics[width=0.95\textwidth]{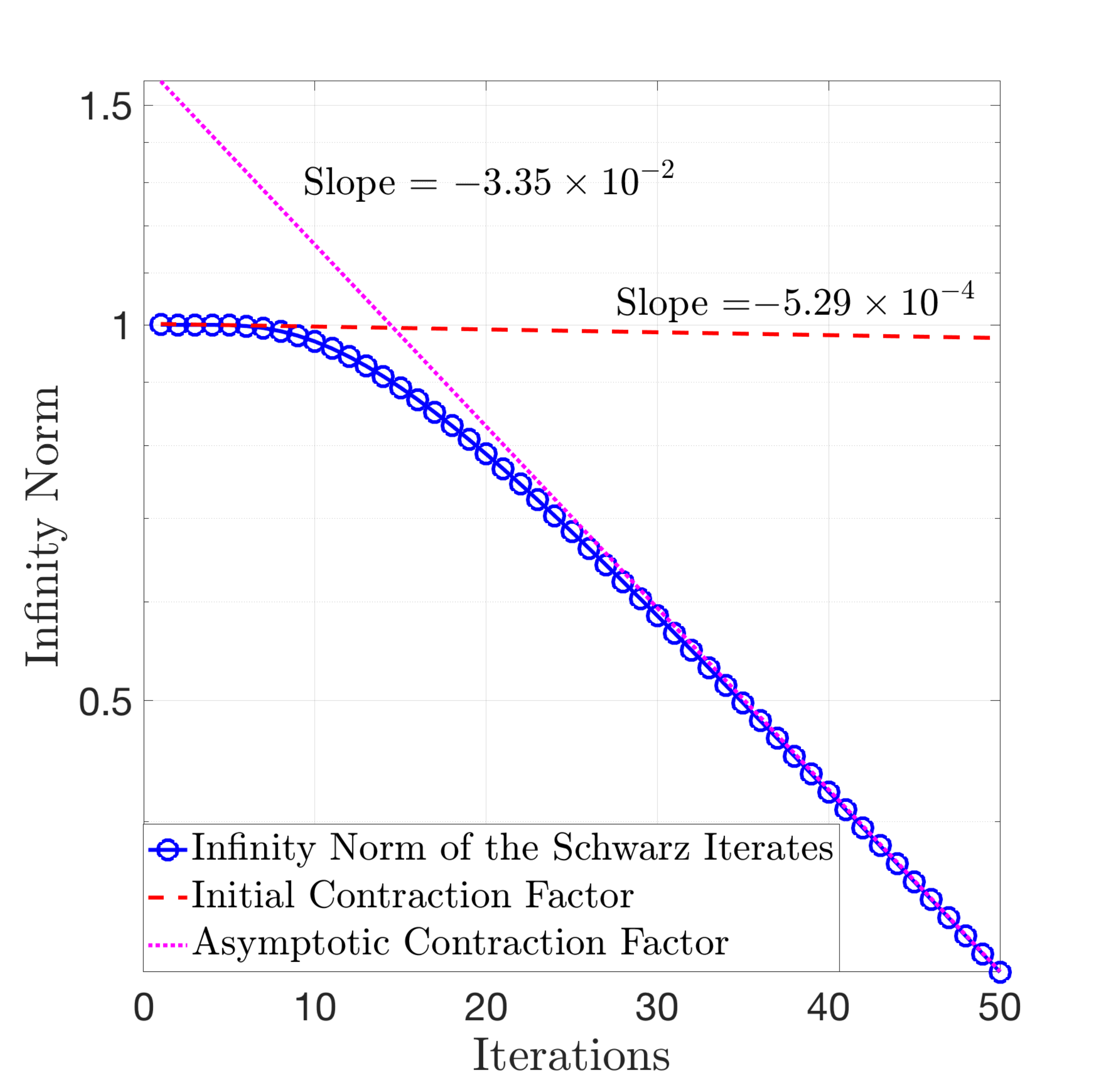} 
			\caption{Contraction Factors for $N_{\max}=4$ layers.}
			\label{fig:Numerics24}
		\end{subfigure}
		\caption{Numerical computations involving a computational domain with quadruple intersection subdomains.}
	\end{figure}
	
	\begin{figure}[h!]
		\centering
		\begin{subfigure}{0.45\textwidth}
			\centering
			\includegraphics[width=0.95\textwidth]{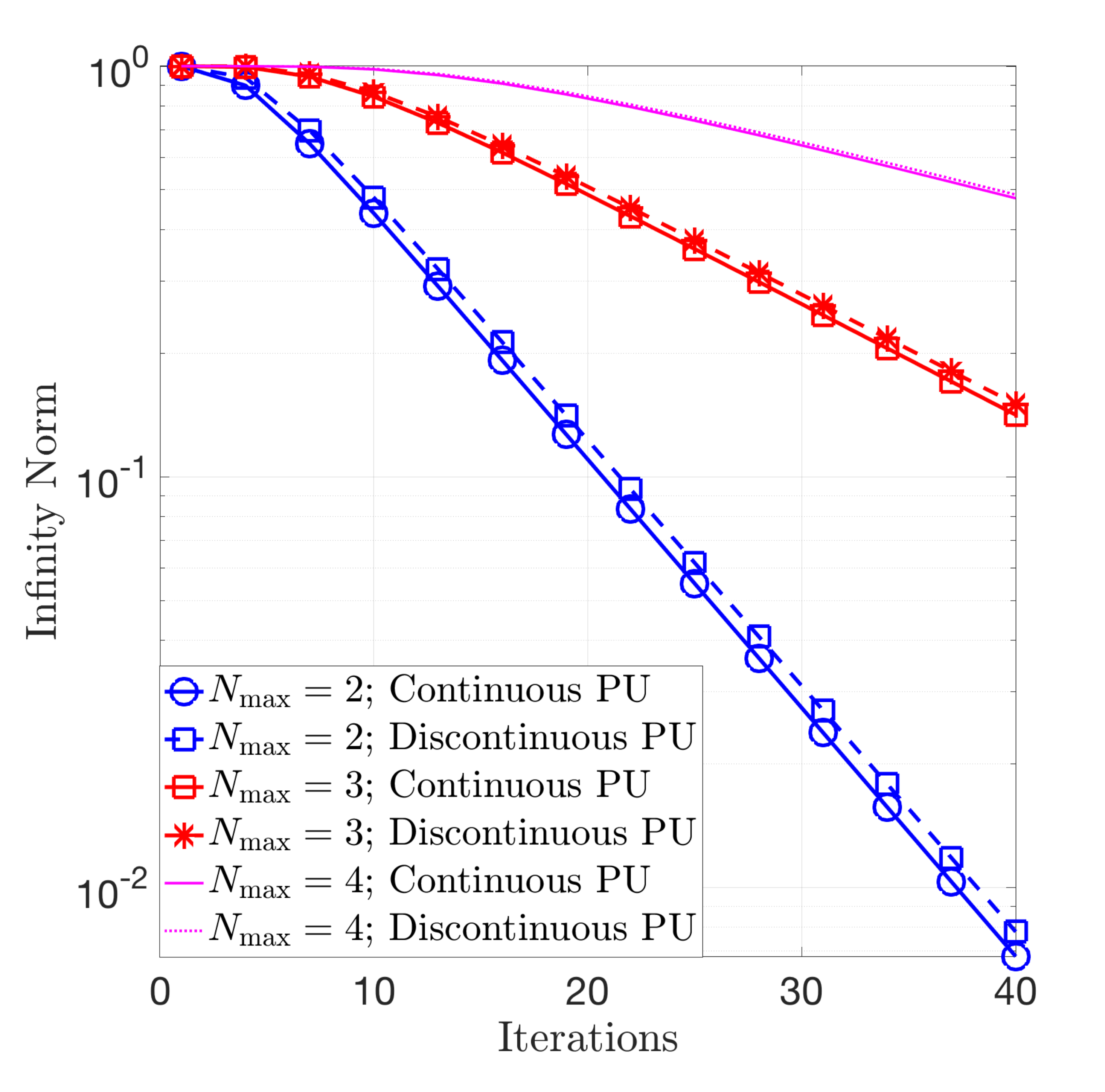} 
			\caption{Log-Lin plot of the Infinity norm of the Schwarz iterates for 40 iterations.}			
		\end{subfigure} %
		\begin{subfigure}{0.45\textwidth}
			\centering
			\includegraphics[width=0.95\textwidth]{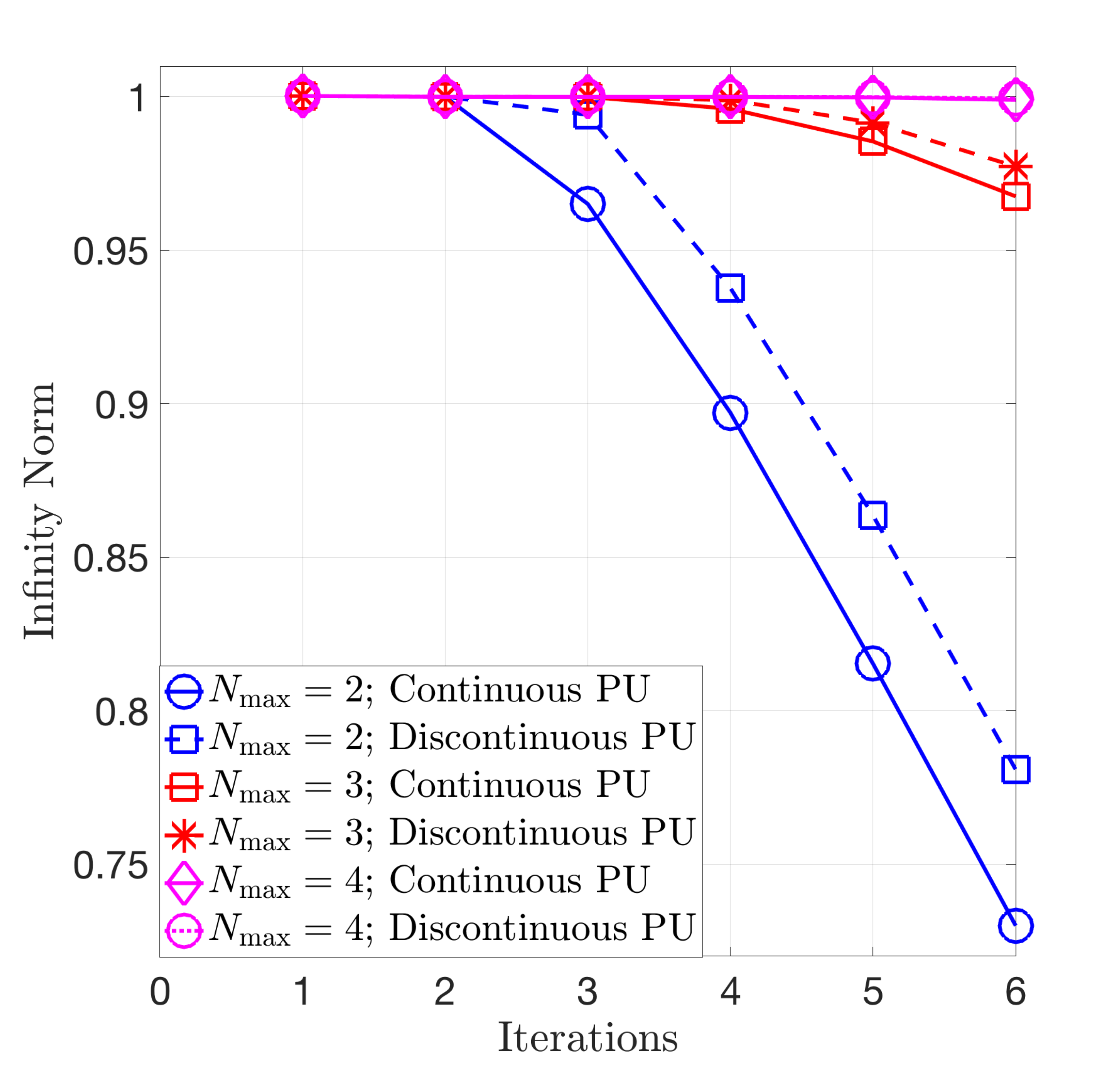} 
			\caption{Infinity norm of the Schwarz iterates for the first 6 iterations.}
		\end{subfigure}
		\caption{Numerical computations involving both continuous and discontinuous partition of unity functions. The computational domain consists of triple intersection subdomains as displayed in Figure \ref{fig:Numerics1a}.}
		\label{fig:Numerics33}
	\end{figure}
	
	Figure \ref{fig:Numerics21} displays the infinity norms of the Schwarz iterates as the number of iterations increases for the three geometries in a log-lin plot. Once again we observe that the asymptotic contraction factor degrades as $N_{\max}$ increases. In addition, we plot in Figures \ref{fig:Numerics23} and \ref{fig:Numerics24} the exact infinity norms of the Schwarz iterates together with the (numerically obtained) initial and asymptotic contraction factor for geometries with $N_{\max}=3$ and $N_{\max}=4$ layers. As before, the initial contraction factor can be seen to be at least an order of magnitude smaller than the asymptotic convergence factor.
	
	We conclude this section by returning to a question posed in Remark \ref{rem:0}: Does the choice of partition of unity functions affect the iterates of the Parallel Schwarz method and the asymptotic contraction factor? We decided to compute the infinity norm of the Schwarz iterates generated through the Error Equation \eqref{eq:5} using the initialisation $\bold{e}^0 \equiv 1$ on $\Pi_{j=1}^{j=N} \mathcal{S}_j$ for the geometries displayed in Figure \ref{fig:Numerics1a} using both the continuous partition of unity functions employed previously as well as discontinuous partition of unity functions as chosen in \cite{Gander:2008}[Section 2, Page 4]. The discontinuous partition of unity functions are defined as follows: Given three intersecting subdomains $\Omega_1, \Omega_2, \Omega_3$
	\begin{itemize}
		\item We set the partition of unity functions $\chi_1^2, \chi_1^3$ on $\partial \Omega_1$ as $\chi_1^2 = 1$ on $\overline{\Gamma_1^{2, 3}} \cup \overline{\Gamma_1^{2, 0}} $ and $\chi_1^2=0$ otherwise, and $\chi_1^3 = 1$ on $\overline{\Gamma_1^{3, 0}}$ and $\chi_1^3=0$ otherwise.
		\item We set the partition of unity functions $\chi_2^3, \chi_2^1$ on $\partial \Omega_2$ as $\chi_2^3 = 1$ on $\overline{\Gamma_2^{1, 3}} \cup \overline{\Gamma_2^{3, 0}} $ and $\chi_2^3=0$ otherwise, and $\chi_2^1 = 1$ on $\overline{\Gamma_2^{1, 0}}$ and $\chi_2^1=0$ otherwise.
		\item We set the partition of unity functions $\chi_3^1, \chi_3^2$ on $\partial \Omega_3$ as $\chi_3^1 = 1$ on $\overline{\Gamma_3^{1, 2}} \cup \overline{\Gamma_3^{1, 0}} $ and $\chi_3^1=0$ otherwise, and $\chi_3^2 = 1$ on $\overline{\Gamma_3^{2, 0}}$ and $\chi_3^2=0$ otherwise.
	\end{itemize}
	
	In other words given any subdomain $\Omega_j$, on the portion of the boundary $\Gamma_j^{k, i} \subset \partial \Omega_j$ that intersects simultaneously with \emph{two} other subdomains $\Omega_k$ and $\Omega_i$, we take boundary data from only one neighbour. Our results are displayed in Figures \ref{fig:Numerics33}, and they suggest that 
	
	\begin{itemize}
		\item While the choice of different partition of unity functions can lead to quantitatively slightly different infinity norms of the Schwarz iterates, the asymptotic contraction factor remains unchanged;
		\item The first contraction in the infinity norm is always observed at iteration number $N_{\max}+1$ regardless of the choice of the partition of unity functions.
	\end{itemize}
	
	For more information and details on the effect of the choice of partition of unity functions on the convergence of Schwarz methods, see also \cite{NewRef}.

	\section{Application to bio-molecules: van der Waals and SAS cavities}\label{sec:Protein}~\\
	
	\begin{figure}[h]
		\centering
		\begin{subfigure}{0.49\textwidth}
			\centering
			\includegraphics[width=\textwidth]{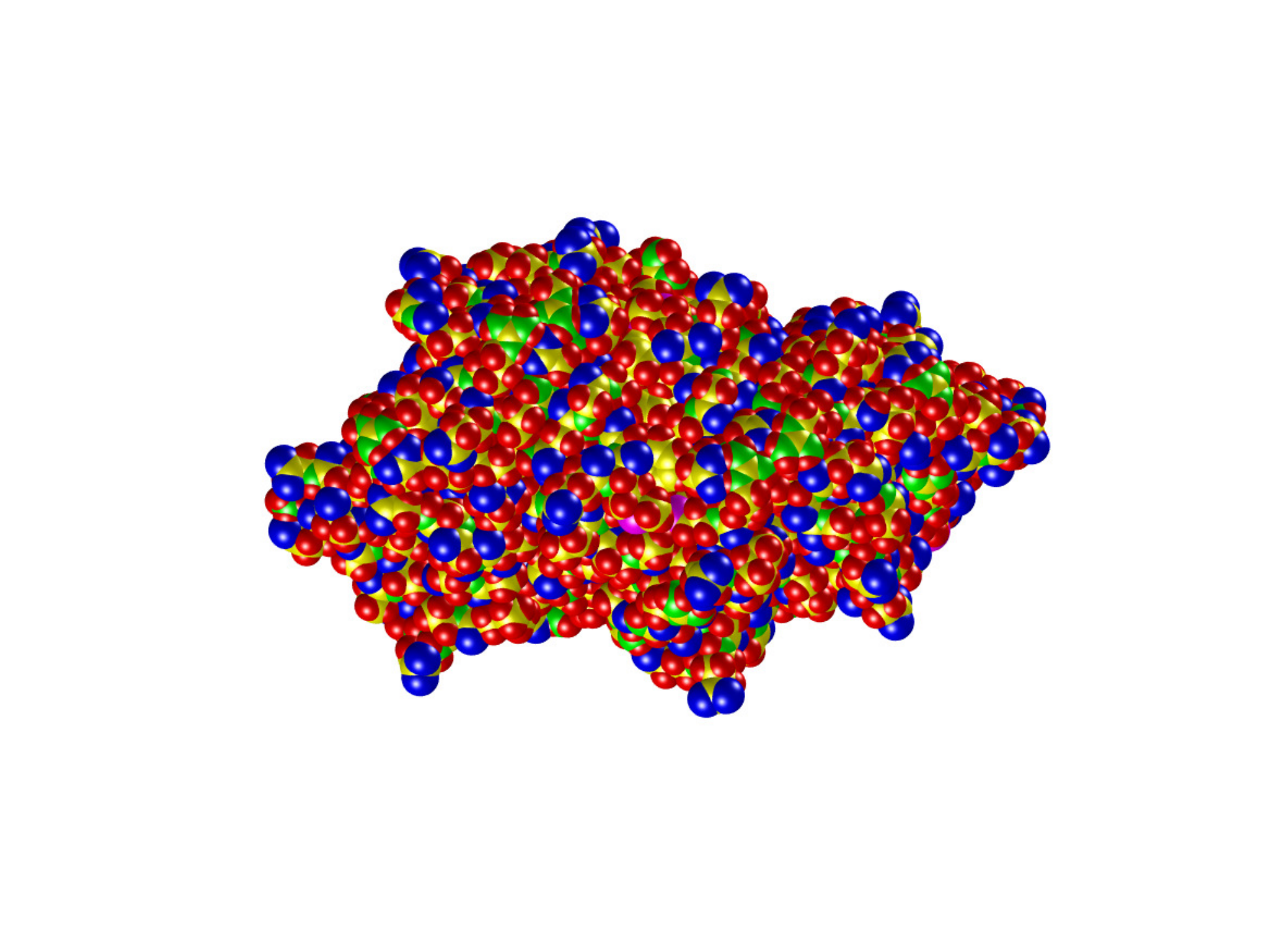} 
			\caption{Atomic Structure of Carboxylase. }			
		\end{subfigure} %
		\begin{subfigure}{0.49\textwidth}
			\centering
			\includegraphics[width=\textwidth]{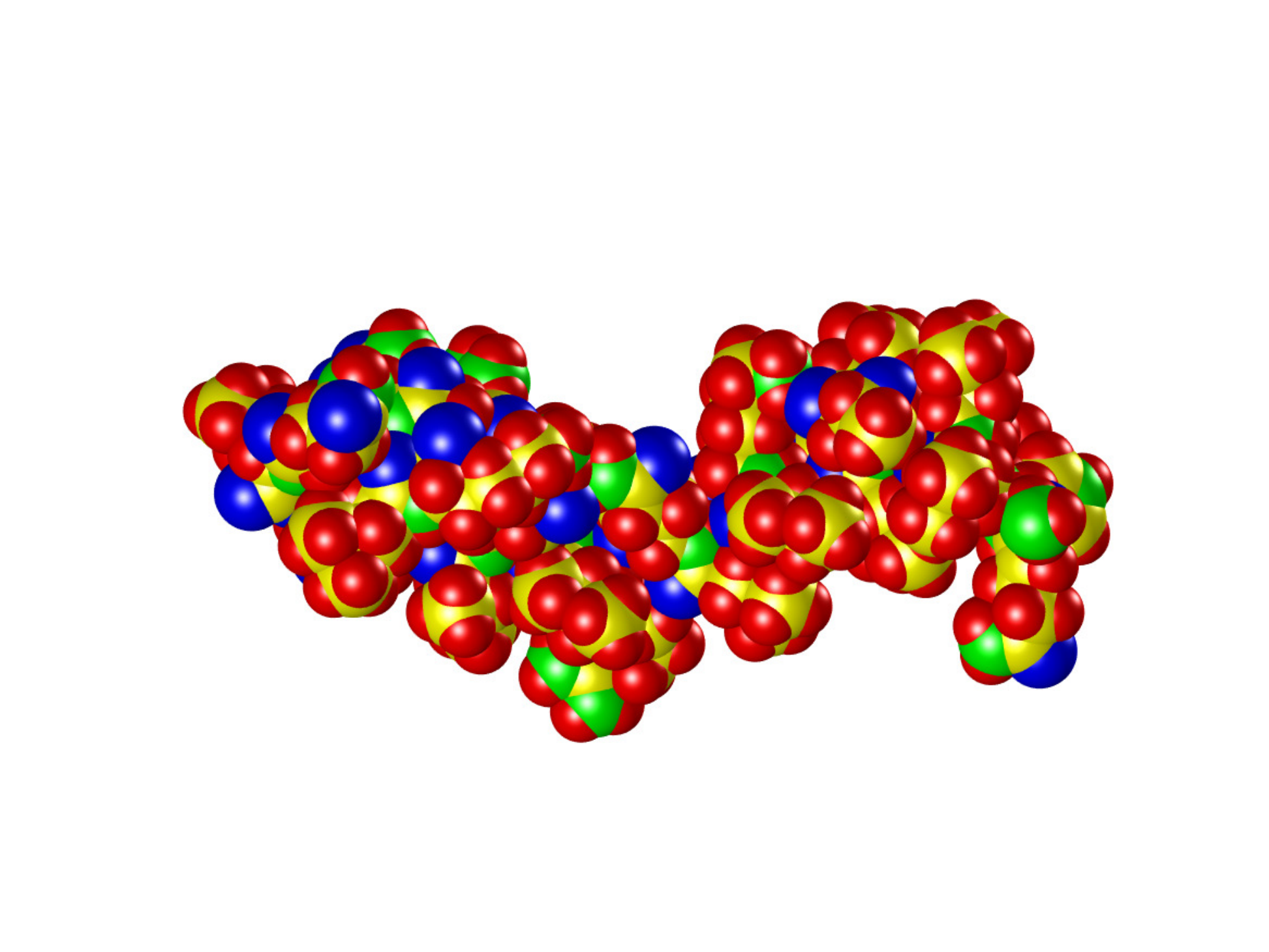} 
			\caption{Atomic Structure of Hiv-1-gp41. }
		\end{subfigure}
		\caption{Atomic Structures of `globular' Carboxylase and `chain-like' Hiv-1-gp41. The colour of each atomic cavity denotes the type of atom: yellow denotes Carbon, red denotes Hydrogen, green denotes Nitrogen, blue denotes Oxygen and magenta denotes Sulphur. The atomic cavities were constructed using 1.1 times the radii obtained from the Universal Force Field (UFF) \cite{UFF}.}
		\label{fig:Protein1}
	\end{figure}
	
	%.
	The examples considered thus far in Section \ref{sec:num} have been `toy' problems, which are interesting from a mathematical point of view but do not have a physical origin. The goal of this section is to consider actual biological molecules that have been studied in solvation models in computational chemistry (see, e.g., \cite{BenPaper}) and to understand the consequences of the preceding analysis as it pertains to these complex bio-molecules.
	
	As a reference, we consider the biological molecules that were analysed in the ddCOSMO numerical study presented in \cite{BenPaper}. Following standard practice for the COSMO implicit solvation model, these molecules are represented as the union of intersecting balls with each ball corresponding to one atom in the molecule. The computational domain thus consisted of a union of open balls, and the Laplace equation with a prescribed, smooth boundary condition was solved on this computational domain. The problem setting in \cite{BenPaper} therefore is exactly of the form of Equation \eqref{eq:1}, except of course that the problem is posed in three dimensions. It is important to note however that we must make a choice for the value of the radius of each atomic cavity. While this is a non-trivial question in general (see, e.g., the references in \cite{Quan}), the most basic convention (see, e.g., \cite{Stamm3}) is to construct the atomic cavities using 1.1 times the radii obtained from the Universal Force Field (UFF) \cite{UFF}. This is the convention we adopt as a first step.
	
	Two representative examples of these biological molecules, Carboxylase and Hiv-1-gp41 are displayed in Figure \ref{fig:Protein1}. It is readily seen that the Carboxylase molecule is `globular' in structure and consists of thousands of atoms, and thus thousands of subdomains in the domain decomposition setting. On the other hand, the Hiv-1-gp41 molecule is more `chain-like' in structure and consists of less than 600 atoms. Based on these visualisations, one could expect the Carboxylase molecule to contain a large number of layers and the Hiv-1-gp41 molecule to contain fewer layers. As our preceding analysis shows, this could result in very different computational efficiencies of the parallel Schwarz method (PSM) for the two molecules as described by Equation \eqref{eq:3} and implemented through the ddCOSMO numerical algorithm. 
	\begin{table}[h]
		\begin{center}
			\begin{adjustbox}{max width=\textwidth}
				\begin{tabular}{ | l || c | c | c | c | c |r |}
					\hline \hline
					Bio-Molecule & Vancomycin & UBCH5B & L-plectasin & Hiv-1-gp41& Glutaredoxin & Carboxylase\\ \hline
					Wall-Times & 1 & 4& 1 & 1& 2 & 9 \\ \hline
					Convergence Factors & 0.8801 & 0.9287 & 0.8990 & 0.9105& 0.9269 & 0.9324\\  \hline
					Number of Atoms & 377 & 2364& 567 & 530& 1277 & 6609 \\
					\hline \hline
				\end{tabular}
			\end{adjustbox}
		\end{center}
		\caption{Wall-times (in seconds) and asymptotic convergence rates of the ddCOSMO algorithm. Times less than 1 second have been rounded to one. The wall-times were obtained by solving the underlying linear system using the DIIS algorithm introduced by Peter Pulay in quantum computational chemistry \cite{Peter}. Note that the DIIS algorithm is simply Anderson acceleration. These accelerations are applied to the PSM as a stationary method, which is essentially equivalent to GMRES preconditioned with the PSM \cite{Anderson}. The asymptotic convergence rates were calculated using the PSM as a stationary method.}
		\label{table:Protein1}
	\end{table}
	
	Table \ref{table:Protein1} displays the wall-time of the ddCOSMO algorithm as reported in \cite{BenPaper}, and we can observe the surprising result that the PSM (as a preconditioner) seems to work extremely efficiently for all the biological molecules considered here. In particular, we observe excellent scaling of the method with respect to the number of atoms in each molecule. Furthermore, our own numerical computations, which are shown in the second row of Table \ref{table:Protein1} and displayed in Figure \ref{fig:Protein2} indicate that there is minimal degradation of the asymptotic convergence rate of the ddCOSMO algorithm for these bio-molecules. In order to explain this observation, we decided to analyse carefully the structure of these protein molecules. 
	
	\begin{figure}[h]
		\centering
		\includegraphics[width=0.8\textwidth]{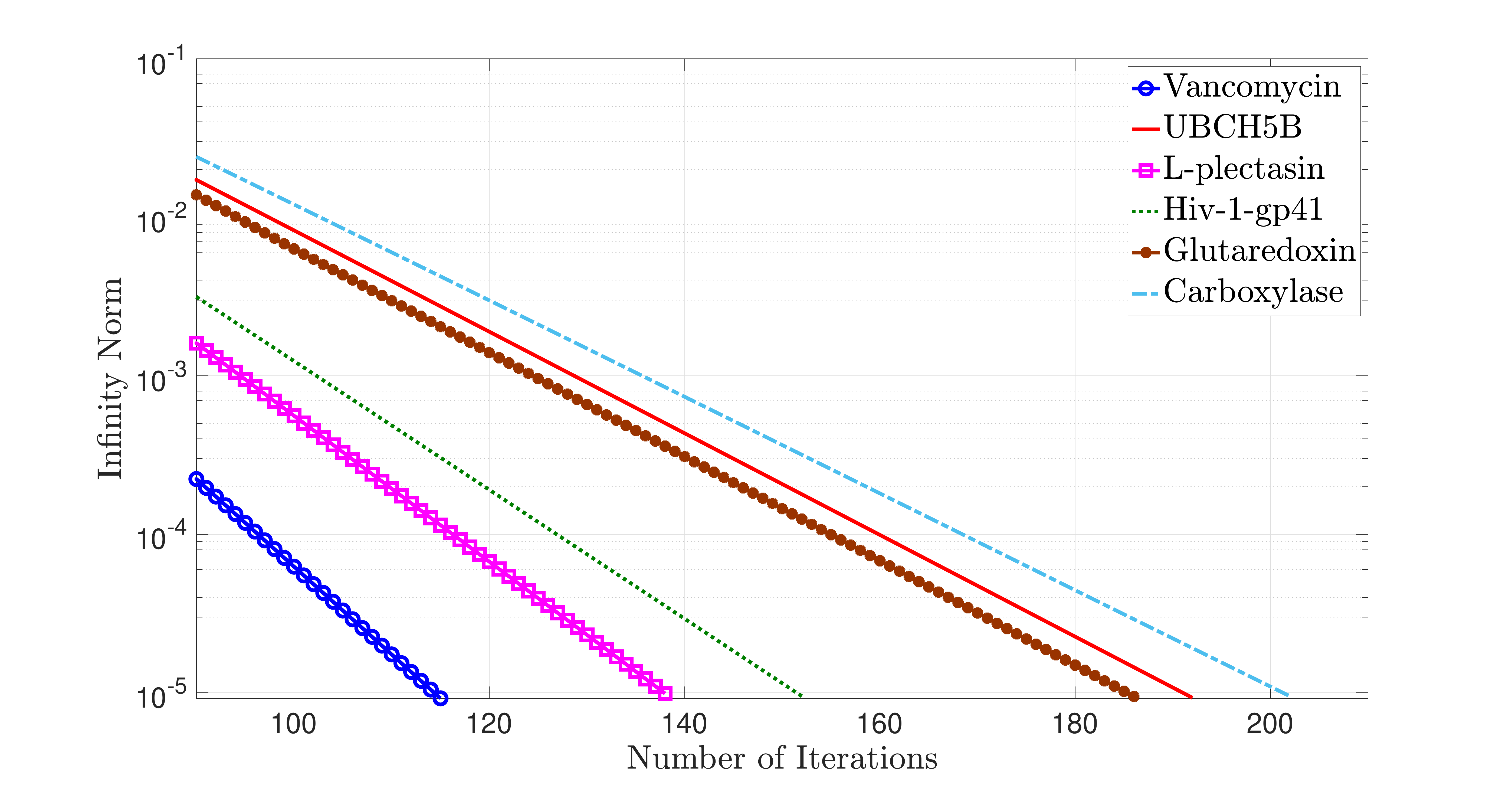} 
		\caption{Log-Lin plot of the Infinity norm of the approximate solution after each iteration of the PSM. All simulations were initialised with a constant 'one' vector and a zero right-hand side was imposed.}
		\label{fig:Protein2}
	\end{figure}

	Table \ref{table:Protein2} contains information about geometrical aspects of each of these protein molecules that affect the convergence rate of the PSM. We indicate the total number of atoms, the number of layers, which is the central component of our analysis, the average number of neighbours, the average over different subdomains of the maximum degree of intersection, and the average over different subdomains of the overlap distance between neighbouring atoms. The following two observations are now key:
	
	\begin{itemize}
		\item There are exactly two layers in each bio-molecule;
		\item The other geometrical parameters which could potentially affect the convergence rate of the PSM do not vary wildly across the different molecules.
	\end{itemize}
	
	\begin{table}[h]
		\begin{center}
			\begin{adjustbox}{max width=\textwidth}
				\begin{tabular}{ | l || c | c | c | c | c |r |}
					\hline \hline
					Bio-Molecule & Vancomycin & UBCH5B & L-plectasin & Hiv-1-gp41& Glutaredoxin & Carboxylase\\ \hline
					Number of Atoms & 377 & 2364& 567 & 530& 1277 & 6609 \\ \hline
					Number of Layers & 2 & 2 & 2 & 2 & 2 & 2\\
					\hline
					Average Number of Neighbours & 15.6 & 19.8 & 16.6 & 16.5& 18.1 & 20.2\\
					\hline
					Average Maximum Intersection Degree & 6.16 & 6.27 & 6.11 & 6.12& 6.50 & 6.24\\
					\hline
					Average Overlap & 2.3 & 2.3 & 2.3 & 2.2& 2.3 & 2.3\\
					\hline \hline
				\end{tabular}
			\end{adjustbox}
		\end{center}
		\caption{Various geometrical properties of each biological molecule considered in the numerical simulations. The average overlap distance is calculated by considering a pair of neighbouring atoms, computing the difference between the sum of the two radii and the distance between the centres, and then taking the average over all pairs of neighbouring atoms.}
		\label{table:Protein2}
	\end{table}

	These observations help explain our numerical results and also explain why the ddCOSMO numerical algorithm works so well in practice even for apparently globular biological molecules \cite{Stamm1, Stamm2, Stamm3}.

	In the computational chemistry community, it is generally argued that using scaled van der Waals radii to construct atomic cavities leads to a sub-optimal definition of the molecular cavity that is unable to accurately capture the physical behaviour of the molecule in a polarisable medium. A more elaborate definition of the molecular cavity is given by the so-called Surface Accessible Surface (SAS) (see, e.g., \cite{Quan}). According to this definition, each atomic cavity is constructed using the sum of two radii: the atomic van der Waals radius and a so-called `probe radius', which can generally be taken to be 1.4 Angstrom, i.e., the approximate radius of a water molecule. As a second step therefore, we adopt the SAS convention to construct the atomic cavities. Figure \ref{fig:Protein300} displays the SAS cavity-based atomic structure of the biological molecules Carboxylase and Hiv-1-gp41 (c.f., Figure \ref{fig:Protein1}).

	\begin{figure}[h]
		\centering
		\begin{subfigure}{0.49\textwidth}
			\centering
			\includegraphics[width=\textwidth]{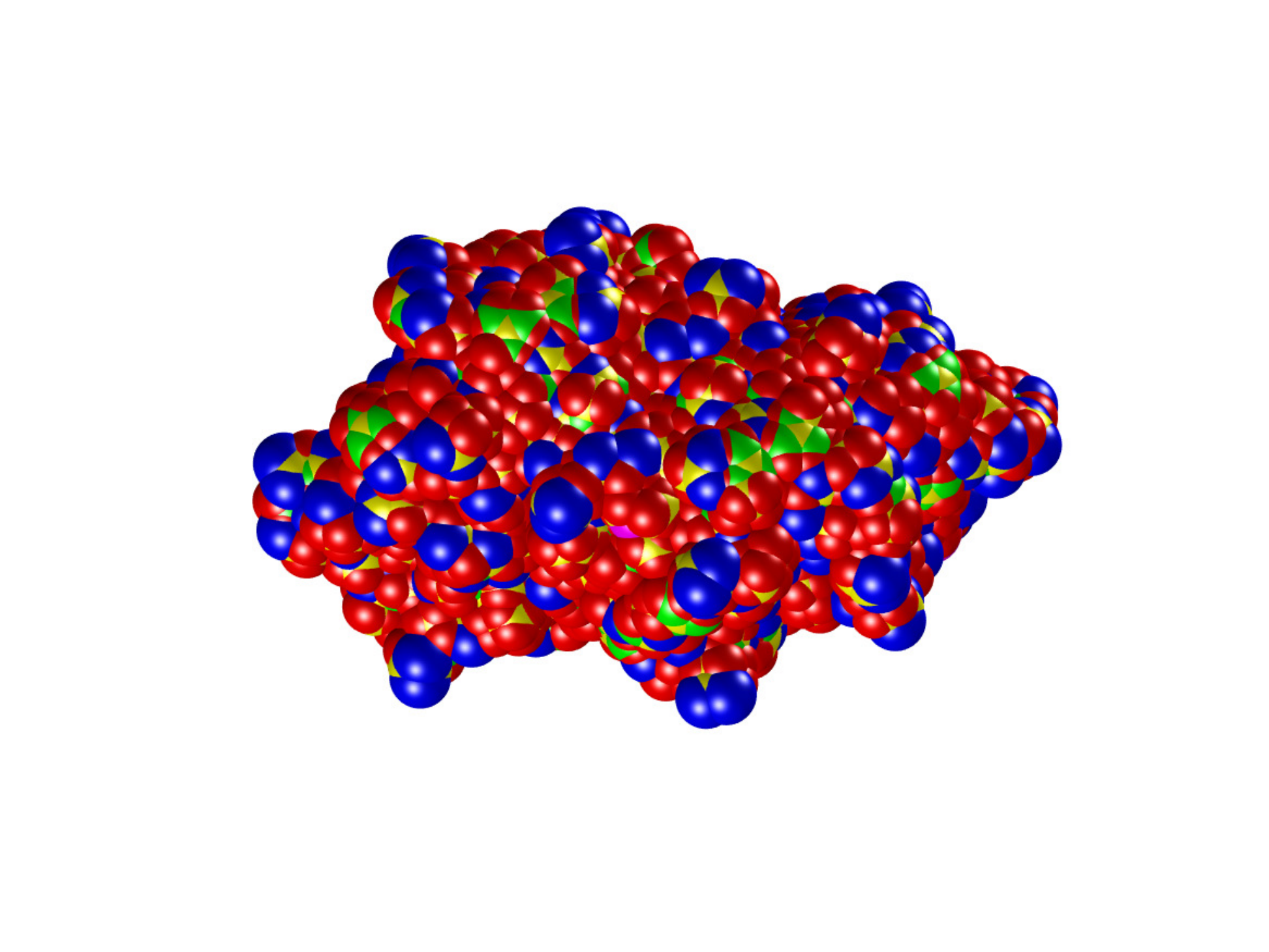} 
			\caption{Atomic Structure of Carboxylase with SAS cavities.}			
		\end{subfigure} %
		\begin{subfigure}{0.49\textwidth}
			\centering
			\includegraphics[width=\textwidth]{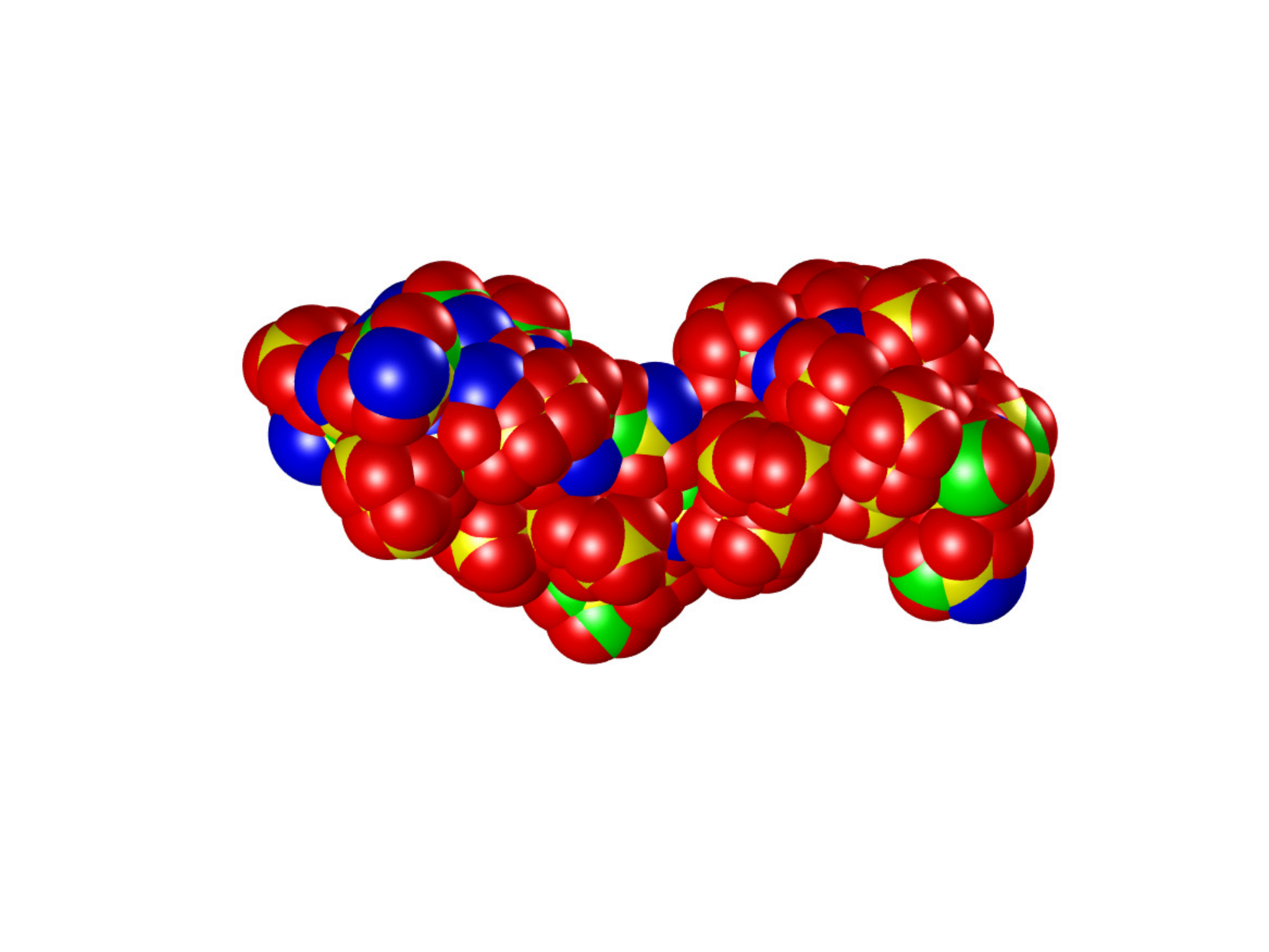} 
			\caption{Atomic Structure of Hiv-1-gp41 with SAS cavities.}
		\end{subfigure}
		\caption{Atomic Structures of `globular' Carboxylase and `chain-like' Hiv-1-gp41. Once again, the colour of each atomic cavity denotes the type of atom: yellow denotes Carbon, red denotes Hydrogen, green denotes Nitrogen, blue denotes Oxygen and magenta denotes Sulphur.}
		\label{fig:Protein300}
	\end{figure}

	We consider once again the six biological molecules analysed in \cite{BenPaper}. Table \ref{table:Protein300} contains information about various geometrical aspects of each of these protein molecules using the SAS cavity definition. We observe immediately that the number of layers in each molecule is now greater than two while the globular molecules UBCH5B and Carboxylase have six and eight layers respectively. Furthermore, the average over different subdomains of the maximum degree of intersection and the average over different subdomains of the overlap between neighbouring atoms are both higher than before but do not differ significantly between the molecules. Thus, the only significant differences between these molecules are now the number of layers and the average number of neighbours. Taken together, these observations suggest that the ddCOSMO algorithm should display slower convergence behaviour for all molecules and worse behaviour for the globular molecules.
	\begin{table}[h]
		\begin{center}
			\begin{adjustbox}{max width=\textwidth}
				\begin{tabular}{ | l || c | c | c | c | c |r |}
					\hline \hline
					Bio-Molecule & Vancomycin & UBCH5B & L-plectasin & Hiv-1-gp41& Glutaredoxin & Carboxylase\\ \hline
					Number of Atoms & 377 & 2364& 567 & 530& 1277 & 6609 \\ \hline
					Number of Layers & 3 & 6 & 4 & 3 & 5 & 8\\
					\hline
					Average Number of Neighbours & 50.5 & 80.7 & 59.5 & 56.3& 69.2 & 86.3\\
					\hline
					Average Maximum Intersection Degree & 17.3 & 19.0 & 17.6 & 17.9& 18.9 & 19.1\\
					\hline
					Average Overlap & 3.6 & 3.7 & 3.8 & 3.6& 3.7 & 3.8\\
					\hline \hline
				\end{tabular}
			\end{adjustbox}
		\end{center}
		\caption{Various geometrical properties of each biological molecule with the atomic cavities constructed using the SAS definition.}
		\label{table:Protein300}
	\end{table}
	
	As a test case, we consider the more `chain-like' molecules Vancomycin and Hiv-1-gp41 and the `globular' molecule UBCH5B. Notice that both the Vancomycin and Hiv-1-gp41 molecules consist of 3 layers when using the SAS definition of the cavity. On the other hand, the molecule UBCH5B, due to it's globular structure, consists of 6 layers. We solve the ddCOSMO system of equations using the PSM. Table \ref{table:Protein400} lists the number of iterations required to solve the linear system up to a fixed tolerance.
	
	\begin{table}[h]
		\begin{center}
			\begin{adjustbox}{max width=\textwidth}
				\begin{tabular}{ | l || c | c | c | c | c |r |}
					\hline \hline
					Bio-Molecule & Vancomycin  & Hiv-1-gp41 & UBCH5B \\ \hline
					Iterations for van der Waals Cavity & 115  & 152& 192\\ \hline
					Iterations for SAS Cavity & 189 & 164& 538 \\ \hline
					\hline
				\end{tabular}
			\end{adjustbox}
		\end{center}
		\caption{The number of iterations of PSM required to solve the ddCOSMO system of equations up to a fixed tolerance for each molecule.}
		\label{table:Protein400}
	\end{table}
	
	As expected, using the SAS definition of cavities results in a larger number of iterations being required to solve the underlying linear system. It is important to note however, that even though the Hiv-1-gp41 molecule contains about $40 \%$ more atoms than Vancomycin, the associated linear system for Hiv-1-gp41 can be solved in a fewer number of iterations than the linear system for Vancomycin. On the other hand, the linear system associated with the UBCH5B molecule, which contains 6 layers, now requires a significantly larger number of iterations to solve. These observations suggest the following conclusions:
	\begin{itemize}
		\item The key geometrical quantities which determine fast or slow convergence of the ddCOSMO algorithm for a given molecule are the number of layers (not the number of atoms by themselves) and the number of neighbouring atoms.
		\item For chain-like molecules, it is possible to use the SAS definition of the cavity without affecting the number of layers too significantly, and hence the convergence properties of ddCOSMO too adversely. On the other hand, for globular molecules, using the SAS definition results in a larger number of layers and consequently worse convergence properties.
	\end{itemize}
	
	We therefore observe that even though the SAS definition might represent a more physical representation of molecular cavities, the use of this definition for globular molecules results in a very serious deterioration in the computational efficiency of ddCOSMO.
	
	\section{Conclusions}\label{sec:conclusions}
	In this work, we presented a detailed convergence and scalability analysis of the PSM
	for the solution of the Laplace equation defined on a domain obtained as the union of
	several intersecting subdomains. The geometric framework we considered is capable of describing solvation problems of interests in practical applications.
	Our analysis, characterizing the properties of the infinite-dimensional Schwarz operator, allowed us to prove that a first contraction in the infinity norm can be achieved in at most $N_{\max}+1$ iterations. Here $N_{\max}$ is the number of layers and represents the maximal distance of the subdomains from the exterior boundary. Numerical experiments and an application to real biological molecules demonstrate the effectiveness of our analysis.

	% The next command determines the bibliography style. Please do not
	% change this.
	\bibliographystyle{plain}
	
	%  This includes the bib file
	\bibliography{biblio}
	
\end{document}